\documentclass{scrartcl}
\usepackage[utf8]{inputenc}
\usepackage[T1]{fontenc}

\usepackage{amsthm}
\usepackage{amsfonts}
\usepackage{stmaryrd}
\usepackage{MnSymbol}
\usepackage{xspace}
\usepackage{tikz-cd}
\usepackage{mathrsfs}
\usepackage{mathtools}

\usepackage[french=guillemets]{csquotes}
\MakeOuterQuote{"}
\usepackage[shortlabels]{enumitem}
\setlist{nosep}
\usepackage{xcolor}
\usepackage{cite}
\usepackage[color,leftbars]{changebar}
\usepackage{microtype}

\usepackage[colorlinks=true, breaklinks=true, urlcolor= black, linkcolor= black, citecolor= black,
bookmarksopen=true,linktocpage=true,plainpages=false,pdfpagelabels]{hyperref}

\usepackage[capitalise]{cleveref}
\RemoveFromHook{label}[firstaid/cleveref]

\newtheorem{theorem}{Theorem}[section]
\newtheorem*{theorem*}{Theorem}
\newtheorem{claim}{Claim}[theorem]
\newtheorem*{claim*}{Claim}
\newtheorem{lemma}[theorem]{Lemma}
\newtheorem{corollary}[theorem]{Corollary}
\newtheorem{fact}[theorem]{Fact}
\newtheorem{proposition}[theorem]{Proposition}

\theoremstyle{definition}
\newtheorem{remark}[theorem]{Remark}
\newtheorem{definition}[theorem]{Definition}
\newtheorem{notation}[theorem]{Notation}

\newtheorem{examples}[theorem]{Examples}

\newcommand\ACVF{\ensuremath{\mathrm{ACVF}}}

\newcommand\K{\mathrm{K}}

\newcommand{\eq}[1]{#1^{\mathrm{eq}}}
\newcommand{\acl}[1][]{\mathrm{acl}\ifstrempty{#1}{}{_{#1}}}
\newcommand{\dcl}[1][]{\mathrm{dcl}\ifstrempty{#1}{}{_{#1}}}

\newcommand\RV{\mathrm{RV}}
\newcommand\rv{\mathrm{rv}}

\newcommand\Qq{\mathbb{Q}}
\newcommand\Zz{\mathbb{Z}}
\newcommand\Nn{\mathbb{N}}

\newcommand\RCVF{\mathrm{RCVF}}

\newcommand\Trop{\mathrm{Trop}}

\newcommand\rest{\upharpoonright}

\newcommand{\tp}{\operatorname{tp}}
\renewcommand\epsilon{\varepsilon}

\newcommand\cU{\mathcal{U}}

\newcommand\m{\mathfrak{m}}

\newcommand\Lring{\mathcal{L}_\mathrm{ring}}
\newcommand\Lmac{\mathcal{L}_\mathrm{Mac}}

\newcommand\res{\mathrm{res}}

\renewcommand\mid{:}

\newcommand\cO{\mathcal{O}}
\newcommand\cF{\mathcal{F}}

\newcommand{\pCF}{\mathrm{pCF}}

\newcommand\Hom{\mathrm{Hom}}

\newcommand\trop{\mathrm{trop}}
\newcommand\an{\mathrm{an}}

\newcommand\Lau[2]{{#1}(\!(t^{#2})\!)}
\newcommand\lau[2]{{#1}(\!({#2})\!)}

\newcommand\Th{\mathrm{Th}}

\newcommand\RF{\mathbf{k}}

\newcommand\ZZ{\mathbb{Z}}
\newcommand\RR{\mathbb{R}}
\newcommand\FF{\mathbb{F}}
\newcommand\CC{\mathbb{C}}
\newcommand\QQ{\mathbb{Q}}

\newcommand\rF{\mathscr{F}}
\newcommand\rK{\mathscr{K}}

\newcommand\ccO{\mathscr{O}}

\newcommand\bm{\mathrm{bm}}

\newcommand\st{\operatorname{st}}

\title{Analytification and tropicalization via hyperfields and definable types}

\author{Pablo Cubides Kovacsics\thanks{The author was supported by the research grant FAPA (INV-2023-158-3157) \emph{New methods in non-archimedean geometry and its applications}, funded by Universidad de los Andes.}
\and Mariana Vicar\'ia\thanks{The author was supported by the Alexander von Humboldt Foundation (Postdoctoral Fellowship).} \and Jinhe Ye\thanks{The author was supported by his startup fund at the University of Hong Kong.}}

\date{}

\begin{document}

\maketitle

\begin{abstract}
We study the analytification and tropicalization of algebraic varieties over valued fields via Jun's spaces of hyperfield points. Our main new objects are a family of \emph{power-tropical hyperfields} $H_n$ and their inverse limit $H$, recording at each level the class of a leading term modulo $n$th powers of the residue field. They yield an \emph{enhanced tropicalization} of the associated space $X(H)$, strictly finer than the Berkovich one exactly when the multiplicative group of the residue field is non-divisible: it specialises to the Berkovich analytification when the residue field is algebraically closed, to the real analytification of Jell--Scheiderer--Yu when it is real closed, and to a new $p$-adic analytification when it is $p$-adically closed. Over henselian valued fields of equicharacteristic $0$ we then construct a model-theoretic avatar of $X(H)$: the space $\widehat{X}_T$ of definable types orthogonal to the value group, which sharpens the geometry in the spirit of Hrushovski and Loeser and recovers their space over algebraically closed valued fields. We show that $\widehat{X}_T$ is strictly pro-definable, satisfies a model-theoretic Payne theorem, and definable compactness in such spaces satisfies a Heine--Borel type characterization. As an application we analytify $p$-adic semi-algebraic sets and prove a strong density theorem for their tropicalizations, the exact $p$-adic counterpart of the real one of Jell--Scheiderer--Yu.
\end{abstract}

\section{Introduction}\label{sec:intro}

Analytification and tropicalization associate topological and combinatorial spaces to algebraic varieties over valued fields. Berkovich analytification records valuations of regular functions, while tropicalization describes the resulting geometry through finitely many such functions. Refinements of these constructions can retain additional arithmetic information. For example, the real analytification of Jell, Scheiderer, and Yu~\cite{jell} records signs as well as valuations. This suggests a common framework in which analytification reflects the arithmetic of the residue field.

We develop such a framework through two complementary descriptions: spaces of hyperfield points and spaces of definable types orthogonal to the value group. Our guiding principle is to retain residue field arithmetic information while preserving the topological and definable structures that allow a uniform treatment of the resulting spaces. This brings together Jun's geometry over hyperfields~\cite{Jun} and the model-theoretic approach to non-Archimedean geometry developed by Hrushovski and Loeser~\cite{HruLoe}. Our main results compare the two constructions, establish strict pro-definability, and characterize definable compactness by closedness and boundedness.

Note that our treatment is purely topological: we equip the analytifications with a topology, without introducing a structure sheaf or a class of analytic functions. Accordingly, our theorems concern the underlying topological spaces and their definable counterparts. The constructions extend to arbitrary varieties by working on affine open subsets and gluing along their overlaps. For simplicity of exposition, we formulate and prove our results for affine varieties, where a single coordinate ring suffices to describe the spaces, their topologies, and their tropicalizations. It will be interesting to study the geometry induced by the associated tropical functions into our new power-tropical hyperfields, following for example,
~\cite{tropical-function}.

Let $(K,v)$ be a valued field with value group $\Gamma$ and residue field $k$. We introduce hyperfields $H_n$ equipped with a topology, for $n\geq 2$, which record the valuation of an element together with the class of its leading term modulo $n$th powers. When $[k^\times:(k^\times)^n]<\infty$ $\text{for every }n\geq 2$, their inverse limit $H=\varprojlim_n H_n$ is again a hyperfield, with a natural homomorphism $\rho:K\to H$ (\Cref{prop:H-hyperfield}). Following \cite{Jun}, an affine variety $X$ over $K$ then determines the space $X(H)=\operatorname{Hom}_K(K[X],H)$, equipped with the topology of pointwise convergence.

Certain residue field information is reflected by this construction. If $k$ is algebraically closed, then $H=\mathbb T_\Gamma$ is the tropical hyperfield; when $\Gamma=\mathbb R$ and $K$ is complete, $X(H)$ is the Berkovich analytification. If $K$ is real closed with value group $\mathbb R$, we recover the real analytification of~\cite{jell} (\Cref{prop:H-reduction}). Other residue fields yield spaces that retain further arithmetic information through their inverse systems of power classes. It is worth mentioning here that our treatment is somewhat orthogonal to the work of Baldi and Telek \cite{baldi} on so called ``o-minimal tropicalizations'', where the authors deal with tropicalization of definable sets in a polynomially bounded o-minimal expansion of a real closed field, also extending \cite{jell}.  

In the parallel, the model-theoretic analytification takes place over \emph{bounded Macintyre fields}, or \emph{bm-fields}. These are henselian valued fields of equicharacteristic zero such that $\mathrm{RV}^{\times}/(\mathrm{RV}^{\times})^n$ is finite for every $n\geq 2$ and the residue field eliminates quantifiers in Macintyre's language, after naming suitable constants. Here $\mathrm{RV}^{\times}=K^\times/(1+\mathfrak m_K)$ is the "leading term sort". For a complete theory $T$ of a bm-field and a definable set $Y$, we write $\widehat Y_T$ for the space of global definable types concentrating on $Y$ and orthogonal to the value group. Its $K$-points are the types definable over $K$.

Evaluation of regular functions gives a canonical embedding 
\[ 
\Theta_K:\widehat X_T(K)\longrightarrow X(H(K)), 
\] 
which induces the topology on $\widehat X_T(K)$. When $T$ is the theory of algebraically closed valued fields, this construction recovers the Hrushovski--Loeser stable completion as a topological space (\Cref{prop:ACVF-HL}). For $K=\Lau{k}{\RR}$, $k\in\{\mathbb C,\mathbb R,\mathbb Q_p\}$, the map $\Theta_K$ is a homeomorphism (\Cref{thm:surjective1}, \Cref{cor:homeomorphism_theta}). Thus, in these cases, the model-theoretic construction describes the entire space of hyperfield points.

This comparison allows us to study the topology of hyperfield points through definable types. As in the work of Hrushovski and Loeser, the model-theoretic space provides a setting in which definable methods can be applied to questions about the topology of Berkovich spaces. However, there is indeed a distinction between the two approaches: for $K=\Lau{\lau{\CC}{s}}{\RR}$, we exhibit a hyperfield point of the affine line that is not induced by a type over $K$ (\Cref{prop:counterexample}).

To develop the definable topology of $\widehat Y_T$, we need an extra hypothesis on the residue field which we call Hypothesis $(B)$ (see later \Cref{subsec:strict-pro}). It holds for algebraically closed, real closed, and $p$-adically closed residue fields. Under this hypothesis, $\widehat Y_T$ is strictly pro-definable for every definable set $Y$ (\Cref{prop:strict-prodefinability}). This places the model-theoretic analytification in a setting tools from model theory can be applied systematically.

Our main topological result in this part is a definable Heine--Borel theorem. Following Hrushovski and Loeser, a pro-definable topological space is \emph{definably compact} if every global definable type concentrating on it has a limit in the space. We prove that a pro-definable subset of $\widehat X_T$ is definably compact if and only if it is closed and bounded, if the bm-field has an algebraically closed, real closed, or $p$-adically closed field as its residue field, and its value group is divisible or a $\mathbb Z$-group (\Cref{thm:heine-borel-vhat}). Here a $\mathbb Z$-group means a model of an ordered abelian group elementarily equivalent to $\ZZ$, and boundedness means having a common lower bound for the coordinatewise valuations. The theorem gives a uniform characterization of definable compactness across these different arithmetic settings. Together with the comparison results, it connects compactness in the definable setting with the compactness of the topology of hyperfield analytifications introduced earlier.

The hyperfield analytification also comes with a notion of tropicalization, which gives a finitary description of these spaces. For a finite tuple $F=(f_1,\ldots,f_\ell)$, evaluation defines a map $\operatorname{trop}_F:X(H)\longrightarrow H^\ell$, with finite-level projections to $H_n^\ell$. We recover $X(H)$ as the inverse limit of its tropicalizations (\Cref{prop:Payne-XH}), extending the perspective of Payne's theorem \cite{Payne}. We obtain the corresponding result for $\widehat X_T(K)$ when $K$ is maximally complete and every type over $k(K)$ is definable (\Cref{prop:Payne}).

Forgetting the power-class information gives a continuous map from $X(H)$ to $X(\mathbb T_\Gamma)$. When $\Gamma=\mathbb R$, we prove properness of this comparison map and of the tropicalization maps associated with generators of $K[X]$ (\Cref{prop:properness-XH} and \Cref{prop:shadow-proper}). We prove this under a technical condition on the residue field (see \Cref{def:star}) which ensures that hyperaddition is closed. Again, all fields of interest satisfy this condition.

As an application, we define an analytification of $p$-adic semi-algebraic sets. The elementary embedding $\mathbb Q_p\preccurlyeq\Lau{\QQ_p}{\RR}$ in the language of rings allows us to interpret a semi-algebraic set $S\subseteq\mathbb Q_p^m$ over $K=\Lau{\QQ_p}{\RR}$. We take its analytification to be $\widehat S_T(K)$, with the distinguished $t$-adic valuation. This places the $p$-adic and real constructions within the same framework.

More generally, for a $K$-definable set $Y$, the simple points coming from $Y(K)$ are dense in $\widehat Y_T(K)$ (\Cref{prop:density-semialg}). We also prove that every finite-level tropicalization of $\widehat Y_T(K)$ is already realized by points of $Y(K)$ (\Cref{prop:strong-density}). Thus analytification introduces additional points while preserving the tropical data realized over the valued base at each finite level.

The paper is structured as follows: \Cref{sec:prelim} collects the preliminaries on valued fields. \Cref{sec:hyperfields} and \Cref{sec:hyperfield-points} develop the hyperfields, their spaces of points, and their tropicalizations. \Cref{sec:analytification} introduces the model-theoretic analytification and establishes the comparison results. \Cref{subsec:strict-pro} treats strict pro-definability, tropicalization, and density. \Cref{sec:defcompact} proves the Heine--Borel characterization of definable compactness. The appendix gives a characterization of formally $p$-adic extensions through hyperfield morphisms, which is needed for the comparison between analytification via types and hyperfield points over $p$-adically closed residue fields.

\subsection*{Acknowledgements}
The authors would like to thank Thomas Scanlon for an insightful conversation that led to the final development of the present paper. 
\subsection*{AI Disclaimer} 
This project was for a large part developed before LLMs started giving significant contributions to mathematics. More precisely, Sections $2$,$3$,$4$,$5$,$6$ and a preliminary version of Section $7$ were developed by the authors without the use of $\mathrm{AI}$. The major interaction we had with LLMs helped us to finish Section $7$ and \Cref{prop:counterexample}. We give a detailed record of the AI contributions below:
\begin{itemize}
   
    \item  Regarding Section $7$: The main theorem of this section is the characterization of definably compact sets in the analyfication as closed and bounded (\emph{c.f.} \cref{thm:heine-borel-vhat}). We (human) authors had already proved without AI assistance that any closed and bounded set is definably compact. We also had a sketch of the proof of the converse, but to show that a set is closed we needed the key component that for any point $a$ in the closure of a pro-definable set $Y$ one could build a definable type $p$ whose limit is $a$. We (human) authors had a proof of this statement when the residue field is either the $p$-adics or a real closed field (\emph{c.f.} \cref{prop:curve-selection-DR}). We did not know how to generalize this argument to $\lau{\mathbb{C}}{t}$, in particular due to the gap in \cite[Lemma $4.2.12$]{HruLoe} indicated in \cref{rmk:limit-not-definable}. We asked ChatGPT Astra whether it could see a way to fix this gap, and it elaborated the statements and proofs of  \cref{lem:gamma-curve-selection}, \cref{lem:stable-gamma-lift}, and \cref{thm:stable-residue-curve}. The latter were carefully reviewed by the authors.
     \item \Cref{rmk:FLT} was suggested by Claude Opus $4.8$.
    \item \Cref{prop:counterexample} was also suggested by ChatGPT Astra. We wondered whether $\Theta$ could be surjective and suspected it was not the case, but we could not see the counterexample. The construction of the counterexample is completely produced by ChatGPT Astra, but the authors carefully proofread and verified it. 
    \item The proof of \Cref{lem:residue-realization} in the appendix was suggested by ChatGPT Astra.
    \item ChatGPT Astra substantiated the proof of \Cref{lem:translation-patterns} from some rough descriptions of the ideas involved.
\end{itemize}
Claude Opus $4.8$ was used for editing and proofreading of the paper. As a consequence, it spotted a gap in an argument of the authors in  \Cref{prop:bounded-limits} and suggested to add the hypothesis of largeness, which helps us to argue that the type $r$ is consistent (the definition of the type and the strategy of the proof is of the authors but we did not justify this before). Consequently, we added  \Cref{rmk:babyfactlargeness}  (the argument in the proof of this remark is due to the authors).

\section{Preliminaries}\label{sec:prelim}

Let $(K,v)$ be a valued field with value group $(\Gamma,<,+,0)$.
 We write $\Gamma_\infty$ for the group $\Gamma$ together with an element $\infty$ bigger than every element of $\Gamma$ and we extend the addition by the rules $\infty+\gamma = \gamma+\infty = \infty+\infty = \infty.$

 We write $\mathcal{O}_K$ to denote its valuation ring, i.e. $\{ x \in K \ | \ v(x) \geq 0\}$, $\mathcal{M}_K$ its maximal ideal (i.e. $\{ x \in K \ | v(x)>0\}$), $\RF_K$ its residue field $\mathcal{O}_K/\mathcal{M}_K$ and $\mathrm{res}\colon \mathcal{O}_K\to\RF_K$ the quotient map. We denote by \(\RV_K^\times\) the group
\(K^\times/(1+\mathcal{M}_K)\) and \(\rv : K\to\RV_K = \RV_K^\times \cup\{0\}\) the canonical
projection (extended by \(\rv(0) = 0\)). We call the set $\RV_K$ the \emph{leading terms}. It satisfies the short exact sequence
\[
1\longrightarrow \RF_K^\times\xrightarrow{\ \iota\ }\RV_K^\times \xrightarrow{\ \bar v } \Gamma \longrightarrow 0,
\]
where $\bar v(\rv(a))=v(a)$ and $\iota$ sends $\res(u)\in\RF_K^\times$ to $\rv(u)$, for $u\in\mathcal O_K^\times$ (where both maps are well defined independent of the chosen representative).

For each $n\geq 2$ reducing modulo $n$th powers preserves exactness:
\[
1\longrightarrow \RF_K^\times/(\RF_K^\times)^n\xrightarrow{\ \iota_n\ }\RV_K^\times/(\RV_K^\times)^n\xrightarrow{\ \bar v_n\ }\Gamma/n\Gamma\longrightarrow 0.\tag{$\dagger$}
\]
We write $G_n\coloneqq \RF_K^\times/(\RF_K^\times)^n\cup\{0\}$, with multiplication extended by $0g=g0=0$, and $G_n^\times=G_n\setminus\{0\}$. We let
\[
\rv_n\colon K\longrightarrow \RV_K^\times/(\RV_K^\times)^n\cup\{0\}
\]
be the composite of $\rv$ with the quotient modulo $(\RV_K^\times)^n$ (with $\rv_n(0)=0)$; it is multiplicative, surjective, and satisfies $\bar v_n(\rv_n(a))=v(a)+n\Gamma$.

We drop the subscript $K$ when no confusion arises, writing $\RF,\RV,\mathcal O,\m$ for the residue field, leading terms, valuation ring and maximal ideal.

\begin{lemma} \label{lem:divisibilities}
Let $(K,v)$ be a henselian valued field. 
    \begin{enumerate}
\item  Suppose $\operatorname{char}(\RF)\nmid n$. Then, $d \in (K^{\times})^n$ if and only if $\rv(d) \in \RV^{n}$. In particular, the map
$K^{\times}/(K^{\times})^{n} \;\longrightarrow\; \RV^{\times}/(\RV^{\times})^{n}$
is an isomorphism.

\item If $\Gamma$ is divisible, then every coset of $K^{\times}/(K^{\times})^{n}$ has a representative in $\mathcal{O}^{\times}$. If moreover $\operatorname{char}\RF\nmid n$, there is an isomorphism
$K^{\times}/(K^{\times})^{n} \;\cong\; \RF^{\times}/(\RF^{\times})^{n}.$
Consequently, for every $\gamma \in \Gamma$ and coset $c_{i}K^{n}$ if $a \in K^{\times}$ such that $v(a)=\gamma$ and $a \in c_{i}P_{n}$.

\item For an arbitrary valued field $(K,v)$, the group $\RV^{\times}/(\RV^{\times})^{n}$ is finite if and only if both $\RF^{\times}/(\RF^{\times})^{n}$ and $\Gamma/n\Gamma$ are finite.
\end{enumerate}
\end{lemma}

\begin{proof}
\begin{enumerate}
\item The implication from left to right is immediate. For the converse, assume that $\rv(d) \in \RV^{n}$. Then $\rv(d) = \rv(b)^n$ for some $b \in K$, so
\[
v\!\left(\frac{d}{b^n} - 1\right) > 0.
\]
Consider the polynomial
\[
p(x) = x^n - \frac{d}{b^n} \in \mathcal{O}[x].
\]
Then $v(p(1)) > 0$ and $v(p'(1)) = v(n) = 0$, the latter since $\operatorname{char}\RF\nmid n$. By henselianity, there exists $u \in \mathcal{O}^{\times}$ such that $p(u) = 0$, hence $u^n = \frac{d}{b^n}$ and therefore $d = (ub)^n$. Therefore, the homomorphism
\[
K^{\times}/(K^{\times})^{n} \;\longrightarrow\; \RV^{\times}/(\RV^{\times})^{n},
\quad a (K^{\times})^{n} \mapsto \rv(a)\, (\RV^{\times})^{n},
\]
is bijective by the previous argument.

\item Let $a \in K^{\times}$. Since $\Gamma$ is divisible, $v(a) \in n\Gamma$, so there exists $b \in K$ such that $v(a) = n v(b)$. Thus $a = u b^n$ for some $u \in \mathcal{O}^{\times}$, and hence $a (K^{\times})^{n} = u (K^{\times})^{n}$ with $u \in \mathcal{O}^{\times}$. Consider the natural homomorphism
\[
\phi_n : \mathcal{O}^{\times}/(1+\mathcal{M}) \;\longrightarrow\; \RV^{\times}/(\RV^{\times})^{n},
\]
given by $\rv(a) \mapsto \rv(a)\, (\RV^{\times})^{n}$. By the previous paragraph and part~(1), this map is surjective. By the first isomorphism theorem,
\[
\mathcal{O}^{\times}/(1+\mathcal{M})/\ker(\phi_n) \;\cong\; \RV^{\times}/(\RV^{\times})^{n}.
\]
Moreover, $\ker(\phi_n) = (\mathcal{O}^{\times})^{n}/(1+\mathcal{M})=(\RF^{\times})^{n}$, yielding the desired isomorphism.\\
For the last part, fix $\gamma \in \Gamma(M)$. We first find an element $a \in K^{\times}$ which is $n$-divisible and has valuation $\gamma$.

Let $a,b \in K^{\times}$ be such that $\rv(b)^n = \rv(a)$ and $v(a) = \gamma$. Then
\[
v\!\left(\frac{a}{b^n} - 1\right) > 0.
\]
Consider the polynomial
$p(x) = x^n - \frac{a}{b^n}.$
Then $v(p(1)) > 0$ while $v(p'(1)) = v(n) = 0$ (as $\operatorname{char}\RF\nmid n$). By henselianity, there exists $u \in K$ such that $p(u) = 0$. Hence
$u^n = \frac{a}{b^n}, \quad \text{so } a = (ub)^n.$

Now let $cK^n$ be any coset. Choose $d \in K^{\times}$ such that $n v(d) = v(c)$. Then
$cK^n = \frac{c}{d^n} K^n.$
Let $u = \frac{c}{d^n}$, so that $v(u) = 0$. Then
$v(ua) = \gamma \quad \text{and} \quad uaK^n = uK^n = cK^n,$

\item This is immediate from the exactness of $(\dagger)$.
\end{enumerate}
\end{proof}

\section{Hyperfields and tropicalization}\label{sec:hyperfields}

In this section we introduce the power-tropical hyperfields $H_n$ and their inverse limit $H$, central objects of the paper. For background on hyperrings and fields see \cite{Jun}. 

In what follows we use $\oplus$ for hyperaddition and standard multiplicative notation. Given hyperrings  $(R,\oplus,\cdot,0,1)$ and $(R',\oplus',\cdot',0',1')$, recall that a \emph{homomorphism} $f\colon R\to R'$ is a map satisfying $f(0)=0'$, $f(1)=1'$, $f(x\cdot y)=f(x)\cdot' f(y)$ for all $x,y\in R$, and
\[
f(x\oplus y)\subseteq f(x)\oplus' f(y)\qquad\text{for all }x,y\in R,
\]
where $f(x\oplus y):=\{f(z):z\in x\oplus y\}$. We say $f$ is \emph{strict} if equality
\[
f(x\oplus y)= f(x)\oplus' f(y)\qquad\text{for all }x,y\in R,
\]
holds. A homomorphism of hyperfields is a homomorphism of the underlying hyperrings; it is automatically multiplicative on the unit groups and, being a hyperring homomorphism, satisfies $f(\ominus x)=\ominus f(x)$.

A uniform source of hyperfields, due to Krasner, is the quotient of a field by a subgroup of its multiplicative group; all the examples below, and the hyperfields $H_n$ introduced in the next subsection, arise this way.

\begin{definition}\label{def:krasner-quotient}
Let $F$ be a field and $G\leq F^\times$ a subgroup. The \emph{Krasner quotient} $F/G$ is the set of orbits $\bar a:=aG$ (with $\bar 0:=\{0\}$), equipped with the multiplication $\bar a\cdot\bar b:=\overline{ab}$ and the hyperaddition
\[
\bar a\oplus\bar b:=\{\,\overline{a'+b'} \ : \ a'\in \bar a,\ b'\in \bar b\,\}.
\]
\end{definition}

\begin{fact}[Krasner \cite{Krasner1983}]\label{fact:krasner}
For a field $F$ and a subgroup $G\leq F^\times$, the Krasner quotient $(F/G,\oplus,\cdot,\bar 0,\bar 1)$ is a hyperfield.
\end{fact}

\begin{examples}\label{ex:hyperfields}\leavevmode
\begin{enumerate}
\item (\emph{Krasner hyperfield}) $\mathbb K=\{0,1\}$, with $1\cdot 1=1$ and hyperaddition $0\oplus x=\{x\}$, $1\oplus 1=\{0,1\}$. It is the Krasner quotient $F/F^\times$ of any field $F$ with $|F|\geq 3$, and records only whether an element is zero or not.
\item (\emph{Hyperfield of signs}) $\mathbb S=\{0,1,-1\}$, with the usual multiplication of signs and hyperaddition $x\oplus 0=\{x\}$, $1\oplus 1=\{1\}$, $(-1)\oplus(-1)=\{-1\}$, $1\oplus(-1)=\{0,1,-1\}$. It is the Krasner quotient $\RR/\RR_{>0}$, and records the sign of a real number.
\item (\emph{Tropical hyperfield}) For an ordered abelian group $\Gamma$, let $\mathbb T_\Gamma$ have underlying set $\Gamma_\infty=\Gamma\cup\{\infty\}$, multiplicative unit $0$, absorbing element $\infty$, multiplication $\gamma\cdot\epsilon=\gamma+\epsilon$, and hyperaddition
\[
\gamma\oplus\epsilon=\begin{cases}\{\min(\gamma,\epsilon)\} & \text{if }\gamma\neq\epsilon,\\[2pt] \{\delta\in\Gamma_\infty:\delta\geq\gamma\} & \text{if }\gamma=\epsilon.\end{cases}
\]
It is the Krasner quotient $K/\mathcal O_K^\times$ of any valued field $(K,v)$ with value group $\Gamma$, and records the valuation. When $\Gamma=\RR$ one writes $\mathbb T:=\mathbb T_\RR$, the tropical hyperfield of \cite{Jun}.
\item (\emph{Real tropical hyperfield}) $\mathbb{RT}=(\RR\times\{1,-1\})\cup\{\infty\}$, with multiplicative unit $(0,1)$, absorbing element $\infty$, multiplication $(\gamma,\epsilon)\cdot(\gamma',\epsilon')=(\gamma+\gamma',\epsilon\epsilon')$, and the hyperaddition (where $\bar v(\gamma,\epsilon)=\gamma$, $\bar v(\infty)=\infty$, and $\infty$ is the additive neutral)
\[
(\gamma,\epsilon)\oplus(\gamma',\epsilon')=\begin{cases}
\{(\gamma,\epsilon)\} & \text{if }\gamma<\gamma',\\
\{(\gamma',\epsilon')\} & \text{if }\gamma>\gamma',\\
\{(\gamma,\epsilon)\} & \text{if }\gamma=\gamma',\ \epsilon=\epsilon',\\
\{z:\bar v(z)\geq\gamma\} & \text{if }\gamma=\gamma',\ \epsilon=-\epsilon'.
\end{cases}
\]
It combines the rules of $\mathbb T$ and $\mathbb S$, recording both the valuation and the sign, and is the Krasner quotient $K/G$ of a real closed valued field $K$ with value group $\RR$ by its group $G$ of positive units (units of positive residue). It is the real tropical hyperfield $\mathbb{RT}$ of Jell--Scheiderer--Yu \cite{jell}.
\end{enumerate}
\end{examples}

\subsection{The power-tropical hyperfields}

Let $(K,v)$ be an arbitrary valued field, as in \Cref{sec:prelim}, with value group $\Gamma$ and residue field $\RF$. We use freely the leading-term structure $\RV$, the maps $\rv,\bar v,\iota$, the exact sequence $(\dagger)$ and the notation $G_n=\RF^\times/(\RF^\times)^n\cup\{0\}$, $G_n^\times=G_n\setminus\{0\}$ introduced there. 

We define the multiplicative part of the $n$th power-tropical hyperfield as the \emph{fiber product} of $\bar v_n$ with the quotient $q\colon\Gamma\twoheadrightarrow\Gamma/n\Gamma$,
\[
H_n^\times\coloneqq \Gamma\times_{\Gamma/n\Gamma}\RV^\times/(\RV^\times)^n=\big\{(\gamma,\xi)\in\Gamma\times\RV^\times/(\RV^\times)^n : q(\gamma)=\bar v_n(\xi)\big\},
\]
a subgroup of $\Gamma\times\RV^\times/(\RV^\times)^n$ under the coordinatewise operation $(\gamma,\xi)\cdot(\gamma',\xi')=(\gamma+\gamma',\xi\xi')$, which preserves the fiber-product condition since $q$ and $\bar v_n$ are homomorphisms. We adjoin an absorbing element $\mathbf 0$ and set $H_n\coloneqq H_n^\times\cup\{\mathbf 0\}$, with $\mathbf 0\cdot x=x\cdot\mathbf 0=\mathbf 0$ and unit $\mathbf 1\coloneqq(0,\bar 1)$. Consider the map
\[
\rho_n\colon K\to H_n,\qquad \rho_n(a)=\big(v(a),\,\rv_n(a)\big)\ (a\neq 0),\qquad \rho_n(0)=\mathbf 0,
\]
which is well defined since $q(v(a))=v(a)+n\Gamma=\bar v_n(\rv_n(a))$. Given $(\gamma,\xi)\in H_n$, any $a\in K$ with $\rho_n(a)=(\gamma,\xi)$ is called a \emph{$K$-representative} of $(\gamma,\xi)$; we let $H_n^\times=H_n\setminus\{\mathbf 0\}$. It will play the role of our distinguished hyperfield morphism later on.

\begin{lemma}\label{lem:surjectivity}
The map $\rho_n\colon K^\times\to H_n^\times$ is a surjective group homomorphism; together with $\rho_n(0)=\mathbf 0$, the map $\rho_n\colon K\to H_n$ is surjective.
\end{lemma}
\begin{proof}
Multiplicativity is immediate since both coordinates are homomorphisms on $K^\times$:
\[
\rho_n(ab)=\big(v(a)+v(b),\,\rv_n(a)\rv_n(b)\big)=\rho_n(a)\rho_n(b).
\]
For surjectivity, $\rho_n(0)=\mathbf 0$ so let $(\gamma,\xi)\in H_n^\times$ with $\gamma\in\Gamma$, $\xi\in\RV^\times/(\RV^\times)^n$ and $\bar v_n(\xi)=q(\gamma)$. Choose $c\in K^\times$ with $v(c)=\gamma$. Then $\bar v_n(\rv_n(c))=q(v(c))=q(\gamma)=\bar v_n(\xi)$, so $\xi\,\rv_n(c)^{-1}\in\ker\bar v_n=\operatorname{im}\iota_n$ by exactness of $(\dagger)$; write $\xi\,\rv_n(c)^{-1}=\iota_n(\bar w)$ with $\bar w\in\RF^\times/(\RF^\times)^n$. As $\iota_n$ is induced by $\iota(\res (u))=\rv(u)$, choosing $u\in\mathcal O^\times$ with $\res(u)$ in the class $\bar w$ gives $\rv_n(u)=\iota_n(\bar w)$ and $v(u)=0$. Then $a:=uc$ satisfies $v(a)=\gamma$ and $\rv_n(a)=\rv_n(u)\,\rv_n(c)=\iota_n(\bar w)\,\rv_n(c)=\xi$, i.e.\ $\rho_n(a)=(\gamma,\xi)$.
\end{proof}

\begin{corollary}\label{cor:hyperfield}
    The structure $(H_n, \oplus, \cdot, \mathbf{0}, \mathbf{1})$ is a hyperfield.
\end{corollary}

\begin{proof}
By \Cref{fact:krasner}, $\rho_n$ identifies $H_n$ with the Krasner quotient $K/N$ of the field $K$ by the subgroup $N=\ker(\rho_n|_{K^\times})\leq K^\times$, and the quotient map $\rho_n\colon K\to K/N$ is a strict homomorphism of hyperfields. 
\end{proof}

\begin{remark}\label{rmk:coordinates}
Note that $H_n^\times$ fits into a short exact sequence
\[
1\longrightarrow G_n^\times\xrightarrow{\ j\ } H_n^\times\xrightarrow{\ \mathrm{pr}_1\ }\Gamma\longrightarrow 0,\qquad j(g)=(0,\iota_n(g)),\quad \mathrm{pr}_1(\gamma,\xi)=\gamma,
\]
exhibiting it as an extension of the value group $\Gamma$ by $G_n^\times$ (indeed $\mathrm{pr}_1$ is onto because $\bar v_n$ is, $j$ is injective by $(\dagger)$, and $\ker\mathrm{pr}_1=\{0\}\times\iota_n(G_n^\times)=\operatorname{im}j$). Whenever $\RV^\times/(\RV^\times)^n$ is finite for every $n$, a section of $\bar v_n$ presents $H_n^\times$ in coordinates over $\RF$. For $s\in\RF^\times$ we write $[s]_n\in G_n^\times$ for its class modulo $(\RF^\times)^n$. Fix representatives $c_0=1,c_1,\dots,c_k\in K^\times$ with $\{v(c_i)\}$ a complete set of representatives of $\Gamma/n\Gamma$. For $a\in K^\times$ with $v(a)\in v(c_i)+n\Gamma$, write $v(a/c_i)=nv(b)$ and set
\[
[a]_n\coloneqq \big[\res(a/(c_ib^n))\big]_n\in G_n^\times,\qquad [0]_n\coloneqq 0,
\]
independent of $b$; then $\rho_n(a)=(v(a),[a]_n)$ under the resulting bijection $H_n^\times\cong\Gamma\times G_n^\times$. This bijection identifies $H_n^\times$ with the \emph{direct product} $\Gamma\times G_n^\times$ --- equivalently, the sequence above splits and $[\cdot]_n$ is multiplicative --- exactly when the $c_i$ can be chosen with $c_ic_j\equiv c_l\pmod{(K^\times)^n}$ whenever $v(c_l)\equiv v(c_i)+v(c_j)$. We henceforth use these coordinates, writing elements of $H_n$ as pairs $(\gamma,g)$ with $\rho_n(a)=(v(a),[a]_n)$ and $\mathbf 0=(\infty,g)$ for any $g$, when possible. 
\end{remark}

\begin{definition}[Topologies]\label{def: topologyinHn}
We endow $\Gamma$ with the order topology refined so that each coset of $n\Gamma$ is open. We extend this topology to $\Gamma_\infty$ by adding as open neighborhood of $\infty$ all the rays $(\gamma,\infty]$ as open sets. We equip $\RV^\times/(\RV^\times)^n$ with the discrete topology. We then topologize $H_n=H_n^\times\cup\{\mathbf 0\}$ by declaring:
\begin{itemize}
\item on $H_n^\times=\Gamma\times_{\Gamma/n\Gamma}\RV^\times/(\RV^\times)^n$, the subspace topology from the product topology on $\Gamma\times\RV^\times/(\RV^\times)^n$;
\item a basis of neighborhoods of $\mathbf 0$ given by the sets $\{\mathbf 0\}\cup\{x\in H_n^\times:\bar v(x)>\gamma_0\}$ for $\gamma_0\in\Gamma$,
\end{itemize}
so that $\bar v$ extends to a continuous map $\bar v\colon H_n\to\Gamma_\infty$ with $\bar v(\mathbf 0)=\infty$. We equip $H_n^\ell$ with the product topology.
\end{definition}

Suppose now that $\RF^\times/(\RF^\times)^n$ and $\Gamma/n\Gamma$ are finite (so that $G_n^\times$ is finite and $v(c_0),\dots,v(c_k)$ is a finite set of representatives of $\Gamma/n\Gamma$, in the notation of \Cref{rmk:coordinates}). Then, it is not difficult to see that the topology above is generated, in coordinates, by three kinds of sub-basic open sets, to which we refer in the sequel:
\begin{itemize}
    \item (Type A) for each $g\in G_n^\times$ and each order-open $U\subseteq\Gamma$, the set $\{(\gamma,g) : \gamma\in U\}$ (whenever $\Gamma$ is not discrete);
    \item (Type B) for each $g\in G_n^\times$ and each $i\le k$, the coset-fibre $\{(\gamma,g) : \gamma \in v(c_{i})+n\Gamma\}$;
    \item (Type C) the neighborhoods of $\mathbf{0}$ of the form $\bigcup_{g\in G_n^\times} \{(\delta, g) : \delta>\gamma_g \}$ for a family $\{\gamma_g\in \Gamma : g\in G_n^\times\}$.
\end{itemize}
Note that the topology on $\Gamma$ coincides with the order topology when $\Gamma$ is divisible. The type-C family, being indexed by the finite set $G_n^\times$, reduces to the single threshold $\gamma_0=\max_g\gamma_g$, recovering the canonical neighborhoods of $\mathbf 0$.

\begin{remark}
\label{rmk:topologyincoordinates}
Note that $H_n$ is Hausdorff. Fix the coordinates of \Cref{rmk:coordinates}, and equip $G_n^\times$ with
the discrete topology. After naming the representatives used to define
these coordinates, the bijection
\[
\kappa_n\colon \Gamma\times G_n^\times\longrightarrow H_n^\times,
\qquad
(\gamma,g)\longmapsto
\bigl(\gamma,s_n(q(\gamma))\iota_n(g)\bigr),\]
where
$s_n\colon\Gamma/n\Gamma
\mapsto \RV^\times/(\RV^\times)^n$. 
is the corresponding section of $\bar v_n$, is a definable homeomorphism. Indeed, each fiber of $q\colon\Gamma\to\Gamma/n\Gamma$ is a coset of
$n\Gamma$, and hence is clopen in the topology of
\Cref{def: topologyinHn}. On each such coset, the map $s_n\circ q$ is
constant, so the restriction of $\kappa_n$ is the product of the identity
on that coset with a bijection between discrete spaces. The same applies
to its inverse. For $g\in G_n^\times$, set
$H_{n,g}
:=
\{\mathbf 0\}\cup\kappa_n\bigl(\Gamma\times\{g\}\bigr).$ Then $H_{n,g}$ is closed in $H_n$, and the map
\[
\theta_g\colon\Gamma_\infty\longrightarrow H_{n,g},\ 
\theta_g(\gamma)=\kappa_n(\gamma,g)
\quad\text{and}\quad
\theta_g(\infty)=\mathbf 0,
\]
is a definable homeomorphism. Indeed, the complement of $H_{n,g}$ is the
open set
$\kappa_n\bigl(
\Gamma\times(G_n^\times\setminus\{g\})
\bigr)\subseteq H_n^\times,$
and, for every $\gamma_0\in\Gamma$,
\[
\theta_g\bigl((\gamma_0,\infty]\bigr)
=
H_{n,g}\cap
\Bigl(
\{\mathbf 0\}
\cup
\{x\in H_n^\times:\bar v(x)>\gamma_0\}
\Bigr).
\]
Moreover,
$\bar v\circ\theta_g
=
\operatorname{id}_{\Gamma_\infty}.$ More generally, for
$\mathbf g=(g_1,\ldots,g_\ell)\in(G_n^\times)^\ell$, set
$H_{n,\mathbf g}
:=
H_{n,g_1}\times\cdots\times H_{n,g_\ell}$
and
$\theta_{\mathbf g}
:=
\theta_{g_1}\times\cdots\times\theta_{g_\ell}
\colon
\Gamma_\infty^\ell\longrightarrow H_{n,\mathbf g}.$
Then $H_{n,\mathbf g}$ is closed in $H_n^\ell$,
$\theta_{\mathbf g}$ is a definable homeomorphism, and
$\bar v^\ell\circ\theta_{\mathbf g}
=
\operatorname{id}_{\Gamma_\infty^\ell}.$
If $G_n^\times$ is finite, the sets $H_{n,\mathbf g}$ form the finite
closed cover
$H_n^\ell
=
\bigcup_{\mathbf g\in(G_n^\times)^\ell}H_{n,\mathbf g}$.

Note that this extents to infinite products. Indeed, if $A$ is any set, the same argument show that $H_n^A$ is the union of closed sets $\bigcup_{\mathbf g\in(G_n^\times)^A} H_{n,\mathbf g}$. 
\end{remark}

For $m,n \in \Nn_{\geq 2}$ with $m \mid n$, the inclusion $(\RV^\times)^n\subseteq(\RV^\times)^m$ induces a surjective homomorphism $\RV^\times/(\RV^\times)^n\to\RV^\times/(\RV^\times)^m$, compatible with $\bar v_n,\bar v_m$ over $\Gamma/n\Gamma\to\Gamma/m\Gamma$. On the fibre products it induces the transition map
\[
\tau_{mn}\colon H_n\to H_m,\qquad (\gamma,\xi)\mapsto\big(\gamma,\ \xi\bmod(\RV^\times)^m\big),\quad \tau_{mn}(\mathbf 0)=\mathbf 0,
\]
which is continuous and satisfies $\tau_{mn}\circ\rho_n=\rho_m$. It is multiplicative, and a homomorphism of hyperfields: using $\tau_{mn}\rho_n=\rho_m$ and the surjectivity of $\rho_n$,
\begin{align*}
\tau_{mn}(\xi\oplus\eta) & =  \{\rho_m(a+b):\rho_n(a)=\xi,\ \rho_n(b)=\eta\} \\ 
& \subseteq  \{\rho_m(a'+b'):\rho_m(a')=\tau_{mn}\xi,\ \rho_m(b')=\tau_{mn}\eta\}\\
&=\tau_{mn}\xi\oplus\tau_{mn}\eta,
\end{align*}
the inclusion holding because $\rho_n(a)=\xi$ implies $\rho_m(a)=\tau_{mn}\xi$.

\begin{remark}\label{rmk:tau-not-strict}
The inclusion above can be proper, so $\tau_{mn}$ need not be a strict hyperfield morphism --- already in equicharacteristic $0$. Take $K=\RR$ with trivial valuation and $n=6$, $m=3$. Since $(\RR^\times)^6=\RR_{>0}$, a sum of two positives is never $0$, so in $H_6=\RR/(\RR^\times)^6$ one has $[1]_6\oplus[1]_6=\{[1]_6\}$ and $\tau_{36}([1]_6\oplus[1]_6)=\{[1]_3\}$. But $(\RR^\times)^3=\RR^\times$, so $H_3=\RR/(\RR^\times)^3$ has a single nonzero class, and $1+(-1)=0$ yields $\tau_{36}[1]_6\oplus\tau_{36}[1]_6=[1]_3\oplus[1]_3\ni\mathbf 0$. Thus $\{[1]_3\}\subsetneq[1]_3\oplus[1]_3$. The non-strictness must be taken seriously for the inverse limit below: an inverse limit of hyperfields along non-strict hyperfield homomorphisms need \emph{not} be a hyperfield, because the hypersum $x\oplus y=\varprojlim_n(x_n\oplus_n y_n)$ may be empty. Whether $H=\varprojlim_n H_n$ is a hyperfield turns out to depend only on the residue field: it holds when $[\RF^\times:(\RF^\times)^n]<\infty$ for all $n$ (\Cref{prop:H-hyperfield}) and can fail otherwise (\Cref{rmk:FLT}).
\end{remark}

\begin{definition}\label{def:tropical-hyperfield}
We define the limit $H$ as the inverse limit $H = \varprojlim_n H_n,$
taken with respect to the directed system given by the morphisms $\tau_{mn}$, endowed with the limit topology; we denote by $\rho\colon K \to H$ the map given by $\rho(x) = (\rho_n(x))_n$. 
\end{definition}

Multiplication, distributivity and additive inverses are computed coordinatewise and pass to the limit, so $H$ is a topological multiplicative monoid carrying a reversible hyperaddition $x\oplus y=\varprojlim_n(x_n\oplus_n y_n)$. The only point left open --- and the only place where the non-strictness of \Cref{rmk:tau-not-strict} plays a role --- is whether this hypersum is everywhere defined. When it is, $H$ is a hyperfield, which we then call the \emph{power-tropical hyperfield}; by \Cref{prop:H-hyperfield} this is the case whenever each $G_n^\times$ is finite, but not in general (\Cref{rmk:FLT}).

\begin{proposition}\label{prop:H-hyperfield}
Suppose $G_n^\times=\RF^\times/(\RF^\times)^n$ is finite for every $n$ (equivalently $[\RF^\times:(\RF^\times)^n]<\infty$). Then $H=\varprojlim_n H_n$ is a hyperfield, its topology is Hausdorff, and $\rho\colon K\to H$ is a homomorphism of hyperfields. 
\end{proposition}

\begin{proof}
Multiplicativity and distributivity, the absorbing role of $\mathbf 0$, and the existence and uniqueness of additive inverses pass to the inverse limit formally, since each $\tau_{mn}$ is multiplicative and each $H_n$ is a hyperfield (\Cref{cor:hyperfield}); the limit topology is a hyperfield topology because the operations are coordinatewise and the $\pi_n$ are continuous. Since each $H_n$ is Hausdorff, so is $H$. The one nonformal point is that $x\oplus y=\varprojlim_n(x_n\oplus_n y_n)$ is nonempty for all $x,y\in H$, an inverse limit of nonempty sets being possibly empty.

Fix $x,y\in H$. As $\tau_{mn}$ preserves $\bar v$, the values $\gamma=\bar v(x)$ and $\epsilon=\bar v(y)$ are constant along the threads. If $\gamma\ne\epsilon$, say $\gamma<\epsilon$, then $x_n\oplus_n y_n=\{x_n\}$ at every level (the term of smaller value dominates), so $x\in x\oplus y$. If $\gamma=\epsilon=\infty$ then $x=y=\mathbf 0$ and $x\oplus y=\{\mathbf 0\}$. If $\gamma=\epsilon$ is finite and $y=\ominus x$ (i.e. $y_n=\ominus x_n$ for all $n$), then $\mathbf 0\in x_n\oplus_n y_n$ for all $n$ and the constant thread $\mathbf 0$ lies in $x\oplus y$.

There remains $\gamma=\epsilon=:\mu$ finite and $y\ne\ominus x$. The set $\{n: y_n=\ominus x_n\}$ is closed under divisors, so its complement is cofinal in $(\Nn,\mid)$; for $n$ in the complement there is no additive cancellation, hence every element of $x_n\oplus_n y_n$ has $\bar v=\mu$. Therefore every thread of $x\oplus y$ has constant value $\mu$, and
\[
x\oplus y=\varprojlim_n L_n,\qquad L_n:=\{z\in x_n\oplus_n y_n:\bar v(z)=\mu\}.
\]
Each $L_n$ is nonempty and lies in the single fibre $\bar v^{-1}(\mu)\subseteq H_n$. With the valuation coordinate pinned to $\mu$, this fibre corresponds to the kernel of $\RV^\times/(\RV^\times)^n\to\Gamma/n\Gamma$, hence finite by hypothesis; so each $L_n$ is a finite nonempty set. The transition maps restrict to $L_n\to L_m$, making $(L_n)_n$ an inverse system of finite nonempty sets over the countable directed poset $(\Nn,\mid)$, whose inverse limit is nonempty. Thus $x\oplus y\ne\varnothing$. 

Lastly, we verify $\oplus$ is associative. Let $w\in (x\oplus y)\oplus z$, and we would like to show that $w\in x\oplus (y\oplus z)$. Unwrapping the definitions, it means to construct $a\in y\oplus z$ such that $w\in x\oplus a$. Since each $H_n$ is associative, that for each $n$, there is $a_n\in y_n\oplus z_n$ such that $w_n\in x_n\oplus a_n$. In other words, $W_n=(y_n\oplus z_n)\cap (w_n\oplus (\ominus(x_n))$ is nonempty for each $n$. By definition of $\oplus$, $W_n$ is finite whenever $z_n\neq \ominus y_n$ or $w_n\neq x_n$. If this happens for all $n$, we can take $a=\mathbf 0$. So we may assume that  $z_n\neq \ominus y_n$ or $w_n\neq x_n$ for some $n$, and hence the inequality holds for all of its multiples, and we may just zoom in to those multiples to compute the inverse limit, so we may assume that  $W_n$ is finite for each $n$. Note that the transition maps between $H_n$'s respect the $W_n$'s. This means that $\varprojlim_nW_n=\{(u_n)\in\prod_n W_n: \tau_{mn}(u_n)=u_m\text{ for }m\mid n \}$ is nonempty by compactness, and picking $a\in \varprojlim_nW_n$ works. The reverse containment follows similarly.
\end{proof}

\begin{remark}\label{rmk:FLT}
The valuation never obstructs: with the splitting $\RV^\times=\Gamma\times\RF^\times$ one has $L_n=g_n\cdot N_n(d_n)$ for $N_n(d)=\{[u^n+Dw^n]_n:u,w\in\RF^\times,\ u^n+Dw^n\neq0\}\subseteq\RF^\times/(\RF^\times)^n$, so $H$ is a hyperfield if and only if the inverse-limit Krasner hyperfield $\varprojlim_n\big(\RF/(\RF^\times)^n\big)$ of the residue field is --- the value group plays no role.

This can fail. Take $\RF=\Qq$ (e.g. $K=\Lau{\Qq}{\RR}$ or $\lau{\Qq}{t}$), for which $[\Qq^\times:(\Qq^\times)^n]=\infty$. Enumerate the primes $p_1,p_2,\dots$ and let $d\in\varprojlim_n\Qq^\times/(\Qq^\times)^n$ be the thread of positive sign with $p_i$-adic exponent $4\,i!$; this is well defined since $4\,i!\equiv0\pmod n$ once $i\ge n$. Put $x=\rho(1)$ and $y=(0,d_n)_n$, so $\bar v(x)=\bar v(y)=0$ and $y\neq\ominus x$ (already $d_4=[1]_4\neq[-1]_4$). Then $x\oplus y=\varprojlim_n N_n(d_n)$, and we claim it is empty. Let $e$ be a thread in this limit and $p=p_i$, $\alpha=4\,i!$. For any $n$ with $\alpha\not\equiv0\pmod n$ the two summands of a representative $u^n+Dw^n$ of $e_n$ have $p$-adic valuations $\equiv0$ and $\equiv\alpha\pmod n$, which differ; by the ultrametric there is no $p$-adic cancellation and $v_p(e_n)\equiv0$ or $\alpha\pmod n$. Varying $n$ forces the profinite exponent $\alpha_p(e)\in\{0,\alpha\}\subseteq\widehat\Zz$, in particular $\alpha_{p_i}(e)\equiv0\pmod4$ for every $i$. As $d_4=[1]_4$ and every value $u^4+w^4>0$, the class $e_4$ has positive sign and all exponents $\equiv0\pmod4$, whence $e_4=[1]_4$; this means $u^4+w^4=z^4$ for some $u,w,z\in\Qq^\times$, contradicting Fermat's theorem for exponent $4$. So $x\oplus y=\varnothing$ and $H$ is not a hyperfield. 
\end{remark}

\section{The space of hyperfield points}\label{sec:hyperfield-points}

In \cite{Jun}, the author presents a unifying way to describe several spectral constructions in algebraic geometry as \emph{spaces of hyperfield points}. For a topological hyperfield $\mathbb{H}$ equipped with a hyperfield homomorphism $\rho_\mathbb{H}\colon \K\to\mathbb{H}$ and a $K$-algebra $A$, one sets $\Hom_K(A,\mathbb{H})$, to be the set of hyperring homomorphisms $A\to\mathbb{H}$ restricting to the map $\rho_\mathbb{H}$ on $K$, with the coarsest topology making $\phi\mapsto\phi(f)$ continuous for every $f\in A$. Taking $X=\operatorname{Spec}A$ and writing $X(\mathbb{H})=\Hom_K(\ccO_X,\mathbb{H})$, the hyperfields of \Cref{ex:hyperfields} recover the classical constructions, as summarised in the following table.
\begin{center}
\begin{tabular}{lll}
\hline
$\mathbb{H}$ & $X(\mathbb{H})$ & reference \\
\hline
$\mathbb{K}$ & underlying scheme $X=\operatorname{Spec}\ccO_X$ & \cite[Prop.~4.1]{Jun} \\
$\mathbb{T}$ & Berkovich analytification $X^\an$ & \cite[Prop.~5.8]{Jun} \\
$\mathbb{S}$ & real spectrum $X_r$ & \cite[Lem.~6.11]{Jun} \\
$\mathbb{RT}$ & real analytification $X_r^\an$ & \cite[\S5]{jell} \\
\hline
\end{tabular}
\end{center}

Throughout, $X$ is an affine variety over $K$ with coordinate ring $\ccO_X$. We develop tropicalization at the level of a general coefficient hyperfield and only afterwards specialize. Thus, until stated otherwise, $\mathbb H$ denotes an arbitrary topological hyperfield equipped with a hyperfield homomorphism $\rho_{\mathbb H}\colon K\to\mathbb H$. 

\begin{lemma}\label{lem:hausdorff}
If $\mathbb H$ is Hausdorff, then $X(\mathbb H)$ is Hausdorff. 
\end{lemma}
\begin{proof}
The topology on $X(\mathbb H)$ is the initial topology for the evaluations $\mathrm{ev}_f\colon X(\mathbb H)\to\mathbb H$; these separate points, a hyperring homomorphism being determined by its values. An initial topology for a point-separating family of maps into Hausdorff spaces is Hausdorff, so $X(\mathbb H)$ is Hausdorff whenever $\mathbb H$ is. 
\end{proof}

\begin{definition}\label{def:trop-XH}
For a finite sequence $\cF=(f_1,\dots,f_\ell)\in\ccO_X^\ell$, we define the \emph{tropicalization map}
\[
\trop_\cF\colon X(\mathbb H)\to\mathbb H^\ell,\qquad \phi\mapsto(\phi(f_1),\dots,\phi(f_\ell)).
\]
For finite sequences $\cF=(f_1,\dots,f_\ell)$ and $\cF'=(f'_1,\dots,f'_r)$ and any map $\tau\colon\{1,\dots,\ell\}\to\{1,\dots,r\}$ satisfying $f'_{\tau(i)}=f_i$ for all $i$, the coordinate projection
\[
\pi_\tau\colon\mathbb H^{r}\to\mathbb H^{\ell},\qquad (z_1,\dots,z_r)\mapsto(z_{\tau(1)},\dots,z_{\tau(\ell)}),
\]
satisfies $\pi_\tau\circ\trop_{\cF'}=\trop_\cF$, and hence restricts to a surjection $\trop_{\cF'}(X(\mathbb H))\to\trop_\cF(X(\mathbb H))$. Following \cite[\S6.3]{jell}, we write $\varprojlim_{\cF}\trop_\cF(X(\mathbb H))$ for the limit of this system, whose points are the families $(z_\cF)_\cF$ with $z_\cF\in\trop_\cF(X(\mathbb H))$ compatible with all the projections $\pi_\tau$, topologized as a subspace of $\prod_{\cF}\mathbb H^{|\cF|}$; and we set
\[
\Trop\colon X(\mathbb H)\to\varprojlim_{\cF}\trop_\cF(X(\mathbb H)),\qquad \phi\mapsto(\trop_\cF(\phi))_\cF .
\]
\end{definition}

\begin{proposition}\label{prop:trop-XH-cont}
The maps $\trop_\cF$ and $\Trop$ are continuous.
\end{proposition}
\begin{proof}
By definition of the topology on $X(\mathbb H)$, each evaluation map $\mathrm{ev}_f$ is continuous, so $\trop_\cF=(\mathrm{ev}_{f_1},\dots,\mathrm{ev}_{f_\ell})$ is continuous into $\mathbb H^\ell$. The map $\Trop$ is then continuous by the universal property of the inverse limit, its components being the maps $\trop_\cF$.
\end{proof}

The following is the hyperfield-point form of Payne's theorem \cite[Theorem~1.1]{Payne}: it recovers $X(\mathbb H)$ as the inverse limit of its finite tropicalizations. The statement is completely abstract and its proof follows the ideas from \cite{jell} but is purely formal.

\begin{proposition}\label{prop:Payne-XH}
For an affine variety $X$ over $K$, the map
\[
\Trop\colon X(\mathbb H)\longrightarrow\varprojlim_{\cF}\trop_\cF(X(\mathbb H))
\]
is a homeomorphism, where $\cF$ ranges over all finite sequences of elements of $\ccO_X$.
\end{proposition}
\begin{proof}
Since a hyperring homomorphism is determined by its values, $\Trop$ is injective. Also, since the topology on $X(\mathbb H)$ is by definition the initial topology, $\Trop$ is a homeomorphism onto its image. (The same argument shows that the full evaluation map $E\colon X(\mathbb H)\to\prod_{f\in\ccO_X}\mathbb H$, $\phi\mapsto(\phi(f))_f$, is a homeomorphism onto its image, which will be later used). 

For surjectivity, let $(z_\cF)_\cF\in\varprojlim_\cF\trop_\cF(X(\mathbb H))$ and define $\phi\colon\ccO_X\to\mathbb H$ by $\phi(f):=z_{(f)}$. First note that
\[
z_\cF=(\phi(f_1),\dots,\phi(f_\ell))\qquad\text{for every }\cF=(f_1,\dots,f_\ell).
\]
Indeed, for each $j$ the map $\tau\colon\{1\}\to\{1,\dots,\ell\}$, $1\mapsto j$, satisfies $f_{\tau(1)}=f_j$, and compatibility with $\pi_\tau$ yields $(z_\cF)_j=z_{(f_j)}=\phi(f_j)$. (Note that this is not trivial as $\phi$ is---not yet---an element of $X(\mathbb H)$). 

To show that $\phi\in X(\mathbb H)$, we need to show that $\phi$ is a hyperring homomorphism extending $\rho_{\mathbb H}$; the point is that each defining condition involves only finitely many elements of $\ccO_X$, and every $z_\cF$ is realized by an element of $X(\mathbb H)$.
\begin{itemize}
\item (Constants) For $c\in K$, taking $\cF=(c)$ yields $\psi\in X(\mathbb H)$ with $\psi(c)=z_{(c)}$; as $\psi|_K=\rho_{\mathbb H}$, we get $\phi(c)=\rho_{\mathbb H}(c)$. In particular $\phi|_K=\rho_{\mathbb H}$, and $\phi(1)=\mathbf{1}$, $\phi(0)=\mathbf{0}$.
\item (Multiplicativity) For $f,g\in\ccO_X$, taking $\cF=(f,g,fg)$ yields $\psi\in X(\mathbb H)$ with $\trop_\cF(\psi)=z_\cF=(\phi(f),\phi(g),\phi(fg))$ by the claim; since $\psi$ is multiplicative, $\phi(fg)=\psi(fg)=\psi(f)\psi(g)=\phi(f)\phi(g)$.
\item (Hyperaddition) For $f,g\in\ccO_X$, taking $\cF=(f,g,f+g)$ yields such a $\psi$, whence $\phi(f+g)=\psi(f+g)\in\psi(f)\oplus\psi(g)=\phi(f)\oplus\phi(g)$.
\end{itemize}
Thus $\phi\in X(\mathbb H)$, and clearly $\Trop(\phi)=(z_\cF)_\cF$ by the claim.
\end{proof}

We now instantiate the hyperfield $\mathbb H$ as $H_n$ or $H$. Throughout we assume that $G_n^\times$ is finite for every $n$ (equivalently $[\RV^\times:(\RV^\times)^n]<\infty$), so that $H$ is a topological hyperfield by \Cref{prop:H-hyperfield}; this holds whenever $\RF$ is finite, local, real closed or algebraically closed, and by \Cref{rmk:FLT} it cannot be dropped.

\begin{remark}\label{rem:Payne-XH-levels}
Since $H=\varprojlim_n H_n$, each $\trop_\cF$ is the inverse limit over $n$ of the maps $\trop_\cF^n$; hence $\Trop$ is equally the inverse limit of the $\trop_\cF^n$ over all pairs $(\cF,n)$, and \Cref{prop:Payne-XH} identifies $X(H)$ with $\varprojlim_{\cF,n}\trop_\cF^n(X(H))$.
\end{remark}

In two cases the coefficient hyperfield $H$ degenerates to a hyperfield already appearing in the literature, so that $X(H)$ \emph{is} a classical analytification, by the identifications of Jun \cite{Jun} and Jell--Scheiderer--Yu \cite{jell}.

\begin{proposition}\label{prop:H-reduction}
\leavevmode
\begin{enumerate}
\item If $\RF$ is algebraically closed, then $H=\mathbb{T}_\Gamma$ is the tropical hyperfield over $\Gamma$. When $\Gamma=\RR$, this is the classical tropical hyperfield $\mathbb{T}$ and, if $K$ is complete, then $X(H)=X^\an$ (see \cite[Propositions~5.9]{Jun}).
\item If $K$ is a real closed valued field with value group $\RR$, then $H=\mathbb{RT}$ is the real tropical hyperfield, and $X(H)=X(\mathbb{RT})=X_r^\an$ is the real analytification \cite[\S5]{jell}.
\end{enumerate}
\end{proposition}

\begin{proof}
(i) Over an algebraically closed residue field $\RF^\times=(\RF^\times)^n$, so $G_n^\times=\{1\}$ and $\RV^\times/(\RV^\times)^n=\Gamma/n\Gamma$; hence $H_n^\times=\Gamma\times_{\Gamma/n\Gamma}\Gamma/n\Gamma=\Gamma$ with $\tau_{mn}=\mathrm{id}$, so $H_n=\Gamma_\infty$ and $H=\varprojlim_n\Gamma_\infty=\Gamma_\infty=\mathbb{T}_\Gamma$ carries the tropical hyperaddition (cf.\ \Cref{rmk:tau-not-strict}). For $\Gamma=\RR$ this is $\mathbb{T}$, and $X(\mathbb{T})=X^\an$ by \cite[Proposition~5.8]{Jun}.

(ii) A real closed field $\RF$ has $\RF^\times/(\RF^\times)^n=\{\pm1\}$ for $n$ even and trivial for $n$ odd, the nontrivial class recording the sign; since $\RR$ is divisible, $H_n^\times=\RR\times G_n^\times$ and the tower stabilizes with $\varprojlim_n G_n^\times=\{\pm1\}$. Thus $H=(\RR\times\{\pm1\})\cup\{\mathbf 0\}$, with $(\gamma,\epsilon)(\gamma',\epsilon')=(\gamma+\gamma',\epsilon\epsilon')$ and the hyperaddition taking the dominant term when the values differ, the same sign when $(\gamma,\epsilon)\oplus(\gamma,\epsilon)$, and the full ball $\{z:\bar v(z)\ge\gamma\}$ when $(\gamma,\epsilon)$ and $(\gamma,-\epsilon)$ cancel; this is exactly the real tropical hyperfield $\mathbb{RT}$. Hence $X(H)=X(\mathbb{RT})=X_r^\an$ by \cite[\S5]{jell}.
\end{proof}

\subsection{Properness of the tropicalization for \texorpdfstring{$\Gamma=\RR$}{Gamma=R}}\label{subsec:properness}

Throughout this subsection we assume, in addition to the standing hypothesis of \Cref{sec:hyperfield-points} that $G_n^\times$ is finite for every $n$, that $\Gamma=\RR$. Under these hypotheses we prove that, when $\cF$ generates the coordinate ring, the tropicalization $\trop_\cF^n\colon X(H)\to H_n^\ell$ is a proper map; this generalizes \cite[Lemma~3.7]{jell}.

Since $\Gamma=\RR$ is divisible, $\Gamma/n\Gamma=0$ and the splitting of \Cref{rmk:coordinates} is canonical: $H_n^\times=\RR\times G_n^\times$ and $H_n=(\RR\times G_n^\times)\cup\{\mathbf 0\}$, with $\bar v(\gamma,g)=\gamma$ and $\bar v(\mathbf 0)=\infty$. For $s\in\RF^\times$ we write $[s]_n\in G_n^\times=\RF^\times/(\RF^\times)^n$ for its power class; on units of $K$ this agrees with the notation $[\,\cdot\,]_n$ of \Cref{rmk:coordinates}.

\begin{definition}\label{def:bounded-Hn}
A subset $Y\subseteq H_n$ (resp.\ $Y\subseteq H$) is \emph{bounded} if $\bar v(Y)$ is bounded below in $\RR$, i.e.\ $Y\subseteq\bar v^{-1}([\gamma_0,\infty])$ for some $\gamma_0\in\RR$.
\end{definition}

\begin{lemma}\label{lem:heineborel}
The spaces $H_n$ and $H$ are locally compact, and a subset is compact if and only if it is closed and bounded. The same holds for the finite powers $H_n^\ell$ and $H^\ell$.
\end{lemma}
\begin{proof}
In $H_n$ the open set $H_n^\times=\RR\times G_n^\times$ is a finite disjoint union of lines, hence locally compact. The sets $\bar v^{-1}([\gamma_0,\infty])=([\gamma_0,\infty)\times G_n^\times)\cup\{\mathbf 0\}$ are compact: given an open cover, an open set $U\ni\mathbf 0$ contains some $\bar v^{-1}((\gamma_1,\infty])$, and the remainder $\bar v^{-1}([\gamma_0,\gamma_1])=[\gamma_0,\gamma_1]\times G_n^\times$ is compact; as these sets form a neighbourhood basis of $\mathbf 0$, the space $H_n$ is locally compact. If $C$ is closed and bounded then $C\subseteq\bar v^{-1}([\gamma_0,\infty])$ is a closed subset of a compact set, hence compact; conversely the open cover $\{\bar v^{-1}((-m,\infty])\}_{m\in\Nn}$ of any set $C$ admits a finite subcover only when $\bar v(C)$ is bounded below. For $H=\varprojlim_n H_n$ the set $\bar v^{-1}([\gamma_0,\infty])$ is the inverse limit of the corresponding compact subsets of the $H_n$, hence compact, and the rest follows as before. For $H_n^\ell$ and $H^\ell$ the result follow by Tychonoff's theorem.
\end{proof}

The closedness of the hyperaddition is governed by an arithmetic condition on the residue field.

\begin{definition}\label{def:star}
The residue field $\RF$ satisfies $(\star_n)$ if the map
\[
\{w\in\RF^\times: w^n\neq 1\}\longrightarrow G_n^\times,\qquad w\longmapsto [\,w^n-1\,]_n
\]
is surjective; it satisfies $(\star)$ if $(\star_n)$ holds for every $n\geq 2$.
\end{definition}

For $x=(\gamma,g)$ and $y=(\gamma,h)$ with $\gamma\neq\infty$, we write
\[
C(g,h):=\big\{[s+t]_n : s,t\in\RF^\times,\ [s]_n=g,\ [t]_n=h,\ s+t\neq 0\big\}\subseteq G_n^\times,
\]
for the set of leading classes realized by a non-canceling sum; by construction the value-$\gamma$ part of $x\oplus y$ is $\{\gamma\}\times C(g,h)$.

\begin{lemma}\label{lem:graphclosed}
If $\RF$ satisfies $(\star)$, then the graph $\Delta:=\{(x,y,w)\in H^3: w\in x\oplus y\}$ is closed in $H^3$.
\end{lemma}
\begin{proof}
Since $w\in x\oplus y$ if and only if $\pi_n(w)\in\pi_n(x)\oplus_n\pi_n(y)$ for all $n$, we have $\Delta=\bigcap_n(\pi_n^{\times3})^{-1}(\Delta_n)$ with $\Delta_n=\{(x,y,w)\in H_n^3: w\in x\oplus_n y\}$ and $\pi_n^{\times3}\colon H^3\to H_n^3$ continuous; so it suffices to show each $\Delta_n$ is closed. As $H_n$ is first countable (recall $\Gamma=\RR$) we can argue with sequences.

We first record two consequences of $(\star_n)$. Fix $g,h\in G_n^\times$ and note that $C(g,h)=g\cdot[\,1+R\,]$, where $R=\{r\in\RF^\times:[r]_n=hg^{-1}\}$ and $[\,1+R\,]:=\{[1+r]_n:r\in R,\ r\neq -1\}$.
\begin{enumerate}[(i)]
\item If $h\neq[-1]_n g$ then $g,h\in C(g,h)$. Indeed $g\in C(g,h)$ means $1\in[\,1+R\,]$, i.e.\ there is $r$ with $[r]_n=hg^{-1}$ and $1+r\in(\RF^\times)^n$; writing $1+r=w^n$ this is $[w^n-1]_n=hg^{-1}$, solvable by $(\star_n)$. By the symmetry $C(g,h)=C(h,g)$, also $h\in C(g,h)$.
\item If $h=[-1]_n g$ then, with $\gamma=\bar v(x)$, one has $x\oplus(\ominus x)=\bar v^{-1}([\gamma,\infty])$, the full closed ball. Indeed $C(g,[-1]_ng)=g[-1]_n\cdot\{[\,w^n-1\,]_n:w^n\neq 1\}=g[-1]_n\,G_n^\times=G_n^\times$ by $(\star_n)$, so the value-$\gamma$ part of $x\oplus(\ominus x)$ is all of $\{\gamma\}\times G_n^\times$; together with the elements of value $>\gamma$ produced by cancellation this is $\bar v^{-1}([\gamma,\infty])$.
\end{enumerate}

Now let $(x_k,y_k,w_k)_k$ be a sequence in $\Delta_n$ converging to $(x,y,w)$ in $H_n^3$ with $w_k\in x_k\oplus_n y_k$, and write $\gamma=\bar v(x)\leq \delta=\bar v(y)$ (this is no loss, by commutativity). As $G_n^\times$ is discrete and finite and $\bar v$ is continuous, after discarding finitely many terms we may assume that $[x_k]_n$ and $[y_k]_n$ are constant wherever the limits $x,y$ lie in $H_n^\times$.
\begin{itemize}
\item \emph{$\gamma<\delta$} Eventually $\bar v(x_k)<\bar v(y_k)$, so $x_k\oplus y_k=\{x_k\}$ and $w_k=x_k\to x$; thus $w=x\in x\oplus y=\{x\}$.
\item \emph{$\gamma=\delta=\infty$.} Then $x=y=\mathbf 0$ and $\bar v(w_k)\geq\min(\bar v(x_k),\bar v(y_k))\to\infty$, so $w=\mathbf 0\in\mathbf 0\oplus\mathbf 0$.
\item \emph{$\gamma=\delta$ finite and $y=\ominus x$.} By continuity $\bar v(w)=\lim_k\bar v(w_k)\geq\liminf_k\min(\bar v(x_k),\bar v(y_k))=a$, so $w\in\bar v^{-1}([\gamma,\infty])=x\oplus y$ by (ii).
\item \emph{$\gamma=\delta$ finite and $y\neq\ominus x$.} Put $g=[x]_n,h=[y]_n$, so $h\neq[-1]_n g$, and eventually $[x_k]_n=g$, $[y_k]_n=h$. Whenever $\bar v(x_k)<\bar v(y_k)$ one has $w_k=x_k\to x=(\gamma,g)$; whenever $\bar v(x_k)>\bar v(y_k)$, $w_k=y_k\to y=(\gamma,h)$; whenever $\bar v(x_k)=\bar v(y_k)$, $w_k\in\{\bar v(x_k)\}\times C(g,h)$ converges to a point $(\gamma,c_0)$ with $c_0\in C(g,h)$. Whichever situation occurs infinitely often determines $w$, and in each case the limit lies in $\{\gamma\}\times C(g,h)=x\oplus y$, using (i) for the first two.
\end{itemize}
Hence each $\Delta_n$, and therefore $\Delta$, is closed.
\end{proof}

\begin{lemma}\label{lem:starcases}
Condition $(\star)$ holds when $\RF$ is algebraically closed, when $\RF$ is real closed, and when $\RF$ is a finite extension of $\Qq_p$.
\end{lemma}
\begin{proof}
If $\RF$ is algebraically closed then $G_n^\times=\{1\}$ and $(\star_n)$ is vacuous.

If $\RF$ is real closed then $G_n^\times=\{1\}$ for $n$ odd, while for $n$ even $(\RF^\times)^n=\RF_{>0}$ and $G_n^\times=\{\pm1\}$. As $w$ ranges over $\RF^\times$ (with $n$ even) the value $w^n$ ranges over all of $\RF_{>0}$, so $w^n-1$ ranges over $\RF_{>-1}\setminus\{0\}$, meeting both $\RF_{>0}$ (e.g.\ $w^n=2$) and $\RF_{<0}$ (e.g.\ $w^n=\tfrac12$); thus $w\mapsto[w^n-1]_n$ is surjective.

Let $\RF$ be a finite extension of $\Qq_p$, with valuation ring $\mathcal O_\RF$, maximal ideal $\mathfrak m_\RF$ and uniformizer $\varpi$. Since $w\mapsto w^n$ maps $\RF^\times$ onto $(\RF^\times)^n$, we have $\{w^n-1:w^n\neq 1\}=\big((\RF^\times)^n-1\big)\setminus\{0\}$, so it suffices that $(\RF^\times)^n-1$ meet every coset of $(\RF^\times)^n$. By Hensel's lemma $x\mapsto x^n$ has open image, so $1+\mathfrak m_\RF^N\subseteq(\RF^\times)^n$ for some $N$, which we may take divisible by $n$; hence $\mathfrak m_\RF^N\setminus\{0\}\subseteq(\RF^\times)^n-1$. Every coset of $(\RF^\times)^n$ has a representative $\varpi^j u$ with $0\leq j<n$ and $u\in\mathcal O_\RF^\times$ (as $\RF^\times=\varpi^{\Zz}\times\mathcal O_\RF^\times$ and $\varpi^n\in(\RF^\times)^n$). The element $z:=\varpi^{N+j}u\in\mathfrak m_\RF^N\setminus\{0\}$ then satisfies $z\in(\RF^\times)^n-1$, and $z=\varpi^N\cdot\varpi^j u$ with $\varpi^N\in(\RF^\times)^n$ (since $n\mid N$), so $z$ lies in the coset of $\varpi^j u$. 
\end{proof}

\begin{lemma}\label{lem:Eclosed}
Assume $(\star)$. Then the evaluation embedding $E\colon X(H)\to\prod_{f\in\ccO_X}H$, $\phi\mapsto(\phi(f))_f$, has closed image.
\end{lemma}
\begin{proof}
By the same argument as in the proof of \Cref{prop:Payne-XH}, $E$ is a homeomorphism onto
\[
E(X(H))=\Big\{(z_f)_f : z_c=\rho(c)\ (c\in K),\ z_{fg}=z_fz_g,\ z_{f+g}\in z_f\oplus z_g\ (f,g\in\ccO_X)\Big\}.
\]
Each family of conditions defines a closed subset of $\prod_f H$: the constant conditions $z_c=\rho(c)$ because $H$ is Hausdorff; the multiplicative conditions $z_{fg}=z_fz_g$ as equalizers of the continuous maps $(z_f)_f\mapsto z_{fg}$ and $(z_f)_f\mapsto z_fz_g$ into the Hausdorff space $H$; and the additive conditions $z_{f+g}\in z_f\oplus z_g$ as preimages of the closed set $\Delta$ of \Cref{lem:graphclosed} under the continuous maps $(z_f)_f\mapsto(z_f,z_g,z_{f+g})$. The intersection is therefore closed.
\end{proof}

\begin{proposition}\label{prop:properness-XH}
Assume $(\star)$ (\Cref{def:star}), and let $\cF=(f_1,\dots,f_\ell)$ generate $\ccO_X$. Then for every $n\geq 2$ the maps
$\trop_\cF^n\colon X(H)\to H_n^\ell$ and $\trop_\cF\colon X(H)\to H^\ell$
are proper.
\end{proposition}
\begin{proof}
We argue for $\trop_\cF^n$; the proof for $\trop_\cF$ is identical, with $H_n^\ell$ replaced by $H^\ell$ throughout. Let $B\subseteq H_n^\ell$ be compact; by \Cref{lem:heineborel}, $B$ is closed and bounded, say $\bar v(\mathrm{pr}_j b)\geq\gamma_0$ for all $b\in B$ and $1\leq j\leq\ell$. Then $B':=(\trop_\cF^n)^{-1}(B)$ is closed in $X(H)$ (\Cref{prop:trop-XH-cont}), and for $\phi\in B'$ and each $j$ we have $\bar v(\phi(f_j))=\bar v(\pi_n(\phi(f_j)))\geq\gamma_0$.

Let $f\in\ccO_X$ be arbitrary and write $f=\sum_\alpha c_\alpha f_1^{\alpha_1}\cdots f_\ell^{\alpha_\ell}$ with $c_\alpha\in K$, possible since $\cF$ generates $\ccO_X$. As $\phi$ is a hyperring homomorphism, $\phi(f)\in\bigoplus_\alpha\rho(c_\alpha)\prod_j\phi(f_j)^{\alpha_j}$, and since $\bar v$ is multiplicative and $\bar v(\bigoplus_i z_i)\subseteq[\min_i\bar v(z_i),\infty]$,
\[
\bar v(\phi(f))\ \geq\ \min_\alpha\Big(v(c_\alpha)+\sum_{j}\alpha_j\,\bar v(\phi(f_j))\Big)\ \geq\ \min_\alpha\Big(v(c_\alpha)+\gamma_0\textstyle\sum_j\alpha_j\Big)=:\gamma_f,
\]
a finite real number depending only on $f$ and $\gamma_0$ (the minimum being over the finitely many monomials of $f$). Hence $E(B')\subseteq\prod_f D_f$ with $D_f:=\bar v^{-1}([\gamma_f,\infty])\subseteq H$ compact (by \Cref{lem:heineborel}). By Tychonoff, $\prod_f D_f$ is compact. Now, $E(B')$ is closed in $E(X(H))$, since $E$ is a homeomorphism onto its image and $B'$ is closed, and $E(X(H))$ is closed in $\prod_f H$ (\Cref{lem:Eclosed}); so $E(B')$ is a closed subset of the compact set $\prod_f D_f$, hence compact. Again, as $E$ is a homeomorphism onto its image, $B'$ is compact.
\end{proof}

\begin{corollary}\label{cor:image-closed}
Under the hypotheses of \Cref{prop:properness-XH}, the images $\trop_\cF^n(X(H))\subseteq H_n^\ell$ and $\trop_\cF(X(H))\subseteq H^\ell$ are closed.
\end{corollary}
\begin{proof}
This follows form (\Cref{lem:heineborel}) and \Cref{prop:properness-XH} since a proper map into a locally compact Hausdorff space is closed.
\end{proof}

\begin{remark}\label{rmk:star-sharp}
Condition $(\star)$ is the precise analogue, for properness, of the finiteness hypothesis $[\RF^\times:(\RF^\times)^n]<\infty$ that governs whether $H$ is a hyperfield (\Cref{prop:H-hyperfield,rmk:FLT}): both are purely arithmetic conditions on the residue field. Moreover $(\star)$ is sharp for the present argument. If $(\star_n)$ fails, choose a class $c$ outside the image of $w\mapsto[w^n-1]_n$ with $c\neq[-1]_n$, and pick $g,h\in G_n^\times$ with $hg^{-1}=c$; then $g\notin C(g,h)$, so for $x=(0,g)$, $y=(0,h)$, $w=(0,g)$ one has $w\notin x\oplus y$, while $x_k=(-\tfrac1k,g)\to x$, $y_k=y$, $w_k=x_k\to w$ exhibit $(x,y,w)$ as a limit of points of $\Delta_n$. Thus the graph $\Delta_n$ is not closed and the proof of \Cref{prop:properness-XH} breaks down. 
\end{remark}

\subsection{Comparison maps and the tropical shadow}

In this section we show how the hyperfield formalism provides a natural way to obtain maps between analytifications. Everything follows from the following simple lemma.      

\begin{lemma}[Functoriality]\label{lem:trop-functorial}
Let $u\colon\mathbb H\to\mathbb H'$ be a homomorphism of topological hyperfields with $u\circ\rho_{\mathbb H}=\rho_{\mathbb H'}$. Then $\phi\mapsto u\circ\phi$ defines a continuous map $u_*\colon X(\mathbb H)\to X(\mathbb H')$ and, writing $u^\ell$ for the finite power of $u$,
\[
\trop_\cF^{\mathbb H'}\circ u_*=u^\ell\circ\trop_\cF^{\mathbb H}
\qquad\text{for every finite }\cF=(f_1,\dots,f_\ell).
\]
Moreover $u^{\ell}$ restricts to a map $\trop_\cF^{\mathbb H}(X(\mathbb H))\to\trop_\cF^{\mathbb H'}(X(\mathbb H'))$ between the images; these restrictions commute with the projections $\pi_\tau$ of \Cref{def:trop-XH} and induce a unique continuous map
\[
u_{\lim} \colon \varprojlim_\cF\trop_\cF^{\mathbb H}(X(\mathbb H))\longrightarrow\varprojlim_\cF\trop_\cF^{\mathbb H'}(X(\mathbb H'))
\]
satisfying $\pi_\cF\circ u_{\lim}=u^{\ell}\circ\pi_\cF$ for every $\cF$ with $|\cF|=\ell$, where $\pi_\cF$ denotes the projection to the $\cF$-component; and $\Trop^{\mathbb H'}\circ u_*=u_{\lim}\circ\Trop^{\mathbb H}$.
\end{lemma}

\begin{proof}
If $\phi\colon\ccO_X\to\mathbb H$ is a hyperring homomorphism with $\phi|_K=\rho_{\mathbb H}$, then $u\circ\phi\colon\ccO_X\to\mathbb H'$ is a hyperring homomorphism (hyperring homomorphisms are closed under composition) with $(u\circ\phi)|_K=u\circ\rho_{\mathbb H}=\rho_{\mathbb H'}$, so $u_*\phi=u\circ\phi\in X(\mathbb H')$. Continuity holds because $\mathrm{ev}_f\circ u_*=u\circ\mathrm{ev}_f$ for every $f$, and the target topology is coasest topology making the evaluation maps $\mathrm{ev}_f$ continuous. Next, 
$\trop_\cF^{\mathbb H'}(u_*\phi)=((u\phi)(f_i))_i=(u(\phi(f_i)))_i=u^\ell(\trop_\cF^{\mathbb H}(\phi)),$
which is the first square. That is $u^{\ell}$ restricts to the images is clear. For $\cF\subseteq\cF'$ the maps $u^{\ell}$ and $u^{\ell'}$ commute with the coordinate projections $\mathbb H^{\ell'}\to\mathbb H^{\ell}$, being defined coordinatewise; hence the restrictions form a morphism of inverse systems and induce a map $u_{\lim}$ between the limits, uniquely determined by $\pi_\cF\circ u_{\lim}=u^{\ell}\circ\pi_\cF$, and continuous since each composite $\pi_\cF\circ u_{\lim}$ is continuous and the target carries the initial topology of its projections. The total square commutes because it does after composing with every $\pi_\cF$, by the first square.
\end{proof}

We now instantiate the previous lemma in our context. 

\begin{proposition}\label{prop:T-shadow}
The coordinate projection $\tau\colon H\to\mathbb{T}_\Gamma$, $(\gamma,(g_n)_n)\mapsto\gamma$, is a continuous hyperfield homomorphism with $\tau\circ\rho=v$, for any value group $\Gamma_K$; it induces a continuous map $\tau_*\colon X(H)\to X(\mathbb{T}_\Gamma)$. 
\end{proposition}

We call the map $\tau_*$ \emph{the tropical shadow}. 

When $\Gamma$ is archimedean (rank $1$) and $(K,v)$ is complete, the target $X(\mathbb T_\Gamma)$ is a subspace of the associated Berkovich analytification, that is, the subspace $X^\an_{|K|}\subseteq X^\an$ of points whose seminorm takes all its values in $|K|$ (using multiplicative notation). When $\Gamma=\RR$, the tropical shadow map $\tau_*$ is then the passage from $X(H)$ to $H(\mathbb{T})=X^\an$, its underlying Berkovich analytification. 
When $K$ is real closed with value group $\RR$, the tropical shadow coincides with the map already considered in  \cite[\S3.2]{jell} 
\[
\tau_*\colon X_r^\an\to X^\an,\qquad (p_x,|\cdot|_x,\le_x)\mapsto(p_x,|\cdot|_x),
\]
dropping the order $\le_x$ and shown there to be proper \cite[Lemma~3.9]{jell}. They do not give it a name; \Cref{prop:T-shadow} exhibits it as the tropical shadow forgetting the residual decoration,  and \Cref{prop:shadow-proper} below recovers its properness in this generality.

Before turning to $\tau_*$ we isolate its finite-dimensional counterpart, the properness of the shadow between the two tropicalizations --- the analogue of the properness of $a\mapsto|a|$ in the proof of \cite[Lemma~3.9]{jell}. The tropicalization of the $\mathbb T_\Gamma$-points, $\trop_\cF^{\mathbb T_\Gamma}\colon X(\mathbb T_\Gamma)\to\mathbb T_\Gamma^\ell$ (\Cref{def:trop-XH} with $\mathbb H=\mathbb T_\Gamma$), is the \emph{classical tropicalization} when $\Gamma=\RR$ (for other ordered abelian groups, is has not been fully studied). Functoriality of tropicalization (\Cref{lem:trop-functorial}) applied to $u=\tau$ places the arithmetic tropicalization $\trop_\cF^{H}$ over it through the finite-power shadow $\tau^\ell\colon H^\ell\to\mathbb T_\Gamma^\ell$, in the commuting square
$\trop_\cF^{\mathbb T_\Gamma}\circ\tau_*=\tau^\ell\circ\trop_\cF^{H},$
exhibiting the classical tropicalization as the tropical shadow of ours.

\begin{lemma}\label{lem:tau-proper}
Assume $\Gamma=\RR$. Then the tropical shadow $\tau\colon H\to\mathbb T_\Gamma$ and each of its finite powers $\tau^\ell\colon H^\ell\to\mathbb T_\Gamma^\ell$ is a proper map. If moreover $\RF$ satisfies $(\star)$ (\Cref{def:star}) and $\cF$ generates $\ccO_X$, then $\tau^\ell$ restricts to a proper map
$\trop_\cF^{H}(X(H))\ \longrightarrow\ \trop_\cF^{\mathbb T_\Gamma}(X(\mathbb T_\Gamma))$
between the two tropicalizations.
\end{lemma}
\begin{proof}
For compact $C\subseteq\mathbb T_\Gamma=\RR_\infty$ one has $\tau^{-1}(C)=\bar v^{-1}(C)$, which is closed ($\tau$ being continuous) and bounded below ($C$ being so), hence compact by \Cref{lem:heineborel}; thus $\tau$ is proper. A finite product of proper maps between locally compact Hausdorff spaces is proper --- the spaces are such by \Cref{lem:heineborel} --- so $\tau^\ell$ is proper. Under $(\star)$ with $\cF$ generating, $\trop_\cF^{H}(X(H))$ is closed in $H^\ell$ (\Cref{cor:image-closed}); the restriction of a proper map to a closed subset of its source is proper, and by the commuting square its image is $\trop_\cF^{\mathbb T_\Gamma}(\tau_*(X(H)))\subseteq\trop_\cF^{\mathbb T_\Gamma}(X(\mathbb T_\Gamma))$, giving the stated map.
\end{proof}

We can now record the properness of the tropical shadow itself; we give a direct proof, and \Cref{rmk:jell-square} recovers it from \Cref{lem:tau-proper} with an analogous argument as in \cite[Lemma 3.9]{jell}.

\begin{proposition}\label{prop:shadow-proper}
Assume $\Gamma=\RR$ and that $\RF$ satisfies $(\star)$ (\Cref{def:star}). Then the tropical shadow $\tau_*\colon X(H)\to X(\mathbb T_\Gamma)$ is a proper map.
\end{proposition}
\begin{proof}
Let $B\subseteq X(\mathbb T_\Gamma)$ be compact. For each $f\in\ccO_X$ the evaluation $\mathrm{ev}_f\colon X(\mathbb T_\Gamma)\to\mathbb T_\Gamma$ is continuous, so $\mathrm{ev}_f(B)$ is a compact subset of $\mathbb T_\Gamma=\RR_\infty$, hence bounded below: there is $\gamma_f\in\RR$ with $w(f)\ge\gamma_f$ for every $w\in B$. Since $\tau$ preserves $\bar v$, every $\phi\in\tau_*^{-1}(B)$ satisfies $\bar v(\phi(f))=(\tau_*\phi)(f)\ge\gamma_f$; hence $E(\tau_*^{-1}(B))\subseteq\prod_{f}D_f$, where $E\colon X(H)\hookrightarrow\prod_f H$ is the evaluation embedding and $D_f:=\bar v^{-1}([\gamma_f,\infty])\subseteq H$ is compact (\Cref{lem:heineborel}); by Tychonoff $\prod_f D_f$ is compact.

Write $E'\colon X(\mathbb T_\Gamma)\hookrightarrow\prod_f\mathbb T_\Gamma$ for the evaluation embedding of the target --- a homeomorphism onto its image, by the argument of \Cref{prop:Payne-XH} --- and $T:=\prod_f\tau\colon\prod_f H\to\prod_f\mathbb T_\Gamma$, which is continuous. Since $T\circ E=E'\circ\tau_*$ and $E'$ is injective, for $\phi\in X(H)$ one has $E(\phi)\in T^{-1}(E'(B))$ if and only if $\tau_*\phi\in B$; hence
\[
E(\tau_*^{-1}(B))=E(X(H))\cap T^{-1}(E'(B)).
\]
Now $E'(B)$ is compact, hence closed in the Hausdorff space $\prod_f\mathbb T_\Gamma$ (\Cref{lem:hausdorff}), so $T^{-1}(E'(B))$ is closed in $\prod_f H$; and $E(X(H))$ is closed in $\prod_f H$ by \Cref{lem:Eclosed} (this is where $(\star)$ is used). Therefore $E(\tau_*^{-1}(B))$ is closed and contained in the compact set $\prod_f D_f$, hence compact. As $E$ is a homeomorphism onto its image, $\tau_*^{-1}(B)$ is compact.
\end{proof}

\begin{remark}\label{rmk:jell-square}
\Cref{prop:shadow-proper} also follows by the diagram-chase of \cite[Lemma~3.9]{jell}, factoring it through the properness of the arithmetic tropicalization. Assume $\cF$ generates $\ccO_X$, and let $B\subseteq X(\mathbb T_\Gamma)$ be compact. The commuting square $\trop_\cF^{\mathbb T_\Gamma}\circ\tau_*=\tau^\ell\circ\trop_\cF^{H}$ gives
\[
\tau_*^{-1}(B)\ \subseteq\ (\trop_\cF^{H})^{-1}\big((\tau^\ell)^{-1}(\trop_\cF^{\mathbb T_\Gamma}(B))\big),
\]
whose right-hand side is compact: $\trop_\cF^{\mathbb T_\Gamma}(B)$ is compact, $\tau^\ell$ is proper (\Cref{lem:tau-proper}), and $\trop_\cF^{H}$ is proper (\Cref{prop:properness-XH}). As $\tau_*^{-1}(B)$ is closed, it is compact.
\end{remark}

\begin{remark}\label{rmk:shadow-proper}
For $K$ real closed with value group $\RR$, \Cref{prop:shadow-proper} is exactly the properness of the forgetful map $X_r^\an\to X^\an$ of \cite[Lemma~3.9]{jell}, under the identifications $H=\mathbb{RT}$, $X(H)=X_r^\an$ and $X(\mathbb T_\Gamma)=X^\an$ (\Cref{prop:H-reduction}). It is already new when $\RF$ is a local field: for $K=\Qq_p(\!(t^\RR)\!)$ the fibres of $\tau\colon H\to\mathbb T_\Gamma$ are the infinite profinite sets $\varprojlim_n G_n^\times$, yet the tropical shadow $\tau_*$ remains proper. 
\end{remark}

\section{The model-theoretic analytification}\label{sec:analytification}

In this section we introduce, for a certain class of tame valued fields, the model-theoretic avatar of the space $X(H)$ introduced in \Cref{sec:hyperfield-points}. Its construction extends Hrushovski-Loeser's stable completion from \cite{HruLoe} and builds on many of their ideas. We start with some preliminaries.

\subsection{Model-theoretic preliminaries}\label{subsec:mt-prelim}

We use Macintyre's language
$\Lmac=\{+,-,\cdot,0,1,(P_n)_{n\geq2}\}$, where $P_n$ denotes the set
of $n$th powers, including zero, and write $P_n^\times=P_n\setminus\{0\}$.
The language $\Lmac^\Gamma$ has a field sort, a value-group sort in the
language of ordered abelian groups, the distinguished point $\infty$,
and the valuation map $v$.

\begin{definition}\label[definition]{def:class}
Let $(K,v)$ be a henselian valued field of equicharacteristic $0$. We say it is a \emph{bounded Macintyre field}, or just a $\bm$-field, if
\begin{enumerate}
    \item $[\RV_K^\times:(\RV_K^\times)^n]<\infty$ for every $n\geq 2$;
    \item the theory of the residue field eliminates quantifiers in $\Lmac$, up to addition of constants.
\end{enumerate}
\end{definition}

Examples include algebraically closed and real closed valued fields,
as well as the power-series fields $\lau{L}{t}$ and $\Lau{L}{\RR}$ for
$L=\CC,\RR,\QQ_p$. Iterated constructions, such as
$\Lau{\lau{\CC}{x}}{\RR}$, provide further examples.

For the rest of the article we fix a complete $\Lmac^\Gamma$-theory $T$ of a $\bm$-field, a
sufficiently saturated and homogeneous universal domain $\cU\models T$,
and a small elementary submodel $K\prec\cU$.  We work in $T^{\mathrm{eq}}$. Parameter sets,
sets of codes, and indexing systems are always small relative to the saturation of $\cU$. 
Following \cite[Sections~2.2--2.3]{HruLoe}, we identify a definable set $Y$ with the set of $\cU$-points $Y(\cU)$. We write $Y(M)$ only when specifying points over a model or parameter set $M$. 

A \emph{pro-definable set} is a small inverse limit of definable sets,
with definable transition maps. It is \emph{strict} if it admits such a
presentation with surjective transition maps. A subset is
\emph{relatively definable} if it is cut out by a formula involving
finitely many coordinates in a pro-definable presentation. As for definable sets, we identify a pro-definable set $Y$ with its set of $\cU$-points $Y(\cU)$.\footnote{That this is a correct identification follows by \cite[Corollary 10]{kamensky}}. Types (partial and complete) over small parameter sets are also viewed as pro-definable sets.

When an affine variety $X$ over $K$ is regarded as a definable set, we
implicitly base change to $\cU$. Thus $\ccO_X=\cU[X]$ in statements
about universal-domain points; statements over a specified model $M$
use $M[X]$.

For a parameter set $C$, we write
$\Gamma(C)=\Gamma\cap\operatorname{dcl}^{\mathrm{eq}}(C)$.
For a field $L$, $\Gamma(L)$ is its value group. A basis
$(e_i)$ of a valued vector space is \emph{valuation-independent} if
\[
 v\left(\sum_i c_i e_i\right)=\min_i\bigl(v(c_i)+v(e_i)\bigr)
\]
for every coefficient tuple from the base field, with $v(0)=\infty$.

We shall use the following relative quantifier-elimination result.

\begin{proposition}\label[proposition]{prop:QE} Let $(K,v)$ be a $\bm$-field and $T=\mathrm{Th}(K,v)$. Then, up to addition of constants, $T$ has relative elimination of field quantifiers in $\Lmac^\Gamma$. Moreover the residue field $\RF$ and the value group $\Gamma$ are orthogonal and fully stably embedded.
\end{proposition}

\begin{proof}
Since $(K,v)$ is henselian of equicharacteristic 0, the field sort eliminates quantifiers relative to the $\RV$-sort \cite[Proposition~5.1]{flenner}.
In turn, the $\RV$-sort eliminates quantifiers down to $\RF$ and $\Gamma$ by \cite[Proposition~3.1]{inpminimal}. This also implies that the two sorts are moreover fully stably embedded \cite[Corollary~3.4]{inpminimal}
 and, being tied only through $\bar v$, orthogonal (more generally, this is the relative elimination of pure short exact sequences of abelian structures \cite{acgz}). Thus every formula in field variables $x$ is equivalent to a Boolean combination of quantifier-free $\Lmac$-formulas in $x$, of conditions $\psi(\res(g_1(x)/h_1(x)),\dots, \res(g_\ell(x)/h_\ell(x)))$ on ratios of equal value with $\psi$ an $\Lmac$-formula, and of conditions $\theta(v(g_1(x)),\dots, v(g_\ell(x)))$ with $\theta$ a formula in the language of ordered abelian groups. It suffices to encode the information of $\psi$. 

Since $K$ is a $\bm$-field, we may assume that $\psi$ is quantifier-free (up to constants). Partition first into cases where the denominators are nonzero and the rational expressions whose residues occur are integral. A residue equation $\res(A/B)=0$ is then equivalent to $v(A)>v(B)$. For a nonzero residue, the condition that $\res(dR(Y))$ be a $k$-th power is equivalent, by Hensel's lemma, to $P_k(dR(Y))$, provided $v(dR(Y))=0$. The case of zero residue is handled separately by $v(dR(Y))>0$. Finally, on $B\neq0$ one has
\[
P_k(A/B)\ \Longleftrightarrow\ P_k(AB^{k-1}),
\]
so clearing denominators gives power conditions on polynomials. This proves the asserted relative elimination of field quantifiers, allowing value-group variables in the formulas $\theta$.
 \end{proof}

Lastly, a valued field
is \emph{maximally complete} if it has no proper immediate extension,
where immediate means that the residue field and value group are
unchanged. Every small parameter set is contained in such a small model:
start with an elementary submodel containing the parameters and take a
maximal immediate extension. The extension is elementary by the
Ax--Kochen--Ershov principle in residue characteristic zero, and it embeds over the original model
into a sufficiently large universal domain/monster model. Maximal completeness also
implies that every finite-dimensional valued vector space over the model
has a valuation-independent basis (see \cite[Corollary 4.3]{cubi-kuh-rze}). 

\subsection{The model-theoretic analytification \texorpdfstring{$\widehat{X}_T$}{of X}}

For a definable set $Y$ over $C$, let $S_Y(C)$ be the space of complete
types over $C$ concentrating on $Y$. A global type is a complete type
over $\cU$. It is \emph{$C$-definable} if, for every formula
$\varphi(x;y)$, there is a formula $d_p\varphi(y)$ over $C$ such that
\[
 \varphi(x;b)\in p\quad\Longleftrightarrow\quad
 \cU\models d_p\varphi(b).
\]
It is \emph{$C$-invariant} if every automorphism of $\cU$ fixing $C$
fixes $p$; definability over $C$ implies invariance over $C$.
We write $S^{\mathrm{def}}_Y$ for all global definable types on $X$,
allowing arbitrary small sets of definition parameters. The same
notation applies to pro-definable sets: formulas involve finitely many
coordinates, and the type contains a small partial type presenting the
set. A definable type over a model is identified with the global
extension given by its definitions.

\begin{definition}\label{def:weak-ortho}\
\begin{enumerate}
\item A type $p\in S_Y(K)$ is \emph{weakly orthogonal to the value group}
if $\Gamma(K(a))=\Gamma(K)$. We denote this subspace by
$S_X^{\perp^w\Gamma}(K)$. 

\item  For a definable map $f:Y\to Z$, the \emph{pushforward} $f_*p$ is defined
by $\psi(y;b)\in f_*p$ if and only if $\psi(f(x);b)\in p$.
For pro-definable maps this definition is applied coordinatewise.
\item A global definable type $p$ is \emph{orthogonal to $\Gamma$} if every
definable map $f:X\to\Gamma_\infty^r$ has realized pushforward $f_*p$.
We denote that realized value by $f(p)$.
\end{enumerate}
\end{definition}

Note that if $p$ is weakly-orthogonal to $\Gamma$ we have in particular that $v(f(a))\in\Gamma(K)_\infty$ for every $f\in K[X]$. Also, for $\bm$-fields, this is equivalent to the fact that $p(x)\cup q(\eta)$ is complete over $K$ for every complete type
$q\in S_{\Gamma^r}(K)$ and every finite $r$ (which is the standard definition given in \cite{HruLoe} for weak orthogonality).

\begin{definition}\label[definition]{def:analytification}
Let $Y$ be a definable set. We write $\widehat Y_T$ for the global definable types concentrating on $Y$ and orthogonal to $\Gamma$. If $C$ is a small parameter set over which $X$ is defined, let $\widehat Y_T(C)$ be the subset of global $C$-definable types. We omit the subscript $T$ when it is clear. 
\end{definition}

For definable types over models, weak orthogonality implies orthogonality. 

\begin{lemma}\label{lem:def-orthogonal}
Let $p\in \mathcal S_Y^{\perp^w\Gamma}(K)$ and suppose $p$ is $K$-definable. Then $p\in \widehat{Y}_T(K)$. In particular, its global extension belongs to $\widehat{Y}_T$. 
\end{lemma}

\begin{proof}
Since the residue field and value group are stably embedded and orthogonal (\Cref{prop:QE}) and $p$ is $K$-definable --- hence $K$-invariant --- the result follows by \cite[Theorem~3.27]{orthogonal}.
\end{proof}

\begin{corollary}\label[lemma]{cor:def-orthogonal}
Let $(K,v)$ be a maximally complete $\bm$-field. Suppose that over its residue field $\RF(K)$ every type is definable. Then all types in $S_Y^{\perp^w\Gamma}(K)$ are $K$-definable. In particular, the restriction function from $\widehat{Y}_T(K)$ to $S_Y^{\perp^w\Gamma}(K)$ is a bijection. \end{corollary}

\begin{proof}
Let $p\in S_Y^{\perp^w\Gamma}(K)$. Fix $a\models p$ in an elementary extension of $K$, and put $L=K(a)$. Maximal completeness implies that $L/K$ has a valuation basis. Since $\Gamma(L)=\Gamma(K)$, any finite tuple $\xi_i\in\RV(L)^\times$ can be normalized by $d_i\in K$ of the same values, leaving $\xi_i/\rv(d_i)\in\RF(L)^\times$. By the hypothesis on residue types and pure stable embeddedness of $\RF$, these normalized tuples have definable types over $\RV(K)$. The value-group coordinates already lie in $\Gamma(K)$. Thus $\tp(\RV(L)\cup\Gamma(L)/\RV(K))$ is definable, and \cite[Proposition~6.3.5(1)]{beautiful} gives definability of $p$. The last assertion follows by \Cref{lem:def-orthogonal}.  
\end{proof}

Relative quantifier elimination gives the following criterion for equality of elements in $\widehat{X}_T$. 

\begin{corollary}\label[corollary]{cor:type-determination}
Let $T$ be the theory of the $\bm$-field $K$ and $p,q\in \widehat{X}_T$ where $X$ is an affine variety base changed to $\mathcal{U}$ and $\ccO_X=\mathcal U[X]$. Assume that:
\begin{enumerate}
\item for every $f\in \ccO_X$, $v(f(p))=v(f(q))$;
\item for every $f\in \ccO_X$, every natural number $n \in \Nn$ and every representative $a_{i}$ of $K^{\times}/ (K^{\times})^{n}$ we have $P_{n}(a_{i}f(x)) \in p$ if and only if $P_{n}(a_{i}f(x)) \in q$.
\end{enumerate}
Then $p=q$. 
\end{corollary} 

\begin{proof}
Note that condition (1) already implies that for every $f\in \ccO_X$ we have $f(x)=0 \in p$ if and only if $f(x)=0 \in q$ since $v(f(p))=\infty$ if and only if $v(f(q))=\infty$. The conclusion now follows from \Cref{prop:QE}. 
\end{proof}

\begin{remark}
The hyperfields introduced in the preceding sections are (pro)definable objects. Indeed, note that since $\RV$ is definable in $\eq{T}$, so are the quotients $\RV^\times/(\RV^\times)^n$. The valuation induces a map $\RV^\times/(\RV^\times)^n\to\Gamma/n\Gamma$, and therefore $H_n^\times = \Gamma\times_{\Gamma/n\Gamma}\RV^\times/(\RV^\times)^n$ (and $H_n=H_n^\times\cup \{\mathbf{0}\}$) is definable in $\eq{T}$. The map $\rho_n$, multiplication and the ternary hyperaddition relation are also definable. Finally, since the maps $\tau_{mn}:H_n\to H_m$ for $m\mid n$ are definable surjections, the inverse limit $H=\varprojlim_n H_n$ is a countable strict pro-definable set (in general, $H$ does not need to be is not a single definable set). We equip $H_n$ and $H$ with the (definable) topologies already defined and treat this sets as definable sets, so evaluated in $\mathcal U$.     
\end{remark}

We realize $\widehat{X}_T$ inside the space of hyperfield points $X(H)$ of \Cref{sec:hyperfield-points}, using $\ccO_X=\mathcal U[X]$. The topology on $\widehat{X}_T$ will be induced by this embedding. Abusing notation, we will aslo use $\infty$ to denote the hyperfield element $\mathbf{0}$.

\begin{proposition}\label[proposition]{prop:hyperfield-embedding}
Let $X$ be an affine variety over $K$, base changed to $\mathcal U$ in this assertion, with coordinate ring $\ccO_X=\mathcal U[X]$. The map
$\Theta\colon \widehat{X}_T \to X(H),$ defined by $\Theta(p)(f)=\rho(f(p))$, is a well-defined injection.
\end{proposition}
\begin{proof}
To see it is well-defined, write $\phi_p=\Theta(p)$ and fix $a\models p$; then $\rho(f(p))=\rho(f(a))$ is independent of the choice of $a$: its valuation is realized by orthogonality, and its remaining coordinate at every level lies in a finite power quotient. Here $a$ realizes the global type in an elementary extension of $\mathcal U$, but $\rho(f(a))$ lives entirely in $\mathcal{U}$. For $f,g\in\ccO_X$, multiplicativity of $\rho$ gives $\phi_p(fg)=\rho(f(a)g(a))=\phi_p(f)\phi_p(g)$, while the hyperaddition on each $H_n$ gives
\[
\phi_p(f+g)=\rho\big((f+g)(a)\big)\in\rho(f(a))\oplus\rho(g(a))=\phi_p(f)\oplus\phi_p(g),
\]
At each finite level the witnesses $f(a),g(a)$ to hyperaddition descend to $\mathcal U$ since their values lie in $H_n$. The inverse-limit definition gives the displayed inclusion in $H$. On constants $\phi_p$ agrees with $\rho$, so $\phi_p\in\Hom_{\mathcal U}(\ccO_X,H)$.

It remains to show injectivity. If $p\neq q$, by \Cref{cor:type-determination}, they are separated by a formula of one of following two kinds. Either $v(f(p))\neq v(f(q))$ for some $f\in\ccO_X$, and therefore $\rho(f(p))\neq\rho(f(q))$; or $p$ and $q$ agree on all such values but differ on a power-coset formula $P_n(af(x))$ ($a\in K^\times$), in which case $v(f(p))=v(f(q))$ while $[f(p)]_n\neq[f(q)]_n$, so $\rho_n(f(p))\neq\rho_n(f(q))$. Either way $\Theta(p)\neq\Theta(q)$.
\end{proof}

\begin{definition}\label[definition]{def:TheTopology}
We equip $\widehat{X}_T$ with the coarsest topology making the embedding $\Theta$ of \Cref{prop:hyperfield-embedding} continuous. We call $\widehat{X}_T$ the \emph{model-theoretic analytification of $X$ with respect to $T$}. For a definable subset $Y\subseteq X$, we equip $\widehat Y_T\subseteq\widehat X_T$ with the subspace topology. 
\end{definition}

\begin{remark}\label{rem:topology_over_M}
For a specified model $M$ over which $X$ is defined, the same construction with $M[X]$ gives an embedding $\Theta_M:\widehat X_T(M)\to X(H(M))$ and its intrinsic topology, using evaluation parameters from $M$. Note that this need not be the subspace topology induced from the universal domain. In addition, $\Theta_M: \widehat{X}_T(M)\to X(H(M))$ factorizes into the following canonical injections
\begin{equation}\label{eq:comparison}
 \widehat X_T(M)
 \longrightarrow \
 S_X^{\perp^w\Gamma}(M)
 \longrightarrow\
 X(H(M)),
\end{equation}
where the first map is restriction to $K$ and the second map, which we call $\Xi_M$, is defined by the same formula as $\Theta_M$
\[
\Xi_M(p)(f)=(\rho_n(f(a)))_n.
\]
for any (some) realization $a$ of $p$. Note that using an analogue of \Cref{cor:type-determination} for weakly orthogonal types, the proof of \Cref{prop:hyperfield-embedding} shows that $\Xi_M$ is injective. Hence, if we also equipp $S_X^{\perp^w\Gamma}(M)$ with the coarsest topology making $\Xi_M$ continuous, both $\Theta_M$ and $\Xi_M$ are are topological embeddings.
\end{remark}

\begin{remark}\label{rem:opens}
The following is a concrete subbasis of open sets for the topology of $\widehat{X}_T$. Let $U\subseteq X$ be a Zariski open subset. For $f\in\ccO_X(U)$, a definable
open interval $I\subseteq\Gamma_\infty$, and $a\in\cU^\times$, put
\[
\begin{aligned}
 (U,I)_f&=\{p\in\widehat U:v(f(p))\in I\},\\
 (U,aP_n^\times)_f&=\{p\in\widehat U:P_n^\times(a^{-1}f(x))\in p\}.
\end{aligned}
\]
These give the valuation and nonzero power-coset opens. To describe
neighbourhoods of zero, choose representatives $a_1,\ldots,a_{c(n)}$
of the nonzero $n$th-power classes and bounds $\gamma_i\in\Gamma$.
Writing $V_f=\{x\in U:f(x)=0\}$, also include
\[
 \widehat V_f\ \cup\
 \bigcup_{i=1}^{c(n)}
 \bigl((U,a_iP_n^\times)_f\cap(U,(\gamma_i,\infty))_f\bigr).
\]
These sets are the inverse images of the corresponding subbasic opens
of $H_n$. In particular, the topology is generated by finite-level
evaluations of regular functions. Using this subbasis, it is easy to see that every basic open is of the form $\widehat{Y}_T\subseteq \widehat{X}_T$ for a definable subset $Y\subseteq X$. 
\end{remark}

In the algebraically closed case this construction recovers the stable
completion, including its topology.

\begin{proposition}\label[proposition]{prop:ACVF-HL}
Suppose $T=\ACVF$. Then, the topological spaces $\widehat{X}_T$ defined in Definition \ref{def:TheTopology} and $\widehat{X}$ as defined by Hrushovski--Loeser in \cite[Sections~3.3--3.5]{HruLoe} are equal. The same holds for the intrinsic topologies on points over a specified model.
\end{proposition}

\begin{proof}

In $\ACVF$, the definable types orthogonal to $\Gamma$ are precisely
the stably dominated types, so the underlying sets agree. Since every
nonzero element is an $n$th power, $H=\Gamma_\infty$ and the topology on $H$ is hence determined by the valuation alone. Its pullback is therefore the topology generated
by valuations of regular functions on Zariski open subsets, as in
\cite[Sections~3.3--3.5]{HruLoe}.
\end{proof}

\subsection{ When is $\Theta$ a homeomorphism?}

In this section we work over the model $K$ and address the following natural question: when is the map $\Theta_K\colon \widehat{X}_T(K)\to X(H(K))$ (resp. the map $\Xi_K \colon  S_X^{\perp^w\Gamma}(K) \to X(H(K))$) a homeomorphism? On the positive side, we will show that for $\bm$-fields of the form $\Lau{k}{\RR}$ with $k$ either algebraically closed, real closed or $p$-adically closed, both $\Theta_K$ and $\Xi_K$ are homeomorphisms (in particular, restriction to $K$ from $\widehat{X}_T(K)$ to  $S_X^{\perp^w\Gamma}(K)$ is a homeomorphism). The result for $k$ algebraically closed is essentially putting together results from \cite{HruLoe} and \cite{Jun}. On the negative side, in \Cref{prop:counterexample} we provide an example where the map $\Xi_K$ is not surjective.

When working with Hahn series fields of the form $\Lau{k}{\Gamma}$ with $\Gamma$ a divisible, the cross-section $\gamma \mapsto t^\gamma$ allow us to use coordinates by letting $\rho_n(a)=(v(a),[\res(t^{-v(a)}a)]_n)$ for $a\ne0$. Thus, we may identify 
\[
H_n(K)=\{\mathbf{0}\}\sqcup(\Gamma\times G_n^\times),\qquad
 H(K)=\{\mathbf{0}\}\sqcup(\Gamma\times G^\times),
\]
where $G=\varprojlim G_n^\times \sqcup\{\mathbf{0}\}$. For $\phi\colon \ccO_X\to H$ an element of $X(H(K))$, the ideal $\mathfrak{p}=\ker{\phi}$ is prime. Hence, the morphism $\phi$ descends to $K[X]/\mathfrak p$ and extends to its fraction field $F=\mathrm{Frac}(\ccO_X/\mathfrak p)$. Composing with the projection $\mathrm{pr_1}$ on $H(K)$ to $\Gamma(K)$ gives a valuation $v_F$ on $F$ extending $v$ with $v_F(F)=\Gamma(K)$. In addition, the projection $\mathrm{pr_1}$ of $H$ to $G$ factors through the residue field $\RF_F$. Indeed, suppose that $v_F(x)=v_F(y)=0$ and $v_F(x-y)>0$. Hyperaddition at $x=y+(x-y)$ gives $\phi(x)\in\phi(y)\oplus \phi(x-y)$, and since $v(y)<v(x-y)$, we get $\phi(y)\oplus \phi(x-y)=\{\phi(y)\}$, thus $\phi(x)=\phi(y)$. Therefore, we obtain a hyperfield morphism $\overline\phi:k_F\longrightarrow G$ extending $\mathrm{pr}_2\circ\rho$ on $\RF_{K}$. This shows that for $f\in F^\times$ 
\begin{equation}\label{eq:normalized-label}
 \phi(f)=\bigl(v_F(f),
       \overline\phi(\res_{v_F}(t^{-\gamma}f))\bigr).
\end{equation}
Using this, the following lemma gives us a criterion for $\phi$ to have a pre-image in $S_X^{\perp^w\Gamma}(M)$ under $\Xi_K$. 

\begin{lemma}\label{lem:criterion for pre image}
Let $\phi\in X(H(K))$ and $F$ be defined as above. If $\RF_F$ embeds into an elementary extension of $\RF_K$ realzing the power-class data by $\phi$, then there is $p\in S_X^{\perp^w\Gamma}(K)$ such that $\Xi_K(p)=\phi$.  
\end{lemma}

\begin{proof}
Let $k'$ be such that $\RF_F$ embeds into $k'$ and $\RF_K\preceq k'$. Take a henselian valued extension $L/F$ with value group $\Gamma(K)$ and residue field $k'$, inducing this chosen residue embedding. Ax--Kochen--Ershov gives $(K,v)\preccurlyeq(L,v)$. Let $a\in X(F)\subseteq X(L)$ be the tuple given by the coordinate functions. We show that $p=\tp(a\K)$ works. First, $\Gamma(K(a))=\Gamma(K)$ since $K(a)\subseteq L$ and
$\Gamma(L)=\Gamma(K)$, so $p\in S_X^{\perp^w\Gamma}(K)$. By construction, $v_F$ is precisely the valuation induced by $\phi$. Fix $f\in \ccO_X$ and $n\in\mathbb{N}_{\geq 2}$. By divisibility of $\Gamma$, $t^\gamma$ is an $n$th power in $K$. For $c\in\RF_K^\times$, Hensel's lemma gives
\[
 P_n\left(\frac{f(a)}{t^\gamma c}\right) \quad\Longleftrightarrow\quad P_n\left(\frac{\res(t^{-\gamma}f(a))}{c}\right).
\]
Together with \eqref{eq:normalized-label} and the choice of $k'$, this
shows that $\rho_n(f(a))=\phi_n(f)$ for every $f,n$, hence $\Xi(p)=\phi$. 
\end{proof}

We now show that when $\RF_K$ is algebraically closed, real closed or $p$-adically closed, the hypothesis of \Cref{lem:criterion for pre image} are always satisfied. 

\begin{proposition}\label{prop:surjectivity}
Let $\phi\in X(H(K))$ and $F$ as defined above. If $\RF_K$ is either algebraically closed, real closed, or $p$-adically closed, then $\RF_F$ embeds into an elementary extension of $\RF_K$.   
\end{proposition}

\begin{proof}
If $\RF_K$ is algebraically closed the result is trivial since every field embeds into an elementary extension of $\RF_K$. Suppose that $\RF_K$ is real closed. In this case, $G = \{-1,0,1\}=\mathbb{S}$ the sign hyperfield and $H^\times = \RR\times \mathbb{S}$. Therefore, we have a hyperfield morphism $\bar{\varphi}\colon \RF_F\to \mathbb{S}$. This induces an ordering on $\RV_F$ by setting 
 $a>0$ if and only if  $\bar{\varphi}(a)=1$,
which makes $\RF_F$ an ordered field (here you use that $\bar{\varphi}$ is a hyperfield morphism). Since every order field embeds into its real closure, $\RF_F$ embeds into an elementary extension of $\RF_K$. 

Finally, for $p$-adically closed fields, this is the content of the Appendix, see \Cref{lem:residue-realization}. 
\end{proof}

\begin{theorem}\label{thm:surjective1}
Suppose $k$ is either algebraically closed, real closed or $p$-adically closed. Let $\Gamma$ be a divisible ordered abelian group. For the $\bm$-field $K=\Lau{k}{\Gamma}$, the map $\Xi_K$ is a homeomorphism. 
\end{theorem}

\begin{proof}
Since $\Xi_K$ is a topological embedding (see \Cref{rem:topology_over_M}), it suffices to show it is surjective. The result now follows from \Cref{prop:surjectivity} and \Cref{lem:criterion for pre image}. 
\end{proof}

\begin{corollary}\label{cor:homeomorphism_theta}
Suppose $k$ is either algebraically closed, real closed or $p$-adically closed and all types over $k$ are definable. For the $\bm$-field $K=\Lau{k}{\RR}$ the map $\Theta_K$ is a homeomorphism. 
\end{corollary}

\begin{proof} By \Cref{cor:def-orthogonal}, restriction to $K$ is a homeomorphism from  $\widehat{X}_T(K)$ to $S_X^{\perp^w\Gamma}(K)$. The result follows by \Cref{thm:surjective1}.  \end{proof}

However, $\Xi_K$ is not surjective in general. We will now construct a point of $\mathbb A^1(H(K))$ which is induced by no complete type over $K$. In fact, the obstruction already occurs at the induced map on $H_2$.

\begin{proposition}\label{prop:counterexample}
Let $k=\lau{\mathbb{C}}{s}$ and $K=\Lau{k}{\mathbb{R}},$
and equip $K$ with the $t$-adic valuation. Then $K$ is a bm-field,
but $ \Xi_K:S_{\mathbb A^1}^{\perp^w\Gamma}(K)
       \to\mathbb A^1(H(K))$
is not surjective.
\end{proposition}
\begin{proof} The field $K$ is henselian of equicharacteristic zero, with value group $\RR$ and residue field $\lau{\mathbb{C}}{s}$. Note that $\lau{\mathbb{C}}{s}^\times/(\lau{\mathbb{C}}{s}^\times)^n$ is $\ZZ/n\ZZ$. Indeed, given $u\in k^\times$ we can write $u=s^m c(1+h)$ where 
$m\in\mathbb Z,$ $ c\in\mathbb C^\times$,
$h\in s\mathbb C[[s]].$
Since $\mathbb C$ is algebraically closed, $c$ admits an
$n$th root, and the formal binomial series
$(1+h)^{1/n}$ belongs to $1+s\mathbb C[[s]]$.
Consequently, $u$ is an $n$th power in $k$ if and only if
$n\mid m$. Thus the $s$-adic valuation induces an isomorphism
$k^\times/(k^\times)^n
\cong
\mathbb Z/n\mathbb Z.$ Since $\mathbb R$ is divisible, it follows that
$K^\times/(K^\times)^n
\simeq k^\times/(k^\times)^n
\simeq \mathbb Z/n\mathbb Z.$

Now, choose an irrational real number
$\alpha\in(-\tfrac12,-\tfrac14)$ and let
$\Lambda=\ZZ+\QQ\alpha\subseteq\RR$,  $k'=\mathbb{C}(\!(s^\Lambda)\!)$, and $E=\Lau{k'}{\RR}$. These are Hahn fields and $K\subseteq E$.

 \begin{claim*} 
For $n\ge2$, $H_n(K)\xrightarrow{\sim}H_n(E)$, (computed with respect to the $t$-adic valuation).
 \end{claim*}
 \begin{proof}
For $A=\ZZ$ or $\Lambda$, write $k_A=\mathbb{C}(\!(s^A)\!)$. 
We have then two valuations on $k_{A}$: one is the $s$-adic valuation, which we refer to as the inner valuation; and the $t$-adic valuation that it inherits from $E$. We will refer to the second one as outer valuation.
As before, the inner valuation
identifies $k_A/(k_A^\times)^n=\{\mathbf{0}\}\sqcup A/nA.$
 The inclusion
induces $\ZZ/n\ZZ\simeq\Lambda/n\Lambda$. Indeed, since $\alpha\notin\mathbb Q$, we have $\mathbb Z\cap\mathbb Q\alpha=\{0\}$. Hence
$\Lambda=\mathbb Z\oplus\mathbb Q\alpha$.  $\mathbb Q\alpha$ is divisible, so $n\Lambda=n\mathbb Z\oplus\mathbb Q\alpha$. Consequently, the inclusion $\mathbb Z\to\Lambda$ induces an isomorphism
$\mathbb Z/n\mathbb Z \cong \Lambda/n\Lambda$. Note that the hyperadditions also coincide. For nonzero elements labelled by
$i,j\in A/nA$, one has
\begin{equation}\label{eq:residue-addition-counter}
 i\boxplus j=
 \begin{cases}
 \{i,j\},&i\ne j,\\
 \{\mathbf{0}\}\sqcup A/nA,&i=j.
 \end{cases}
\end{equation}
Indeed, if $i\ne j$, representatives of these classes have
distinct inner valuations, so the class of their sum is the
class of the summand of smaller valuation. Both possibilities
occur, since adding elements of $n\mathbb Z\subseteq nA$ to
the valuations does not change their classes.\\
Suppose now that $i=j$. Choose $a\in A$ representing $i$.
For any $\ell\in A/nA$, choose $b\in A$ representing $\ell$
with $b>a$; this is possible since $n\mathbb Z\subseteq nA$.
The elements $x=s^a,$ where $y=-s^a+s^b$ both represent $i$, whereas $x+y=s^b$ represents $\ell$.
Here $y$ represents $i$ because its leading coefficient
is $-1$, which is an $n$th power in $\mathbb C$, and its
principal-unit factor is an $n$th power by Hensel's lemma.
Taking $y=-x$ also gives the zero element.
This proves \eqref{eq:residue-addition-counter}.

Both outer hyperfields $H_n(K)$ or $H_n(E)$ have nonzero elements since the outer valued field have value group $\mathbb{R}$ with leading coefficient in $k_A$, whose inner valuation gives $G_n(k_A)\simeq \ZZ/n\ZZ$, so they can be identified as $\RR\times\ZZ/n\ZZ$. At unequal outer values, their sum is the
smaller-value summand. At a common outer value $\gamma$, different
labels in $\ZZ/n\ZZ$ give just the two input elements, by
\eqref{eq:residue-addition-counter}; and the sum of two elements with equal labels gives precisely
$ \{\mathbf{0}\}\cup
 \{(\delta,j):\delta\ge\gamma,\ j\in\ZZ/n\ZZ\}.$ Thus the additions as well as the multiplications agree. These identifications commute with
all transition maps, so $H(K)\simeq H(E)$ over $K$. This concludes the proof of the Claim. 
 \end{proof}

Evaluate at $a=s^\alpha\in E$. Through this identification we obtain
$\phi:K[T]\to H(K),$ defined by $f\to\rho_E(f(a)).$ This is a hyperfield morphism over $K$ with
\begin{equation}\label{eq:inconsistent-values}
 \phi(1+sT^2)=\rho_K(1),\qquad
 \phi(1+sT^4)=\rho_K(s).
\end{equation}
For the first equality, $1+2\alpha>0$, so $1+s^{1+2\alpha}$ is in $1+\mathfrak{m}_s$, where $\mathfrak{m}_s$ denotes the maximal ideal w.r.t. the s-adic valuation. It is an $n$th power for every $n$.
For the second one, $1+4\alpha<0$, and $1+s^{1+4\alpha}=s^{1+4\alpha}(1+s^{-1-4\alpha}).$ The parenthesized factor is again in $1+\mathfrak{m}_s$, while
$4\alpha\in n\Lambda$ for every $n$. Thus its class is that of $s$
at every level. All these elements have outer valuation $0$.
We claim that no type over $K$ can satisfy
\eqref{eq:inconsistent-values}. Compose the outer valuation with the
$s$-adic valuation on $k$ to obtain a henselian valuation
 $w:K^\times\longrightarrow\RR\oplus_{\mathrm{lex}}\ZZ.$
Its least positive element is $\delta=w(s)=(0,1)$. For every $z\in K$,
\begin{equation}\label{eq:valuation-ring-square}
 P_2(1+sz^2)\quad\text{if and only if} \quad w(z)\ge0.
\end{equation} Indeed, if $w(z)\ge0$, then $w(sz^2)>0$ and Hensel's
lemma makes $1+sz^2$ a square. If $w(z)<0$, discreteness at the
least positive value gives $\delta+2w(z)<0$, which is not divisible by $2$ in $\mathbb{R}\oplus\ZZ$. Hence
so $1+sz^2$ cannot be a square. Since the valuation ring of $w$ is closed under squaring,
\eqref{eq:valuation-ring-square} implies
\begin{equation}\label{eq:square-implication}
 K\models\forall z\,
 \bigl(P_2(1+sz^2)\rightarrow P_2(1+sz^4)\bigr).
\end{equation}
But $s$ is not a square in $K$, and the $H_2$ part of
\eqref{eq:inconsistent-values} requires
$P_2(1+sT^2)$ and $\neg P_2(1+sT^4).$
This contradicts \eqref{eq:square-implication} in $K$, so no type over $K$ can realize it, proving non-surjectivity.
\end{proof}

\subsection{Functoriality revisited}

Within a fixed theory, a definable map $f:X\to Y$ induces the pushforward $f_*:\widehat X_T\to\widehat Y_T$. For a morphism of varieties this map is continuous, since evaluation at a regular function $g$ on $Y$ becomes evaluation at $g\circ f$ on $X$. A change of valued-field theory requires a separate argument: an inclusion of fields need not preserve Macintyre's power predicates. We record the comparison with the algebraic closure, which is the more natural case.

\begin{lemma}\label[lemma]{lem:maps}
Let $(K,v)$ be a $\bm$-field whose value group is divisible or a $\Zz$-group, let $T=\mathrm{Th}(K,v)$, and let $X$ be a variety over $K$. There is a canonical map
\[
\eta:\widehat X_T(K)\to \widehat X_{\ACVF}(K)
\]
which preserves all polynomial valuations. The map is continuous, and on an affine chart its evaluation map satisfies 
\[
\Theta_{K^\mathrm{alg}}\circ\eta=\tau_*\circ\Theta_K,
\]
where $\tau$ is the tropical shadow of \Cref{prop:T-shadow} and $K^{\mathrm{alg}}$ is an algebraic closure of $K$.
\end{lemma}

\begin{proof}
Let $p\in\widehat X_T(K)$ and $a\models p$ a realization in $\cU$. Its quantifier-free type over $\mathcal U$ is $K$-invariant, and $v(\mathcal U(a))=v(\mathcal U)$ by orthogonality. By \cite[Lemma~3.2]{orthogonal} applied with the geometric imaginaries over $K$, gives a unique extension of this type to $\mathcal U^{\mathrm{alg}}$ which is orthogonal to the value group and definable over those imaginaries. Note that the hypotheses on the value group place us in the setting of such a result. The geometric imaginaries in question are definable over $K$ in $\ACVF$, so the global definable extension is $K$-definable. In $\ACVF$ it is stably dominated. Therefore, we define $\eta(p)$ to be this type. Uniqueness makes the construction canonical.

For every regular function $f$ over $K$, by construction, we have that $v(f(\eta(p)))=v(f(p))$. Since power quotients become trivial over an algebraically closed residue field, the target's topology is generated, on Zariski charts, by the valuation evaluations. In particular, $\eta$ is continuous and the equality of evaluations is precisely the displayed identity with the tropical shadow.
\end{proof}

\subsection{Analytification of $p$-adic semi-algebraic sets}\label{sec:padic}

Model-theoretically, the construction of the real analytification $Y_r^\an$ in \cite{jell} of a semi-algebraic subset of $\RR^n$ rests on the following trick: the real closed field $\RR$ is an elementary substructure of $K=\Lau{\RR}{\RR}$ in the language of rings (by Tarski's quantifier elimination of real closed fields), so a semi-algebraic set $Y(\RR) \subseteq \RR^n$ determines a semi-algebraic subset $Y(K)\subseteq K^n$ defined by the same formula defining $Y$. By \Cref{cor:homeomorphism_theta}, in this situation the space $\widehat{Y}_\RCVF(K)$ is precisely the real analytification $Y_r^\an$.

The reader can easily note that the previous definition only relies on the following property of a field $K_0$: the fact that $K_0$ is an elementary substructure of $\Lau{K_0}{\RR}$ in the language of rings. This is an interesting class of fields which, in characteristic 0, corresponds to the class of divisibly $t$-henselian fields.\footnote{This class has been investigated by Susana Santoyo, a master's student codirected by Sylvy Anscombe and the first author.} For any such a field $K_0$ having in addition quantifier elimination in $\Lmac$ up to addition of constants, and any $\Lring$-definable subset $Y\subseteq K^n$, one can define its analytification as the corresponding space $\widehat{Y}_T(K)$ where $K=\Lau{K_0}{\RR}$ and $T=\Th_{\Lmac}(K)$. Here the valuation is corresponds to the $t$-adic valuation. In particular, the $p$-adic fields and the Laurent series fields $\CC(\!(x)\!)$ and $\RR(\!(x)\!)$ belong to this class. Since the $p$-adic valuation is $\Lring$-definable, our framework provides a definition for what the analytification of ``semi-algebraic subsets'' of $p$-adically closed fields should be.

\begin{definition}\label{def:p-adic-analytification}
Let $K_0$ be such that $K_0\prec \Lau{K_0}{\RR}=K$ and suppose $K$ is a $\bm$-field. Let $Y\subseteq K_0^n$ be an $\Lmac$-definable set. Let $Y(K)$ be the corresponding subset of $K^n$. We define the analytification of $Y$ as the set $\widehat{Y}_T(K)$ where $T=\Th_{\Lmac}(K)$ (the valuation being the $t$-adic valuation). In particular, taking $K_0=\QQ_p$ and $Y$ semi-algebraic, the $p$-adic analytification of $Y$ is $\widehat{Y}_T(K)$. 
\end{definition}

Over $\lau{\CC}{x}$ a single variety can carry two genuinely different analytifications, according to the valued-field structure one works with. Indeed, let $X$ be a variety over $\CC(\!(x)\!)$ and $Y\subseteq X$ be an $\Lring$-definable subset. On one hand, $\CC(\!(x)\!)$ is itself a $\bm$-field for the $x$-adic valuation, with residue field $\CC$ and value group $\ZZ$. Setting $T_0=\Th_{\Lmac}(\CC(\!(x)\!))$, gives the analytification $\widehat{Y}_{T_0}(\CC(\!(t)\!))$, a space over the value group $\ZZ$. Since the residue field is algebraically closed, $H(\lau{\CC}{x}) = \ZZ_\infty$ which is a totally disconnected space (it is Hausdorff and has a basis of clopen sets). Similarly, $\widehat{X}_{T_0}(\CC(\!(t)\!))$ (and therefore $\widehat{Y}_{T_0}(\CC(\!(t)\!))$) is totally disconnected and corresponds to the subset of $X^\an$ of semi-norms with integer radii. On the other hand, $\CC(\!(x)\!)$ is the residue field of $K=\Lau{\lau{\CC}{x}}{\RR}$, $K$ is also a $\bm$-field and $\CC(\!(x)\!)\preceq K$ in $\Lring$. It therefore falls under the construction in  \Cref{def:p-adic-analytification} (with $K_0=\CC(\!(t)\!)$), and viewing $Y$ over $K$ through the same $\Lring$-formula produces a second analytification $\widehat{Y}_{T}(K)$, with $T=\Th_{\Lmac}(K)$, a space over the value group $\RR$ and with residue field $\CC(\!(t)\!)$. This second analytification provides a very different topological space. Surprisingly, $\widehat{X}_{T}(K)$ (and hence $\widehat{Y}_{T}(K)$) is still totally disconnected, but in this case, what produces this behavior is the projective limit $\varprojlim_n \lau{\CC}{x}^\times/(\lau{\CC}{x}^\times)^n)$.

\section{Strict pro-definability and tropicalization}\label{subsec:strict-pro}

Hrushovski and Loeser showed in \cite{HruLoe} that their stable completion $\widehat X$ is strict pro-definable. We first prove the analogous result for $\widehat X_T$ using beautiful pairs \cite{beautiful}, and then give an explicit presentation using that result. We then study the restriction of the tropical maps in the model-theoretic context. 

\subsection{Strict pro-definbility}

In order to show strict pro-definability, we will need the following property on the residue field: 

\begin{enumerate}
\item[(B)] the class of all stably embedded pairs of $\mathrm{Th}(\RF)$, in the sense of \cite[Definition~2.2.1]{beautiful}, has the extension property, the amalgamation property and beauty transfer \cite[Section~2.4]{beautiful}.
\end{enumerate}

\begin{proposition}\label{prop:strict-prodefinability}
Let $(K,v)$ be a $\bm$-field whose residue field satisfies \textup{(B)}, and let $T=\Th(K,v)$. Then the class $\rF$ of global definable types orthogonal to $\Gamma$, allowing all small parameter sets of definition, is a natural class of definable types, the types in $\rF$ are uniformly definable, and for every definable set $X$ over $K$ the space $\widehat{X}_T$ is strict pro-definable. 

\end{proposition}

\begin{proof}
The class $\rF$ is natural in the sense of \cite[Definition~2.3.1]{beautiful} by \cite[Example~2.3.4(2)]{beautiful}, orthogonality to a sort being preserved under restriction to subtuples, automorphisms and definable push-forwards. The power predicates are ring-definable, and $\RV$ and its quotients are imaginary sorts. Passing to the corresponding $\RV$-presentation therefore preserves definability of types and orthogonality to $\Gamma$; being henselian of equicharacteristic $0$, this theory is benign in the sense of \cite[Definition~6.1.1]{beautiful} by \cite[Fact~6.1.2(1)]{beautiful}. The residue-side data of \cite[Section~8.3]{beautiful} consists of $\RF$ together with the auxiliary sorts $\RF^\times/(\RF^\times)^n$; by condition (1) of \Cref{def:class} these quotients are finite. Hence in any stably embedded pair, the sorts contribute trivially to pairs, and the residue-side data reduces to $\RF$ itself, as in \cite[Example~8.3.3(1)]{beautiful}.

Under the correspondence between natural classes of definable types and natural classes of stably embedded pairs \cite[Lemma~2.3.3]{beautiful}, $\rF$ corresponds to the class $\rK_\rF$ of pairs $(N,M)$ with $\mathrm{tp}(N/M)\in \rF(M)$. This class corresponds to the class induced, in the sense of \cite[Section~8.3]{beautiful}, by the full class of stably embedded pairs of $\mathrm{Th}(\RF)$ on the residue side and the class of trivial pairs of $\mathrm{Th}(\Gamma)$ on the value-group side.

Both side classes are natural \cite[Example~2.3.4(1)]{beautiful}; the residue-side class has the extension property, amalgamation and beauty transfer by (B), and the class of trivial pairs has all three for any theory, its associated types being the realized ones. Hence \cite[Theorem~8.3.1]{beautiful} applies: $\rK_\rF$-beautiful pairs exist and $\rK_\rF$ has beauty transfer. Uniform definability of the types in $\rF$ follows by \cite[Proposition~2.4.12]{beautiful}, and strict pro-definability of $\rF$ --- hence of $\widehat{X}_T=\rF_X$, the set $X$ being definable --- by \cite[Theorem~2.5.4]{beautiful}, the canonical parameters being imaginaries of $\mathcal U$.
\end{proof}

\begin{remark}\label[remark]{rmk:B-examples}
Hypothesis (B) holds when $\RF$ is algebraically closed, real closed, or $p$-adically closed; see \cite[Sections~3.1--3.2 and Theorem~7.2.3]{beautiful}. For iterated valued-field constructions as $\lau{\CC}{t}$ it can be checked using \cite[Theorem~8.3.1]{beautiful}. When one passes to a pure-field reduct, this application also requires the relevant valuation to be definable. We retain (B) as an explicit hypothesis wherever it is used.
\end{remark}

Using \Cref{prop:strict-prodefinability}, one can give a more explicit presentation of $\widehat X_T$ as a strict pro-definable set. Let $(K,v)$ be a $\bm$-field whose residue field satisfies \textup{(B)}, and fix a $K$-definable subset $X\subseteq\mathbb A^m$. We use polynomial families $P(x,y)\in\ZZ[x,y]$, specializing $y$ to tuples from $\mathcal U$. By \Cref{prop:strict-prodefinability}, for every formula $\varphi(x;y)$ there is a uniform formula $\psi(y;z)$ defining the $\varphi$-definitions of points $q\in\widehat X$. After quotienting by equality of definitions, these definitions have canonical imaginary parameters. Strict pro-definability ensures that the image of any finite tuple of such code maps is $K$-interpretable. This finite-image assertion is interpreted in the universal domain, as in \cite[Lemma~2.2.8]{HruLoe}.

\begin{proposition}\label[proposition]{prop:uniformizing-functions} Let $(K,v)$ be a $\bm$-field whose residue field satisfies \textup{(B)}. Let
$\mathcal{F}=\big(P_1(x,y_1),\ldots,P_k(x,y_k)\big),$
where $P_i(x,y_i)\in\ZZ[x,y_i]$, and let $n\geq 2$. Then there exist a $K$-interpretable set $S_{\mathcal{F}}^n=S_{\mathcal{F}}^n(X)$ (we omit $X$ from the notation when no confusion arises) and a definable map
$g_{\mathcal{F}}^n\colon
S_{\mathcal{F}}^n
\times \mathcal U^{|y_1|}
\times\cdots\times \mathcal U^{|y_k|}
\longrightarrow H_n^k$
such that, for every $q\in\widehat{X}$, there is a canonical parameter
$c_q^{\mathcal{F},n}\in S_{\mathcal{F}}^n$ satisfying
$g_{\mathcal{F}}^n
\bigl(c_q^{\mathcal{F},n};b_1,\ldots,b_k\bigr)=\big(
\rho_n(P_1(q,b_1)),\ldots,
\rho_n(P_k(q,b_k))
\big)$
for every
$
\bar{b}=(b_1,\ldots,b_k)
\in
\mathcal U^{|y_1|}\times\cdots\times \mathcal U^{|y_k|}.
$

For a nonzero value, $\rho_n(P_i(q,b_i))$ records its valuation and
its class in $\RV^\times/(\RV^\times)^n$; a zero value gives $\mathbf0\in H_n$.

\end{proposition}

\begin{proof}
Let $\bar y=(y_1,\ldots,y_k)$ and let $\bar u=(u_1,\ldots,u_k)$ range over $H_n^k$. Consider the formula
\[
\theta_{\mathcal{F},n}(x;\bar y,\bar u):=\bigwedge_{i=1}^k\rho_n(P_i(x,y_i))=u_i,
\]
where, if $u_i=(z_i,w_i)\neq\mathbf{0}$, the $i$-th conjunct reads $P_i(x,y_i)\neq0\wedge v(P_i(x,y_i))=z_i\wedge\rv_{n}(P_i(x,y_i))=w_i$, and the condition $u_i=\mathbf{0}$ reads $P_i(x,y_i)=0$. By uniform definability there is a formula $d\theta_{\mathcal{F},n}(\bar y,\bar u;c)$ such that every $q\in\widehat{X}$ admits a parameter $c_q^{\mathcal{F},n}$ with
\[
\theta_{\mathcal{F},n}(x;\bar b,\bar u)\in q\Longleftrightarrow d\theta_{\mathcal{F},n}\bigl(\bar b,\bar u;c_q^{\mathcal{F},n}\bigr)
\]
for every $\bar b$ and every $\bar u\in H_n^k$. Since $q$ is orthogonal to $\Gamma$, for every $\bar b$ there is exactly one $\bar u$ with $\theta_{\mathcal{F},n}(x;\bar b,\bar u)\in q$ (as in \Cref{prop:hyperfield-embedding}), so $d\theta_{\mathcal{F},n}(\,\cdot\,;c_q^{\mathcal{F},n})$ defines a total function $\mathcal U^{|y_1|}\times\cdots\times \mathcal U^{|y_k|}\to H_n^k$; replacing $c_q^{\mathcal{F},n}$ by a code of this function, we may assume it is a canonical parameter. Set $S_{\mathcal{F}}^n:=\{c_q^{\mathcal{F},n}:q\in\widehat{X}\}$, the image in the universal domain of this canonical-code map. It is $K$-interpretable by strict pro-definability. Let $g_{\mathcal{F}}^n(c;\bar b)$ be the unique $\bar u$ such that $d\theta_{\mathcal{F},n}(\bar b,\bar u;c)$ holds. The displayed identities hold by construction.
\end{proof}

\begin{definition}\label[definition]{def:parameter-maps}
Let $\mathcal{F}=\bigl(P_1(x,y_1),\ldots,P_k(x,y_k)\bigr)$ and let $n\geq2$. The \emph{parameter map at level $(\mathcal{F},n)$} is
\[
\mathbf{c}_{\mathcal{F}}^n\colon\widehat{X}\rightarrow S_{\mathcal{F}}^n,\qquad q\mapsto c_q^{\mathcal{F},n}.
\]
By \Cref{prop:uniformizing-functions}, for every $\bar b\in \mathcal U^{|y_1|}\times\cdots\times \mathcal U^{|y_k|}$ the composite $g_{\mathcal{F}}^n\bigl(\mathbf{c}_{\mathcal{F}}^n(\,\cdot\,);\bar b\bigr)$ sends $q$ to $\bigl(\rho_n(P_1(q,b_1)),\ldots,\rho_n(P_k(q,b_k))\bigr)$.
\end{definition}

\begin{lemma}\label[lemma]{lem:transition-maps-definable-trop}
If $m\mid n$, there is a unique definable surjection
$h_{mn}^{\mathcal{F}}\colon S_{\mathcal{F}}^n\to S_{\mathcal{F}}^m$
with $h_{mn}^{\mathcal{F}}(c_q^{\mathcal{F},n})=c_q^{\mathcal{F},m}$ for every $q\in\widehat{X}$. In addition, given $\cF=(f_1,\dots,f_\ell)$ and $\cF'=(f'_1,\dots,f'_r)$ and a map $\tau\colon\{1,\dots,\ell\}\to\{1,\dots,r\}$ satisfying $f'_{\tau(i)}=f_i$ for all $i$, then there is a unique definable surjection
$r_{\mathcal{F}'\mathcal{F}}^n\colon S_{\mathcal{F}'}^n\to S_{\mathcal{F}}^n$
with $r_{\mathcal{F}'\mathcal{F}}^n(c_q^{\mathcal{F}',n})=c_q^{\mathcal{F},n}$. These maps commute with one another.
\end{lemma}
\begin{proof}
The function coded by $c_q^{\mathcal{F},m}$ is obtained from the one coded by $c_q^{\mathcal{F},n}$ composed  with $\tau_{mn}^k$. Since the elements of each of the sets  $S_{\mathcal{F}}^{n}$ are canonical parameters of functions, they are unique up to interdefinability. Hence $\tau^{k}_{mn}(c_{q}^{\mathcal{F},n})$ and $c_{q}^{\mathcal{F},m}$ are interdefinable. Then there is a $\emptyset$-definable map $h_{mn}: S_{\mathcal{F}}^{n}\to S_{\mathcal{F}}^{m}$ that sends $c_{q}^{\mathcal{F},n}$ to $c_{q}^{\mathcal{F},m}$.\\
For $\cF$, $\cF'$ and $\tau$ as in the statement, let $\pi_{\tau}: H_n^{r} \to H_n^{\ell}$ be the map described in \cref{def:trop-XH}. Similarly, the function coded by $c_q^{\mathcal{F},n}$ is obtained from the one coded by $c_q^{\mathcal{F}',n}$  by composition with $\pi_{\tau}$. Since the elements of each of the set $S_{\mathcal{F}}^n$  are canonical parameters of functions, they are unique up to interdefinability and they determine unique definable maps,  then $\pi_{\tau} (c_{q}^{\mathcal{F}',n})$ and $c_{q}^{\mathcal{F},n}$ are interdefinable. Let $r^{n}_{\mathcal{F'}\mathcal{F}}$ be the definable map given by sending $c_{q}^{\mathcal{F'},n}$ to $c_{q}^{\mathcal{F},n}$.\\
By construction, both maps $h_{mn}^{\mathcal{F}}$ and $r^{n}_{\mathcal{F'},\mathcal{F}}$ are surjective. For instance, for $h_{mn}^{\mathcal{F}}$ given $c_{q}^{\mathcal{F},m}$ is the canonical parameter of some type $q \in \hat{X}(M)$. Then $h_{mn}(c_{q}^{\mathcal{F},n})=c_{q}^{\mathcal{F},m}$. And likewise for $r^{n}_{\mathcal{F}', \mathcal{F}}$.
\end{proof}

Let $\mathbb F$ be the set of finite sequences of polynomial families
over $\ZZ$, ordered with respect to the existence of maps $\tau$ as in \Cref{def:trop-XH}. 
Define $\mathbf S_X=\varprojlim_{(\mathcal F,n)}S_{\mathcal F}^n$
using the surjective transition maps of
\Cref{lem:transition-maps-definable-trop}. Their compatibility is the identity
$ h_{mn}^{\mathcal F}\circ r_{\mathcal F'\mathcal F}^{n}
 =r_{\mathcal F'\mathcal F}^{m}\circ h_{mn}^{\mathcal F'}.$
Write $\pi_{\mathcal F}^n$ for the projections. The compatible parameter
maps induce $\mathbf c:\widehat{X}_T\longrightarrow\mathbf S_X,$ defined as
 $q\longmapsto(c_q^{\mathcal F,n})_{\mathcal F,n}.$ By construction, $\mathbf S_X$ is a strict pro-definable set over $K$.

\begin{proposition}\label[proposition]{prop:explicit-presentation}
Let $X$ be an affine variety over $K$. Then
$\mathbf{c}\colon\widehat{X}_T\longrightarrow \mathbf{S}_X$
is a bijection. In particular, $\mathbf{c}$ presents $\widehat{X}_T$ as a strict pro-definable set over the countable indexing family $\mathbb{F}\times\Nn_{\geq2}$.
\end{proposition}

\begin{proof}
If $\mathbf c(q)=\mathbf c(q')$, all coded evaluation functions agree.
Every polynomial over $\cU$ is a specialization of a polynomial family
over $\ZZ$, so $\rho_n(f(q))=\rho_n(f(q'))$ for every $f$ and $n$.
By \Cref{cor:type-determination}, $q=q'$.

For surjectivity, let $s=(s_{\mathcal F}^n)\in\mathbf S_X$ and let
$A_{\mathcal F,n}=
 \{q\in\widehat X:c_q^{\mathcal F,n}=s_{\mathcal F}^n\}.$
Each set is nonempty by the definition of $S_{\mathcal F}^n$.
Any finite collection has nonempty intersection: choose a common upper
index and a type giving the code prescribed there, then use compatibility
with the transition maps. 
Compactness gives a global type $p$ belonging to every $A_{\mathcal F,n}$, whence $\mathbf c(p)=s$.  The type $p$ is $C$-definable for the small set of parameters 
\[
C=\{c_q^{\mathcal F,n} : c_q^{\mathcal F,n}=s_{\mathcal F}^n \text{ for some $(\mathcal{F}, n)\in \mathbb F\times\Nn_{\geq2}$}\}.
\]
Moreover it is orthogonal to $\Gamma$ by construction. Indeed, for any $f\in \cO_X=\cU[X]$, there is a polynomial $P$ over $\ZZ$ and a tuple $b$ such that $f(x)=P(x,b)$ on $X$. Therefore taking $\mathcal{F}=(P)$ and $n=2$ we have that the following formula belongs to $p$ for some $q\in \widehat{X}_T$
\[
\rho_2(f(x))= g_\mathcal{F}^2(c_q^{\mathcal{F},2},b).  
\]
 Since $q$ is orthogonal to $\Gamma$, the projection of $g_\mathcal{F}^2(c_q^{\mathcal{F},2},b)\in H_2$ to $\Gamma$ (the first coordinate) belongs to $\Gamma$, so $v(f(p))\in \Gamma$. 

The pro-definable structure of $\widehat{X}_T$ given by \Cref{prop:strict-prodefinability} associates a type $p$ with the tuple of its canonical parameters over $\Lmac^\Gamma$-formulas. Let $Cb(p)=(c_{p}^{\varphi})_{\varphi \in \Lmac^\Gamma}$ be the canonical base of $p$. Let $p^{qf}$ be the quantifier free type of $p$ in $\Lmac^\Gamma$ and note that $\mathbf{c}(p)$ is a canonical base of $p^{qf}$. We aim to show that that $Cb(p)$ and $\mathbf{c}(p)$ are interdefinable, so let $\sigma$ an automorphism of $\cU$ that fixes $Cb(p)$. Since $p \vdash p^{qf}$, then $\sigma(p^{qf})=p^{qf}$ consequently $\sigma(\mathbf{c}(p))=\mathbf{c}(p)$, thus $\mathbf{c}(p) \subseteq \dcl^{eq}(Cb(b))$. For the converse, suppose $\sigma$ fixes $\mathbf{c}(p)$, then $\sigma(p^{qf})=p^{qf}$ and by quantifier elimination $p^{qf}\vdash p$ so $\sigma(p)=p$. Therefore, $Cb(p) \subseteq \dcl^{eq}(\mathbf{c}(c))$. This shows the map $\mathbf{c}$ is pro-definable. Moreover, this shows that the construction is functorial over models (and, more generally, over $\eq{\acl}$-closed subsets of parameters). If the codes belong to a definably closed parameter set $C$, uniqueness makes the type $C$-invariant; its canonical
definition parameters then belong to $C$. This identifies the functorial
$C$-points as well.
\end{proof}

\begin{remark}\label[remark]{rmk:embedding-independence}
The presentation depends on the chosen closed embedding $X\subseteq\mathbb{A}^m$. A polynomial bijection $\varphi\colon X\to X'$ induces a bijection $\varphi_*\colon\widehat{X}_T\to\widehat{X'}_T$, and since each $P'(\varphi(x),y')$ is again of the form $Q(x,c)$ with $Q\in\ZZ[x,y]$ and $c$ a tuple from $K$, the $(\mathcal{F}',n)$-definition of $\varphi_*q$ is definably determined by finitely many definitions of $q$; hence the presentations attached to two embeddings are canonically isomorphic as pro-definable sets. The indexing family $\mathbb{F}$, on the other hand, depends only on $m$ and not on $K$.
\end{remark}

\begin{remark}\label[remark]{rmk:topology-from-presentation}
Although strict pro-definability is a purely model-theoretic statement, the topology of $\widehat{X}_T$ is recovered from this presentation: equip each $S_{\mathcal{F}}^n$ with the coarsest topology making the maps $c\mapsto g_{\mathcal{F}}^n(c;\bar b)$ continuous for all $\bar b$, and $\mathbf{S}_X$ with the inverse-limit topology. The topology pulled back to $\widehat{X}$ along the bijection $\mathbf{c}$ is then the coarsest one making all maps $q\mapsto\rho_n(f(q))$ continuous, with $f$ a polynomial over $\mathcal U$ and $n\geq2$; since the topology of $H$ is the coarsest topology of its projections to the $H_n$, this is precisely the topology induced by $\Theta$ (\Cref{def:TheTopology}), and $\mathbf{c}$ becomes a homeomorphism.
\end{remark}

\subsection{Tropicalization}

The tropicalizations of $\widehat{X}_T$ are simply the restriction, along the embedding $\Theta$ of \Cref{prop:hyperfield-embedding}, of the tropicalizations of $X(H)$ defined in \Cref{def:trop-XH}. For a finite sequence $\cF=(f_1,\dots,f_\ell)$ in $\ccO_X=\mathcal U[X]$ and $n\geq 2$ we denote by the same symbols the composites
\[
\trop_\cF^n:=\trop_\cF^n\circ\Theta\colon\widehat{X}_T\to H_n^\ell,\qquad
\trop_\cF:=\trop_\cF\circ\Theta\colon\widehat{X}_T\to H^\ell,
\]
so that, for any realization $a\models p$,
\[
\trop_\cF^n(p)=\big(\rho_n(f_1(p)),\dots,\rho_n(f_\ell(p))\big),\qquad
\trop_\cF(p)=\big(\rho(f_1(p)),\dots,\rho(f_\ell(p))\big).
\]
Order the finite sequences by inclusion
and form the inverse limit of the images of these maps. The resulting
map is
\[
\Trop:=\Trop\circ\Theta\colon\widehat{X}_T\longrightarrow
\varprojlim_\cF\trop_\cF(\widehat{X}_T),
\qquad p\longmapsto(\trop_\cF(p))_\cF.
\] By the definition of the topology on $\widehat{X}_T$ (\Cref{def:TheTopology}) these maps are continuous, being the restrictions of the continuous maps of \Cref{prop:trop-XH-cont}. Under hypothesis \textup{(B)}, each finite-level map $\trop_\cF^n$ is pro-definable and has definable image by strict pro-definability; $\trop_\cF$ is its pro-definable inverse limit.

\begin{remark}\label[remark]{rem:surjectivity} Let $\cF=(f_1,\ldots,f_\ell)\subseteq \ccO_X$, and $m,n\geq 2$ with $m|n$.
\begin{itemize}
    \item As a consequence of \cref{lem:surjectivity}, for each $\bar{z}\in \varprojlim_\cF \trop_\cF(\widehat{X}_T)$, each finite $\cF\subseteq \ccO_X$ and each $n\geq 2$, the component $\bar{z}_\cF^n$ lies in $\trop_\cF^n(\widehat{X}_T)$, so there exists $p\in \widehat{X}_T$ with $\trop_\cF^n(p)=\bar{z}_\cF^n$. 

\item The map $\tau_{mn}^\ell$ restricts to a surjection
$ \trop_\cF^n(\widehat X_T)\longrightarrow\trop_\cF^m(\widehat X_T).$
\item For $\cF\subseteq \cF'$, the projection $\pi\colon H_n^{|\cF'|}\to H_n^{|\cF|}$ restricts to a surjection $\trop_{\cF'}^n(\widehat{X}_T)\to \trop_\cF^n(\widehat{X}_T)$.
\end{itemize}
In particular, the inverse systems $(\trop_\cF^n(\widehat{X}_T))_n$ and $(\trop_\cF(\widehat{X}_T))_\cF$ are surjective.
\end{remark}

The following is a model-theoretic version of Payne's theorem for points over a specified model $K$. In the following two statements the residue field is $\RF(K)$, regular functions belong to $K[X]$, and the tropicalizations of $\widehat X_T(K)$ take values in $H_n(K)$ and $H(K)$, with their intrinsic topologies.

\begin{proposition}\label[proposition]{prop:Payne} Assume $(K,v)$ is maximally complete and all types over $\RF(K)$ are definable. Let $X$ be an affine variety over $K$. Then 
\[
\Trop^H: \widehat{X}_T(K) \longrightarrow \varprojlim_{\mathcal{F}} \trop_{\mathcal{F}}(\widehat{X}_T(K))
\]
is a homeomorphism, where $\mathcal{F}$ is the family of all finite sequences $(f_{1},\dots, f_{n})$ where $f_{i}\in K[X]$.
\end{proposition}

\begin{proof}
By \Cref{prop:Payne-XH} the map $\Trop^{H}\colon X(H(K))\to\varprojlim_\cF\trop_\cF^{H}(X(H(K)))$ is a homeomorphism. Since the topology on $\widehat{X}_T(K)$ is the one induced by $\Theta_K$ (\Cref{def:TheTopology}), $\Theta_K$ is a homeomorphism onto its image; hence $\Trop^{H}=\Trop^{H}\circ\Theta_K$ is a homeomorphism onto its image $\Trop^{H}(\widehat{X}_T(K))\subseteq\varprojlim_\cF\trop_\cF^{H}(\widehat{X}_T(K))$, the subspace topology there agreeing with the inverse-limit topology. It therefore remains only to show that $\Trop^{H}$ is surjective, i.e.\ that every thread of $\varprojlim_\cF\trop_\cF^{H}(\widehat{X}_T(K))$ is realized by a type.

Let $\bar z\in\varprojlim_\cF\trop_\cF^H(\widehat X_T(K))$, and write $z_{f,n}\in H_n(K)$ for its coordinate associated to $f$ and $n$. Consider the partial type
\[
q_{\bar z}=\{x\in X\}\ \cup\
\{\rho_n(f(x))=z_{f,n}:f\in K[X],\ n\geq2\}.
\]
The value $z_{f,n}=\mathbf{0}$ here means $f(x)=0$, so zero coordinates are included. Every finite subset involves finitely many functions $\cF$ and finitely many levels. Taking a common multiple $N$ of the levels, the definition of the inverse limit supplies $p_\cF\in\widehat X_T(K)$ with the prescribed $\cF$-coordinate, and hence with all the prescribed finite-level coordinates. Thus $q_{\bar z}$ is consistent.

Choose a complete extension $q\in S_X(K)$. Its polynomial valuations all lie in $\Gamma_\infty(K)$, so $\Gamma(K(a))=\Gamma(K)$ for $a\models q$, so $q$ is weakly orthogonal to $\Gamma$. By \Cref{cor:def-orthogonal} $q$ is definable and its global extension belongs to $\widehat X_T(K)$. By construction its tropicalizations give $\bar z$. Hence $\Trop^H$ is surjective. 
\end{proof}

We finish this section with a couple of density results. 

\begin{definition}
For a $K$-definable set $Y\subseteq X$, we let the map
\[
\iota\colon Y\to \widehat{Y}_T(K),\qquad \iota(a)=\tp(a/K),
\]
A point $\iota(a)\in \widehat{Y}_T(K)$ is called a \emph{$K$-simple point of $\widehat{Y}_T(K)$}. Given a finite family $\cF=(f_1,\dots,f_\ell)$ of regular functions ($n\geq 2$) on $\K[X]$, the tropicalizations of $Y$ correspond to the image of $S$ under $\trop_\cF^n$ and $\trop_\cF$ composed with $\iota$, that is, $\trop_\cF^n(S)=\trop_\cF^n(\iota(S))$. 
\end{definition}

Note that $\trop_\cF^n(S)\subseteq\trop_\cF^n(\widehat{S}_T)$, and likewise for $\trop_\cF$ and $\Trop$. This is the counterpart of Jell, Schreiderer and Yu's real tropicalization $\trop_{r,\cF}(Y)$ of \cite[\S3.1, \S5.2]{jell}, with $\rho_n\colon K\to H_n$ in the role of the signed seminorm $\trop_r\colon K\to\mathbb{RT}$.

Our first result shows that simple points are dense in $\widehat{Y}_T(K)$ exactly as in \cite[Lemma 3.6.1]{HruLoe}.

\begin{proposition}\label{prop:density-semialg}
The set $\iota(Y)$ is dense in $\widehat{Y}_T(K)$.
\end{proposition}
\begin{proof} Let $O$ be a non-empty basic open of $\widehat{Y}_T(K)$, taken inside the ambient space $\widehat{X}_T(K)$ with $X=\mathbb{A}^n$. By \Cref{rem:opens} we may assume $O=\widehat{Z}_T(K)$ for a definable subset $Z\subseteq Y$. Choose $p\in O$ and a realization $b\models p$ in $\cU$. If $Z$ is defined by a formula $\psi(x)$, $\cU\models \exists x \psi(x)$. Since $K\prec \cU$, we have that $K\models\exists x\psi(x)$. Therefore, there is $a\in K^n$ with $a\in Z$ so $\iota(a)=\tp(a/K)\in\widehat{Z}_T(K)=O$.
\end{proof}

With respect to the tropicalization of simple points we have a strong density result, generalizing \cite[Theorem~6.2]{jell}. 

\begin{proposition}\label{prop:strong-density}
Let $Y,X, S$, $\cF=(f_1,\dots,f_\ell)$, and $n\geq 2$ as above. Then
\[
\trop_\cF^n(Y)=\trop_\cF^n(\widehat{Y}_T)\qquad\text{in } H_n^\ell .
\]
\end{proposition}
\begin{proof}
The inclusion $\trop_\cF^n(Y)\subseteq\trop_\cF^n(\widehat{Y}_T)$ follows by definition since $\trop_\cF^n(Y)=\trop_\cF^n(\iota(Y))$ and $\iota(Y)\subseteq\widehat{Y}_T$. For the reverse inclusion, let $p\in\widehat{Y}_T$ and put $z=\trop_\cF^n(p)=(\rho_n(f_1(p)),\dots,\rho_n(f_\ell(p)))\in H_n^\ell$. We produce $a\in Y$ with $\trop_\cF^n(a)=z$.

Fix $1\le j\le\ell$. Since $K$ is a $\bm$-field, $[K^\times:(K^\times)^n]<\infty$, so $K^\times$ is partitioned into finitely many cosets of $P_n^*=(K^\times)^n$; as $p$ is a complete type, it decides which one $f_j(x)$ meets, i.e.\ there is a representative $d_j\in K^\times$ with $\big(f_j(x)\in d_jP_n^*\big)\in p$, unless $\big(f_j(x)=0\big)\in p$. Moreover, $p$ is orthogonal to $\Gamma$, so the pushforward $(v\circ f_j)_*p$ is realized at a point $\gamma_j\in\Gamma_K\cup\{\infty\}$; that is, $\big(v(f_j(x))=\gamma_j\big)\in p$, with $\gamma_j=v(b_j)$ for some $b_j\in K^\times$ when $f_j(p)\neq0$. Put
\[
\theta_j(x):=\begin{cases}
\big(f_j(x)\in d_jP_n^*\big)\wedge\big(v(f_j(x))=v(b_j)\big), & f_j(p)\neq0,\\[2pt]
f_j(x)=0, & f_j(p)=0,
\end{cases}
\]
an $\Lmac(K)$-formula; by the above, $\theta_j\in p$.

We claim $\theta_j(a)$ forces $\rho_n(f_j(a))=\rho_n(f_j(p))$ for every $a$. If $f_j(p)=0$ this is clear, both sides being $\mathbf{0}$. Otherwise, suppose $\theta_j(a)$ holds and write $f_j(a)=d_je^n$ with $e\in K^\times$ and $v(f_j(a))=v(b_j)=\gamma_j$. Using multiplicativity of $\rho_n$ and $[e^n]_n=[e]_n^n=1$ in $G_n^\times=\RF^\times/(\RF^\times)^n$,
\[
\rho_n(f_j(a))=\rho_n(d_j)\,\rho_n(e)^n=(v(d_j),[d_j]_n)\cdot(nv(e),1)=(v(d_j)+nv(e),[d_j]_n)=(\gamma_j,[d_j]_n).
\]
A realization of $p$ also satisfies $\theta_j$, so the same computation gives $\rho_n(f_j(p))=(\gamma_j,[d_j]_n)$. Hence $\rho_n(f_j(a))=\rho_n(f_j(p))$, proving the claim.

Now, if $\psi(x)$ defines $Y$, set $\theta(x):=\psi(x)\wedge\bigwedge_{j=1}^\ell\theta_j(x)$, an $\Lmac$-formula over $K$ with $\theta\in p$. Choosing a realization $b\models p$ in $\cU$ we have that $\cU\models\exists x\,\theta(x)$, hence $K\models\exists x\,\theta(x)$. Any witness $a\in X(K)$ lies in $Y$ and satisfies every $\theta_j$, so $\trop_\cF^n(a)=(\rho_n(f_j(a)))_j=(\rho_n(f_j(p)))_j=z$. Therefore $z\in\trop_\cF^n(Y)$.
\end{proof}

\section{Definable compactness}\label{sec:defcompact}

Throughout this section, $X$ is an affine variety over a small model
$K\prec\cU$, the value group $\Gamma$ is divisible or a $\ZZ$-group, and
hypothesis~\textup{(B)} holds. As in Section~\ref{sec:analytification},
definable and pro-definable sets are evaluated in the universal domain
$\cU$. Unless otherwise specified, types are global and definable, with
arbitrary small sets of definition parameters. Regular functions and
neighbourhoods may have parameters from $\cU$. We use $\RF$ for the 
residue field (as a definable set). The notation $\mathrm{DOAG}$ denotes the theory of
nontrivial divisible ordered abelian groups, and $\mathrm{PRES}$ denotes
Presburger arithmetic.

Our aim in this section is to prove a definable Heine--Borel theorem for $\widehat X$, that is, to show that every pro-definable subset $Y$ of $\widehat{X}$ is definably compact if and only if it is closed and bounded. Here we use the definition of definable compactness due to Hrushovski and Loeser \cite{HruLoe} which we will soon recall. The right-to-left direction assumes in addition that the residue field is large---an assumption which holds for all the examples of interest. The converse direction is more subtle and we show it only for algebraically closed, real
closed, and $p$-adically closed residue fields. For those cases, we establish an analogue of curve selection in the spirit of~\cite[Proposition 4.2.13]{HruLoe}. It intends to replace the curve selection in $o$-minimal expansions of groups. More explicitly, at each point $a$ in the closure of a pro-definable set $Y$ in a pro-definable topological space, it asserts the existence of a
global definable type $p$ on $Y$ converging to the point $a$. The real closed and $p$-adic cases
use directed refinements of definable types, while the algebraically closed case is treated by stable domination.

\begin{notation} In this section we will work with several topological spaces. Given $Z$ a topological space and $Y \subseteq Z$ we write $\mathrm{cl}(Y)$ to denote its topological closure within $Z$. 
\end{notation}
\subsection{Limits, boundedness and generalities}

\subsubsection{Lifting definable types}\label{subsec:lifting-new}
For a surjective map $f:D\to E$, a \emph{definable lift} of a type
$q$ on $E$ is a definable type $p$ on $D$ satisfying $f_*p=q$.
The assertion that every definable type admits such a lift is called
\emph{surjectivity transfer}. The results below concern the valued-field
theories fixed in this section; they are not assertions about arbitrary
theories.
\begin{proposition}\label[proposition]{prop:lifting-definable}
For every definable set $D$, the space $S^{\mathrm{def}}(D)$ is strict
pro-definable. If $f:D\to E$ is a surjective definable map, then
$f_*:S^{\mathrm{def}}(D)\to S^{\mathrm{def}}(E)$ is surjective.
\end{proposition}
\begin{proof}
As in \Cref{prop:strict-prodefinability}, the residue-side data consists
of $\RF$ and its finite power quotients, and hence reduces to $\RF$.
Its class of stably embedded pairs has the extension property and beauty
transfer by \textup{(B)}. The class of all stably embedded pairs of
$\Gamma$ has the same properties by
\cite[Theorem~3.2.1, Corollary~3.2.2]{beautiful} for $\mathrm{DOAG}$ and
\cite[Theorem~4.2.2]{beautiful} for $\mathrm{PRES}$.
The transfer results \cite[Theorems~5.2.5 and~8.1.2,
Example~8.3.3(1)]{beautiful} give the corresponding properties for the
valued-field theory. The extension property yields surjectivity transfer,
and beauty transfer gives strict pro-definability
\cite[\S2.5, Theorem~2.5.4]{beautiful}.
\end{proof}

\begin{proposition}\label{prop:lifting-prodefinable}
For a small inverse system of definable sets,
$S^{\mathrm{def}}\!\left(\varprojlim_iD_i\right)
=\varprojlim_i S^{\mathrm{def}}(D_i).$
If $f:D\to E$ is a surjective pro-definable map, every global definable
type $q$ on $E$ has a global definable lift to $D$. Extra small parameters
of definition are allowed.
\end{proposition}
\begin{proof}

Every formula involves finitely many coordinates. A compatible family
of complete definable types on the stages therefore determines a
complete definable type on the inverse limit. The union of their
definition parameters is small.

For the lifting assertion, choose compatible presentations
$D=\varprojlim_iD_i$, $E=\varprojlim_iE_i$, and
$f=\varprojlim_i f_i$, with transition maps $\lambda_{ji}$ and
$\rho_{ji}$. Write $q_i=(\pi_i^E)_*q$.
The stage maps $f_i$ need not be surjective. Nevertheless,
surjectivity of $f$ gives
$ \pi_i^E(E)\subseteq f_i(D_i).$
Since the right-hand side is definable, $q_i$ concentrates on it.
By \Cref{prop:lifting-definable}, the pro-definable set $L_i=\{p_i\in S^{\mathrm{def}}(D_i):(f_i)_*p_i=q_i\}$
is nonempty. The conditions $p_i\in L_i$ and
$(\lambda_{ji})_*p_j=p_i$ are finitely consistent: lift at a stage
dominating the finitely many indices and push the lift down.
Compactness gives a
compatible family $(p_i)_i$. The corresponding type $p$ satisfies
$f_*p=q$.
\end{proof}

\subsubsection{Limits, boundedness and definable compactness}

We use Hrushovski and Loeser's notion of definable compactness
\cite[Section~4.1]{HruLoe}. All topological spaces considered below
have a basis of relatively definable open sets.

\begin{definition}\label[definition]{def:limit-def-compact}
Let $Z$ be a definable or pro-definable topological space, and let $p$ be a global definable type concentrating on $Z$. A point $a\in Z$ is a \emph{limit} of $p$, written $a=\lim p$, if every relatively definable open neighbourhood of $a$, with parameters from $\cU$, belongs to $p$. We say $Z$ is \emph{definably compact} if every such global definable type has a limit in $Z$. The parameters defining $p$ need not lie in the parameter set defining $Z$.
\end{definition}

\begin{definition}\label[definition]{def:bounded-sets}
\leavevmode
\begin{enumerate}
\item Let $A$ be a set, and  $Y\subseteq\Gamma_\infty^A$ is \emph{bounded} if $Y\subseteq[\beta,\infty]^A$ for some $\beta\in\Gamma$.
\item $Y\subseteq H_n^\ell$ is \emph{bounded} if $\bar v^{\,\ell}(Y)$ is bounded in $\Gamma_\infty^\ell$; for $\ell=1$ this agrees with \Cref{def:bounded-Hn}.
\item For a fixed closed affine embedding $X\subseteq\mathbb A^m$, a set $D\subseteq X$ is \emph{bounded} if $v(x_i)\geq\beta$ for all $x\in D$ and $i\leq m$, for some $\beta\in\Gamma$. This is independent of the closed affine embedding, since polynomial functions on a bounded set have valuations bounded below.
\item $Z\subseteq\widehat X$ is \emph{bounded} if $Z\subseteq\widehat D$ for some bounded definable $D\subseteq X$.
\end{enumerate}
\end{definition}

\begin{lemma}\label{lem:defcompact-generalities}
Let $Z,Z'$ be definable or pro-definable topological spaces.
\begin{enumerate}
\item If a type concentrates on a closed pro-definable subset $C\subseteq Z$
and has a limit in $Z$, then that limit belongs to $C$.
\item If $h:Z\to Z'$ is continuous and pro-definable and $\lim p=a$,
then $\lim h_*p=h(a)$.
 \item If $Z=Z_1\cup\dots\cup Z_k$ with each $Z_i$ closed and pro-$K$-definable, then $Z$ is definably compact if and only if each $Z_i$ is.
\end{enumerate}
\end{lemma}
\begin{proof}
For (1), a limit outside $C$ would have a relatively definable open neighbourhood disjoint from $C$. For (2), pull back neighbourhoods under $h$. For $(3)$ assume first that $Z$ is definably compact. Let $p$ be a definable type concentrating in one of the $Z_{i}$, since $Z$ is definably compact its limit $\mathrm{lim}(p) \in Z$. As $Z_{i}$ is closed, by $(1)$, $\mathrm{lim}(p) \in Z_{i}$. Thus each $Z_{i}$ is definably compact. For the converse, let $p$ be a definable type concentrating on $Z$. Since the type is complete there is some $i$ such that $p \vdash x \in Z_{i}$, since the later is definably compact $\mathrm{lim}(p) \in Z_{i} \subseteq Z$.
\end{proof}

We finish this section by recalling a construction for definable types that does not depend
on the valued-field theory.

\begin{definition}\label{def: integral} Let $R$ be pro-definable, let $r$ be a
global definable type on $R$, and let $\mathscr H(w)$ be a pro-definable
family of global definable types on $D$, indexed by $w\in R$.
Here a pro-definable family means that, for every formula
$\varphi(x;y)$, a relatively definable relation $\chi_\varphi(w;y)$
expresses
$\varphi(x;e)\in\mathscr H(w)
\ \text{if and only if}\ \chi_\varphi(w;e).$\\
The \emph{integral} $\int_r\mathscr H$ is defined by
\[
 \varphi(x;e)\in\int_r\mathscr H
 \quad\text{if and only if}\quad
 \chi_\varphi(w;e)\in r.
\]

\end{definition}
\begin{remark}\label{rmk:propertiesintegral} The type defined in \cref{def: integral} is a complete consistent type. Completeness follows because these relations respect negation.
For consistency, finitely many conditions from $r$ can be realized by
one index $w$, and the corresponding formulas are jointly consistent
in $\mathscr H(w)$. Applying the definitions of $r$ to
$\chi_\varphi$ proves definability over the combined definition
parameters. Equivalently, over a small model $N$ containing those
parameters, choose $w^*\models r\rest_N$ and then
$d^*\models\mathscr H(w^*)\rest_{Nw^*}$. The integral restricts to
$\tp(d^*/N)$.
\end{remark}
We will use this construction  of \cref{def: integral} first for the family
assigning to each point of $\widehat X$ its represented field type.

\subsection{The value group and the hyperfields \texorpdfstring{$H_n$}{Hn}}
\begin{remark}\label{rmk:noteveryvaluegroupworks} Some restrictions on the value group are necessary to obtain a Heine-Borel-characterization.
We refer the reader to \cite{poly} for an introduction to poly-regular groups. In \cite[Proposition $5.1$]{dpminimal}, it is shown that for $\Gamma$ a pure ordered abelian group it is $dp$-minimal if and only if $[\Gamma:n\Gamma]< \infty$ for all $n \in \mathbb{N}$.
By quantifier elimination for poly-regular groups, any type is fully determined by a cut along with the cosets of $n\Gamma+ \Delta_{i}$ where $\Delta_{i}$ are the finitely many possible convex subgroups of $\Gamma$.\\
Let $\Gamma(K)$ be a dense, regular, dp-minimal and  non-divisible ordered abelian group. Let $\beta \in \Gamma(K)$ be such that $\beta> 0$ is not $n$-divisible. Consider the partial type, $q(x)=\{ x > \gamma \ | \ \Gamma(K) \models n\gamma < \beta\}$. By quantifier elimination and dp-minimality any completion of such type is definable. This type concentrates in the bounded and closed set $[0, n\beta]$ but it does not have a limit since its limit in the divisible hull is $\frac{\beta}{n} \notin \Gamma(K)$. Likewise, if $\Gamma$ is poly-regular then it has an $\emptyset$-definable convex subgroup $\Delta$. Any completion of the type $q(x)$ stating that $x \in \Delta$ and $x> \beta$ for all $\beta \in \Delta(\Gamma(K))$ is definable. However, it does not have a limit in $\Gamma(K)$, even though it concentrates in the definable closed subset $[0,\alpha]$ (where $\alpha \in \Gamma(K)$ and $\alpha+\Delta> \Delta$. )
\end{remark}

In view of \cref{rmk:noteveryvaluegroupworks}, we will assume that $\Gamma$ is either a divisible group or a $\ZZ$-group. In the latter case, the topology on $\Gamma$ is discrete and and the intervals $(h,\infty]$ form a neighborhood basis at $\infty$.

\begin{notation}\label[notation]{not:gamma-types}
We use the following global definable types on $\Gamma$.
Their completeness and definability follow from quantifier elimination:
a one-type is determined by its cut; in the $\ZZ$-group case, we also need its
congruence classes modulo each positive integer.
\begin{itemize}
\item $\tau_{\infty}$ (\emph{at $+\infty$}): the conditions $\{x>\gamma : \gamma\in\Gamma\}$, together with $\{x\equiv 0\bmod m : m\geq2\}$ in the $\ZZ$-group case; and symmetrically $\tau_{-\infty}$ (\emph{at $-\infty$}), with $x<\gamma$ throughout.
\item In the divisible case, for $\gamma\in\Gamma$ the two cut-types $\tau_{\gamma^{+}}$ and $\tau_{\gamma^{-}}$, concentrating just above, resp.\ just below, $\gamma$.
\end{itemize}
\end{notation}

Quantifier elimination also gives the following description.

\begin{fact}\label[fact]{fact:one-types-gamma}
Every definable one-type $q$ on $\Gamma$ is either realized in $\Gamma$, has cut at $+\infty$ or $-\infty$, or, in the divisible case, is of the form $\tau_{\gamma^{+}}$ or $\tau_{\gamma^{-}}$ for a unique $\gamma\in\Gamma$. In the $\ZZ$-group case, types at either infinite cut may have different compatible congruence classes. Every definable one-type on $\Gamma_\infty$ concentrating on a bounded subset of $\Gamma_\infty$ has a unique limit in $\Gamma_\infty$.
\end{fact}

\begin{definition}Let $A$ be a set and $b\in\Gamma_\infty^A(K)$. Let $A_\infty$ be the coordinates of $b$ equal to $\infty$, and consider the standard neighborhood basis $\{W_t\}_{t\in I}$ of $b$, where: in the $\ZZ$-group case $I=\Gamma$ and $W_t$ has $(t,\infty]$ at the $\infty$-coordinates and $\{b_a\}$ elsewhere. In the divisible case take $I=\Gamma\times\Gamma_{>0}$ and $W_{(\gamma,\epsilon)}$ has $(\gamma,\infty]$  at the $\infty$-coordinates, and $(b_a-\epsilon,b_a+\epsilon)$ elsewhere. %Let $\nu_b$ be the $K$-definable type on $I$ obtained by pushing $\tau_\infty$ (resp.\ $\tau_\infty\otimes\tau_{0^+}$) forward under the diagonal map $\gamma\mapsto(\gamma,\ldots,\gamma)$.
Let $\nu_b=\tau_\infty$ in the $\ZZ$-group case and $\nu_b=\tau_\infty\otimes\tau_{0^+}$ in the divisible case.
\end{definition}

\begin{remark}\label{rem:cofinal}
For a definable family of neighborhoods $(W_t)_{t\in I}$, a global
definable type $\nu$ on $I$ is called \emph{cofinal} if it contains
$\{t\in I:W_t\subseteq W_{t_0}\}$ for every fixed $t_0\in I$. This expresses that the neighborhoods shrink along the index type. Note that $\nu_b$ as defined above is cofinal. Moreover, every relatively definable open neighborhood $U$ of $b$ contains $W_{t_0}$ for some $t_0\in I$. Indeed, a formula defining $U$ only involves a finite set $A_0\subseteq A$ of coordinates, so $U=\pi_{A_0}^{-1}(U_0)$ with $U_0\subseteq\Gamma_\infty^{A_0}$ definable. Hence $U_0$ contains a box $\prod_{a\in A_0} O_a$, where $O_a=(\gamma_0,\infty]$ if $b_a=\infty$, and otherwise $O_a=(b_a-\varepsilon_0,b_a+\varepsilon_0)$ in the divisible case and $O_a=\{b_a\}$ in the $\ZZ$-group case, for some $\gamma_0\in\Gamma$ and $\varepsilon_0\in\Gamma_{>0}$ (take the maximum, resp.\ minimum, of the finitely many bounds). Then $W_{t_0}\subseteq\pi_{A_0}^{-1}(\prod_{a\in A_0} O_a)\subseteq U$ for $t_0=(\gamma_0,\varepsilon_0)$, resp.\ $t_0=\gamma_0$.    
\end{remark}

\begin{lemma}\label{lem:cofinal-gamma-family}
Let $V\subseteq\Gamma_\infty^A$ be a pro-definable set and $b\in\Gamma_\infty^A$. Then, there is a definable neighbourhood basis
$\{W_t:t\in I\}$ and a global definable type $\nu_b$ on $I$ such that $\{t:W_t\subseteq W_{t_0}\}\in\nu_b$ for all $t_0\in I$. Moreover, if $b\in\operatorname{cl}(V)$ (for the product topology), then there is a global definable type $q$ on $V$ with $\lim q=b$. 
\end{lemma}
\begin{proof}
Consider the pro-definable set 
\[
R=\{(u,t)\in V\times I : u_a\in (W_t)_a \text{ for every } a\in A\},
\]
where $(W_t)_a$ denotes the $a$-th factor of $W_t$. 

First, note that the projection $\pi_I\colon R\to I$ is surjective. Indeed, fix $t\in I$ and consider the partial type $p(u)=V(u)\cup\{u_a\in(W_t)_a : a\in A\}$, where we abuse of notation and write $V(u)$ for the (small) type expressing that $u\in V$. Since $A$ is small, $p$ is a type over a small set of parameters. A finite part of it involves finitely many formulas of $V$ and a finite set $A_0\subseteq A$ of coordinates. The set $\pi_{A_0}^{-1}(\prod_{a\in A_0}(W_t)_a)$ is an open neighborhood of $b$. As $b\in\operatorname{cl}(V)$, this neighbourhood meets $V$, so the finite part is satisfiable. By saturation of $\cU$ the partial type is realized by some $c\in V$, and $(c,t)\in R$.

Since $\pi_I$ is surjective (and pro-definable), by \Cref{prop:lifting-prodefinable}, $\nu_b$ has a global definable lift $r$ concentrating on $R$. Let $q$ be the pushforward of $r$ via $\pi_V$ (the projection to $V$), which is a global definable type concentrating on $V$. We show that $\lim q=b$. Let $c$ be a realization of $q$. Without loss, suppose $(c,s)$ is a realization of $r$, so that $s$ realizes $\nu_b$. Let $U$ be a relatively definable open neighborhood of $b$. By \Cref{rem:cofinal}, let $t_0\in I$ be such that $W_{t_0}\subseteq U$. Since $\nu_b$ is cofinal, it concentrates on the set $\{t\in I: W_t\subseteq W_{t_0}\}$. Therefore, $W_s\subseteq W_{t_0}$. In particular, since $(c,s)\in R$, we have that $c_a\in (W_s)_a$ for every $a\in A$. Hence $c\in W_{t_0}\subseteq U$, completing the proof.
\end{proof}

\begin{lemma}\label{lem:bounded-target-limits}
Every global definable type on a bounded pro-definable subset $Y$ of $\Gamma_\infty^A$
or $H_n^A$ has a unique limit in $\Gamma_\infty^A$
or $H_n^A$ respectively.
\end{lemma}
\begin{proof}
For $\Gamma_\infty^A$, each coordinate type has a limit by
\Cref{fact:one-types-gamma}, and finite intersections of coordinate
neighbourhoods show convergence to the tuple of these limits.
For $H_n^\ell$, let $p$ the definable type on $Y$. For each $a\in A$, the $a$-th coordinate pushforward of $p$, $p_a$ concentrates on one member
$H_{n,\mathbf g}$ of the finite closed cover in
\Cref{rmk:topologyincoordinates}. Pulling $p_a$ back under the definable homeomorphism $\theta_{g_a}\colon\Gamma_\infty\to H_{n,g_a}$ gives a definable type concentrating on a bounded subset of $\Gamma_\infty$, to which the value group case apply. So we get a limit $c_a$ of $p_a$. By the way the topology is defined, it is easy to check that a set in the basis of the topology can only refer to finitely many coordinates on $A$, so 
$p$ converges to $(c_a)_{a\in A}$.  Moreover, these homeomorphisms preserve boundedness. Uniqueness follows from \Cref{lem:defcompact-generalities}(3).
\end{proof}

The following is the Heine-Borel characterization for the value group. 

\begin{proposition}\label{prop:heine-borel-gamma}
Let $Y\subseteq\Gamma_\infty^A$ be a pro-definable set. Then $Y$ is definably compact if and only if it is closed and bounded. 
\end{proposition}
\begin{proof}
Assume that $Y$ is bounded. By \cref{lem:bounded-target-limits}, given $q$ a global definable type centered on $Y$ has a unique limit $\bar{a}$ in $\Gamma_{\infty}^{A}$. If $Y$ is closed within $\Gamma_{\infty}^{A}$, then $\mathrm{lim(q)}=\bar{a} \in Y$. We conclude that $Y$ is definably compact.

For the converse, assume that $Y$ is definably compact.
 We first show that $Y$ must be closed. This is by \Cref{lem:cofinal-gamma-family}. Thus it remains to show that $Y$ is bounded. If the $a$-th coordinate projection is unbounded, then we consider the following $P_a\subseteq Y\times \Gamma$, $P_a=\{(p,\gamma):p\in Y, p_a<\gamma\}$. Unboundedness of the $a$-th coordinate projection implies that for any $\gamma$, there is $(p,\gamma)\in P_a$. So we have a pro-definable surjection, $P_a\to \Gamma$. By \Cref{prop:lifting-prodefinable}, we have that any type concentrating at $-\infty$ in $\Gamma$ has a lift to $P_a$, and hence to $Y$ by taking the first coordinate projection. Note that by choice, the type has no limit on its $a$-th coordinate projection, hence the lift will not have a limit in $Y$.
\end{proof}

The previous result induces a Heine-Borel characterization for the finite powers of the hyperfield $H_{n}$. 
\begin{proposition}\label{prop:heine-borel-Hn}
Let $Y\subseteq H_n^A$ be a pro-definable subset (in $T^{\mathrm{eq}}$). Then $Y$ is definably compact if and only if it is closed and bounded. 
\end{proposition}

\begin{proof}
By the analogue for infinite products of \Cref{rmk:topologyincoordinates}, $H_n^A=\bigcup_{\mathbf g}H_{n,\mathbf g}$ where $\mathbf{g}\in (G_n^\times)^A$, is a closed (small) cover and each $\theta_{\mathbf g}\colon\Gamma_\infty^A\to H_{n,\mathbf g}$ is a definable homeomorphism which, by construction, sends a bounded set to bounded set. The statements now follow from \Cref{prop:heine-borel-gamma} through \Cref{lem:defcompact-generalities}(1)--(3): a complete type on $Y$ concentrates on some $Y\cap H_{n,\mathbf g}$, and boundedness, limits, and definable compactness can all be checked on the pieces. It remains to check if every $Y\cap H_{n,\mathbf g}$ is closed, then $Y$ is closed. Let $b\in \mathrm{cl}(Y)$ and for each finite $F\subseteq A$ define 
\[ 
D_F=\left\{\eta\in(G_n^\times)^F: b\in\operatorname{cl}\left( Y\cap\bigcap_{a\in F}\pi_a^{-1}(H_{n,\eta_a}) \right)\right\}. 
\]
By the fact that the copies of $\Gamma$ corresponding to elements in $D_F$ cover $Y$, $D_F$ is nonempty. Moreover, the transition maps respect the $D_F$'s. Saturation gives $\mathbf g\in\varprojlim_F D_F$. Let $U$ be a neighbourhood of $b$, by construction, the following condition on $y$ is finitely satisfiable: $y\in Y\cap U, y_a\in H_{n,g_a}, (a\in A)$. This means that $U\cap Y\cap H_{n,\mathbf g}\neq \emptyset$ for our choice of $\mathbf{g}$ above. Thus the closedness of $Y\cap H_{n,\mathbf{g}}$ imples $b\in Y$, verifying the closedness of $Y$.

\end{proof}

\begin{lemma}
\label{lem:gamma-curve-selection}
Suppose $\Gamma\models\mathrm{PRES}$. Let $I$ be small, let
$V\subseteq\Gamma_\infty^I$ be pro-definable in the pure value-group
structure, and let $b\in\operatorname{cl}(V)$ for the product topology.
There is a global definable type $q$ on $V$ with $\lim q=b$. The type $q$ is  definable
in the stably embedded structure $\Gamma_{\infty}$, and its canonical base might require us to add a small set of parameters over the value group.
\end{lemma}
\begin{proof}
Let $I_{\mathrm{fin}}=\{i:b_i\ne\infty\}$ and
$I_\infty=\{i:b_i=\infty\}$. Consider
\[
 R=\{(u,h)\in V\times\Gamma \ |  \
       u_i=b_i\ (i\in I_{\mathrm{fin}}),\text{and} \ 
       u_i>h\ (i\in I_\infty)\}.
\]

The projection $R\to\Gamma$ is surjective. Indeed, fix $h\in\Gamma$.
Every finite collection of the coordinate conditions defines an open
neighbourhood of $b$: finite points of $\Gamma_\infty$ are isolated and
$(h,\infty]$ is a neighbourhood of $\infty$. Since $b\in\operatorname{cl}(V)$,
each such collection has a point in $V$. Together with a small presentation of
$V$, the conditions are finitely satisfiable, and saturation supplies
$u\in V$ satisfying all of them.

Lift $\tau_\infty$ along this projection using \Cref{prop:lifting-prodefinable}, and let $q$ be the
pushforward of the lift to $V$. It contains $u_i=b_i$ for
$i\in I_{\mathrm{fin}}$. For $i\in I_\infty$ and every
$\delta\in\Gamma$, the lifted type contains $u_i>h>\delta$, so
$u_i>\delta\in q$. Thus each coordinate converges to $b_i$, and finite
intersections give convergence in the product topology.
\end{proof}

\subsection{Limits on bounded subsets. Closed and bounded implies definably compact. }\label{subsec:standard-part-new}

We adapt the standard-part construction of
\cite[Lemma~4.2.2]{HruLoe}, retaining both valuation and power-class
information. The residue field will be assumed \emph{large}: every
smooth geometrically integral curve with a rational point has infinitely
many rational points. Equivalently, rational points on any smooth
geometrically integral variety with a rational point are Zariski dense;
see \cite{pop-large}. Algebraically closed, real closed, and $p$-adically
closed fields are large. Note that largeness is indeed a topological notion as shown in~\cite{firstpaper}, which is closely related to the topology generated by henselian local domains~\cite{henselianlarge,field-top-1,field-top-2,LargeHenselian}. The fact that largeness of the residue field is required here suggests a reasonable notion of definable compactness in the setting of fraction fields of henselian domains.

We first describe the standard part on value groups and then the
associated residue-field construction. In this subsection, ``bounded''
for an element of an extension of $\Gamma(M)$ means bounded in absolute
value by an element of $\Gamma(M)$.

\begin{definition}\label[definition]{def:standard-part}
Let $M\prec\cU$ be small, and let $F$ be a henselian valued field extension of $M$ such that every element of $\Gamma(F)$ has $\Gamma(M)$-definable type over $\Gamma(M)$. Let
\[
\Delta^{\mathrm{bd}}(M):=\{\eta\in\Gamma(F):|\eta|\leq\gamma\text{ for some }\gamma\in\Gamma(M)\}
\]
be the convex hull of $\Gamma(M)$ in $\Gamma(F)$, and let $\tilde v$ be the coarsening of $v$ by this convex subgroup. Its valuation ring $R_M(F)$ has maximal ideal
\[
\m_{\infty,M}(F):=\{a\in F:v(a)>\gamma\text{ for all }\gamma\in\Gamma(M)\}.
\]
For $\eta\in\Delta^{\mathrm{bd}}(M)$ put
$\st_M(\eta):=\lim\tp(\eta/\Gamma(M))\in\Gamma(M)$, which exists and is unique by \Cref{fact:one-types-gamma}; for $\eta>\Gamma(M)$, and for $\eta=\infty$, put $\st_M(\eta)=\infty$.

On $\Delta^{\mathrm{bd}}(M)$, the map $\st_M$ is a nondecreasing
group homomorphism fixing $\Gamma(M)$. In the divisible case,
$\eta-\st_M(\eta)$ has absolute value smaller than every positive
element of $\Gamma(M)$; the sum of two such errors has the same property.
This proves additivity, and comparison of the corresponding cuts proves
monotonicity. In the $\ZZ$-group case, every bounded definable one-type
is realized, so $\Delta^{\mathrm{bd}}(M)=\Gamma(M)$ and $\st_M$ is the
identity. Let $\Delta^{\mathrm{inf}}(M)=\ker(\st_M).$
This subgroup is convex. In the divisible case it consists of the
elements infinitesimal over $\Gamma(M)$, whereas in the $\ZZ$-group
case it is zero. Thus $\st_M$ preserves weak inequalities; distinct
elements can have the same standard part.

The quotient $F^{\mathrm{st}}_M:=R_M(F)/\m_{\infty,M}(F)$ is a field. The valuation induced by $v$ on this field has value group $\Delta^{\mathrm{bd}}(M)$ and residue field $\RF(F)$. Coarsening it by $\Delta^{\mathrm{inf}}(M)$ gives
\[
v^{\mathrm{st}}(a+\m_{\infty,M}(F)):=\st_M(v(a)).
\]
This valuation extends that of $M$ and has value group $\Gamma(M)$. The composition criterion for henselian valuations shows that $F^{\mathrm{st}}_M$ is henselian of equicharacteristic $0$. \end{definition}

\begin{remark}\label[remark]{rmk:standard-part-equations}
Write $\res^{\mathrm{st}}\colon R_M(F)\to F^{\mathrm{st}}_M$ for the residue map. For a tuple $d$ from $R_M(F)$, put
$d^{\mathrm{st}}:=(\res^{\mathrm{st}}(d_1),\dots,\res^{\mathrm{st}}(d_m)).$ Since $M\cap\m_{\infty,M}(F)=\{0\}$, the map $\res^{\mathrm{st}}$ embeds $M$ into $F^{\mathrm{st}}_M$, and we identify $M$ with its image. As $\res^{\mathrm{st}}$ is a ring homomorphism fixing $M$, $f(d^{\mathrm{st}})=\res^{\mathrm{st}}(f(d))$ for every $f\in M[x_1,\dots,x_m]$; in particular every polynomial equation over $M$ satisfied by $d$ is satisfied by $d^{\mathrm{st}}$.
\end{remark}

\begin{lemma}\label[lemma]{lem:one-plus-m-divisible}
Let $M\subseteq F$ be as in \Cref{def:standard-part}. For every $n\geq2$ the group $1+\m_{\infty,M}(F)$ is $n$-divisible. Consequently, for $a\in R_M(F)\setminus\m_{\infty,M}(F)$ and $c\in M^\times$,
$P_n(ac^{-1})$ holds in $F$ if and only if $P_n\big(\res^{\mathrm{st}}(a)c^{-1}\big)$ holds in $F^{\mathrm{st}}_M$.
\end{lemma}

\begin{proof}
Apply Hensel's lemma for the coarse valuation $\tilde v$ to $X^n-(1+\varepsilon)$ at $1$, for $\varepsilon\in\m_{\infty,M}(F)$. Its derivative $n$ is a unit, since the residue characteristic is zero. The resulting root belongs to $1+\m_{\infty,M}(F)$. An $n$th root in $F$ of $ac^{-1}$ is a unit of $R_M(F)$, so reduction proves the forward implication. Conversely, lift an $n$th root of $\res^{\mathrm{st}}(a)c^{-1}$ to a unit $b\in R_M(F)$. Then $ac^{-1}b^{-n}\in1+\m_{\infty,M}(F)$ is an $n$th power, and hence so is $ac^{-1}$.
\end{proof}

\begin{lemma}\label[lemma]{lem:limits-coordinatewise}
Let $q$ be a global definable type on $\widehat X$ and $r\in\widehat X$. For $f\in\ccO_X$ and $n\geq2$, write $\rho_n f\colon\widehat X\to H_n$ for the evaluation $p\mapsto\rho_n(f(p))$. Then $\lim q=r$ if and only if
$\lim\big(\rho_n f\big)_*q=\rho_n(f(r))$ in $H_n$,
for every $f\in\ccO_X$ and every $n\geq2$.
\end{lemma}

\begin{proof}
Necessity follows from continuity. Conversely, by \Cref{def:TheTopology}, every neighbourhood of $r$ contains a finite intersection of inverse images of definable open neighbourhoods under the evaluation maps $\rho_n f$, allowing parameters from $\cU$. Each inverse image belongs to $q$ by the hypothesis, and hence so does their intersection. Thus $q$ converges to $r$.
\end{proof}

\begin{remark}\label{rmk:babyfactlargeness} Let $L$ be large and $(E,w)$ be a valued field extending (the trivially valued field) $L$. If $L\prec_\exists \RF(E)$, then $L\prec_\exists E$ (where $\prec_\exists$ denotes existential closedness in the language of rings). Indeed, let $M$ be such that $\RF(E)\subseteq M$ and $L\preceq M$. Without loss, we may assume $E$ is henselian. Let $E'$ be an extension of $E$ such that $\Gamma(E)=\Gamma(E')$ and $\RF(E')=M$. By \cite[Theorem 17]{kuhlmann2004}, $M\prec_\exists E'$ and therefore $L\preceq_\exists E'$, which shows that $L\prec_\exists E$. 
\end{remark}

\begin{proposition}\label[proposition]{prop:bounded-limits}
Assume hypothesis~\textup{(B)} and suppose the residue field $\RF$ is large. Let $D$ be a bounded definable subset of $X$ and let $q$ be a global definable type on $\widehat D$. Then $\lim q$ exists in $\widehat X$.
\end{proposition}

\begin{proof}
Choose a small parameter set $A\subseteq\cU$ over which $X,D$, and $q$ are defined. Uniform definability of types makes
$E_\varphi(p;b)\ \text{if and only if}\ \varphi(x;b)\in p$
a relatively definable relation on $\widehat X$, for each formula $\varphi(x;b)$. For the pro-definable family $\mathscr{H}(p)=p$ of field types (indexed by $\widehat{X}$ itself), let $\sigma = \int_q \mathscr{H}$, as in \Cref{def: integral}. In particular, 
\[
\varphi(x;b)\in\sigma
\ \text{if and only if }\ \{p\in\widehat{X}, E_\varphi(p;b)\}\in q, 
\]
which makes $\sigma$ a global $A$-definable type concentrating on $D$.\\
For a fixed coefficient tuple $b$, orthogonality to $\Gamma$ ensures
that every $p\in\widehat D$ determines a unique value
$v(f_b(p))\in\Gamma_\infty$: this is the unique $\gamma$ such that
$v(f_b(x))=\gamma$ belongs to $p$. This induces a natural evaluation map $e_{f,b}:\widehat D\to \Gamma_\infty,
$ given by $ p\mapsto v(f_b(p)).$\\
Its pushforward $(e_{f,b})_*q$ is a global definable type on
$\Gamma_\infty$, recording the valuation of $f_b$ as $p$ varies
according to $q$. Since $D$ is bounded, these valuations are bounded
below, so \Cref{fact:one-types-gamma} gives a unique limit, denoted
$\lambda_f(b)$. 

Equivalently, for every definable open neighbourhood
$W$ of $\lambda_f(b)$, the set
$\{p\in\widehat D:v(f_b(p))\in W\}$ belongs to $q$.
In the $\ZZ$-group case, a finite limit $\gamma$ means that
$v(f_b(p))=\gamma$ belongs to $q$. In the divisible case, it means
that $\gamma-\varepsilon<v(f_b(p))<\gamma+\varepsilon$
belongs to $q$ for every $\varepsilon>0$. In either case,
$\lambda_f(b)=\infty$ precisely when $v(f_b(p))>h$ belongs to $q$
for every $h\in\Gamma$. Moreover, $\lambda_f$ is $A$-definable. Writing $e_f(p,b)=v(f_b(p))$, its graph at a finite $\lambda$ is given in the divisible case by $\forall\varepsilon>0\ (d_qp)\big(\lambda-\varepsilon<e_f(p,b)<\lambda+\varepsilon\big),$
and in the discrete case by $(d_qp)(e_f(p,b)=\lambda)$. In both cases its value is $\infty$ exactly when $\forall\gamma\in\Gamma\ (d_qp)(e_f(p,b)>\gamma).$
Here $(d_qp)\psi(p,b)$ denotes the formula in $b$ defining whether $\psi(p,b)$ belongs to $q$; the displayed conditions are formulas over $A$.\\
We define a global type $r$ by assigning $f_b(x)$ the valuation $\lambda_f(b)$, declaring $f_b(x)=0$ exactly when $\lambda_f(b)=\infty$, and otherwise keeping its nonzero power predicates from $\sigma$. Thus, with $P_n^\times=P_n\setminus\{0\}$, $P_n^\times(f_b(x))\in r$ if and only if 
$\lambda_f(b)\ne\infty$  and $P_n^\times(f_b(x))\in\sigma.$

\begin{claim}The type $r$ is a complete consistent global type.
\end{claim}
\begin{proof}[Proof of the claim:] 
By compactness, it is sufficient to realize any finite collection $\Sigma_{0}(X)$ of these
conditions in a field obtained from a standard-part construction and then embed it into the monster model. \\
Choose a small model $M\prec\cU$ containing $A$ and all parameters in finitely many prescribed conditions, and realize $\sigma\rest_M$ by $d$ in an elementary extension $N$ of $M$. Let $F$ be the relative algebraic closure of $M(d)$ in $N$. This field is henselian, since it is an algebraic extension of the henselization of $M(d)$ inside $N$. Every value in $\Gamma(M(d))$ has a definable type over $\Gamma(M)$, by definability of $\sigma\rest_M$ and stable embeddedness of $\Gamma$. Since $\Gamma(F)/\Gamma(M(d))$ is torsion, the same holds for every value in $\Gamma(F)$. Indeed, a positive integer multiple belongs to $\Gamma(M(d))$ and determines the value uniquely. Thus we can apply  \Cref{def:standard-part}. Boundedness places $d$ in $R_M(F)$, and polynomial evaluation commutes with reduction. In particular,
$v^{\mathrm{st}}(f_b(d^{\mathrm{st}}))=\lambda_f(b) \ (b\in M).$ The value on the right belongs to $\Gamma_\infty(M)$, since $\lambda_f$ is $A$-definable.

Let $E=\RF(F_M^{\mathrm{st}})$. The finer valuation on $F_M^{\mathrm{st}}$ induces on $E$ a valuation trivial on $\RF(M)$, with residue field $\RF(F)$. Since $\RF(M)\subseteq\RF(F)\subseteq\RF(N),$ and $\RF(M)\prec\RF(N),$
we have $\RF(M)\prec_\exists\RF(F)$. By \cref{rmk:babyfactlargeness} we have  $\RF(M)\prec_\exists E$. Also $\Gamma(F_M^{\mathrm{st}})=\Gamma(M)$. By, the Ax-Kochen-Ershov principle we obtain
$(M,v)\prec_\exists(F_M^{\mathrm{st}},v^{\mathrm{st}}).$ By compactness, $F_M^{\mathrm{st}}$ embeds as a valued field over $M$ into an elementary extension of $M$.\\
Choose representatives $c_1,\dots,c_s\in M^\times$ of $M^\times/(M^\times)^n$; the same finite list represents all classes in $N$. If $a\in R_M(F)^\times$, write $a=c_jb^n$ in $N$. Since $b$ is algebraic over $F$, it lies in $F$, and $nv(b)=v(a)-v(c_j)\in\Delta^{\mathrm{bd}}(M)$ implies $b\in R_M(F)^\times$. Reduction shows that all nonzero power classes in $F_M^{\mathrm{st}}$ are represented by this list. They are distinct by \Cref{lem:one-plus-m-divisible}: an $n$th root of $c_i/c_j$ in $F_M^{\mathrm{st}}$ lifts to one in $F$, hence in $N$, and therefore in $M$. Any embedding over $M$ into an elementary extension of $M$ consequently preserves and reflects all power predicates. The embedded tuple $d^{\mathrm{st}}$ realizes the finitely many conditions in $\Sigma_{0}(x)$.\\
Thus $\Sigma_{0}(x)$ is consistent, and so it is $r$. By \Cref{prop:QE}, they determine a complete type, and their uniform definitions give a global $A$-definable type $r$. By \Cref{rmk:standard-part-equations} it concentrates on $X$. This concludes the proof of the Claim. 
\end{proof}

Note that $r\rest_{M}$ is $M$-definable and weakly orthogonal by construction. Then, by \Cref{lem:def-orthogonal}, $r\in \widehat{X}$. 
Finally, fix $f\in\ccO_X$ and $n\geq2$. The type $(\rho_nf)_*q$ has a unique limit by \Cref{lem:bounded-target-limits}. Its valuation coordinate is the value assigned to $f$ by $r$. If this value is finite, the finite power-class coordinate is the class chosen by $\sigma$, which is also the class assigned by $r$. If it is infinite, every branch converges to zero in $H_n$. Therefore
$\lim(\rho_nf)_*q=\rho_n(f(r)).$ By \Cref{lem:limits-coordinatewise}, $\lim q=r$.
\end{proof}

\begin{theorem}\label[theorem]{thm:closed-bounded-defcompact}
Assume \textup{(B)}, suppose the residue field $\RF$ is large, and let $Y\subseteq\widehat X$ be a closed and bounded pro-definable subset. Then $Y$ is definably compact.
\end{theorem}

\begin{proof}
Let $q$ be a global definable type on $Y$. By boundedness, $Y\subseteq\widehat D$ with $D$ bounded and definable, so $\lim q$ exists in $\widehat X$ by \Cref{prop:bounded-limits}. It lies in $Y$ by closedness and \Cref{lem:defcompact-generalities}(1).
\end{proof}
\subsection{Definable compactness implies closed and bounded}
Now we want to show that a pro-definable set $Y$ contained in the analyification $\widehat{X}(M)$, that is definably compact is necessarily closed and bounded.  
This part builds on \cite[Lemma~4.2.12, Proposition~4.2.13]{HruLoe}. The authors argue that any point can be extended to a global definable type in products of $\Gamma_{\infty}$, and then lift such type to the space of stably dominated types centered on the variety $X$. We follow closely their strategy, but instead we work with the approach of directed refinements of definable types introduced in \cite{AndujarJohnsonTypes} to derive the closure condition (\Cref{subsec:directed-refinements}).

\subsubsection{Finite tropical images and boundedness}

\begin{lemma}\label[lemma]{lem:level-images-closed}
Assume \textup{(B)}. Let $Y\subseteq\widehat X$ be a definably compact pro-definable subset. Then, for every finite sequence $\cF=(f_1,\dots,f_\ell)$ of regular functions on $X_{\cU}$ and every $n\geq2$, the image $Y_\cF^n:=\trop^n_\cF(Y)\subseteq H_n^\ell$ is closed.
\end{lemma}

\begin{proof}
Let $b\in H_n^\ell$ be in the closure of $Y^n_\cF$. The sets $H_{n,\mathbf g}$ described in \Cref{rmk:topologyincoordinates} form a finite closed cover of $H^\ell_n$, so $b$ lies in the closure of $Y^n_\cF\cap H_{n,\mathbf g}$ for some $\mathbf g$. Recall that $\theta_{\mathbf g}^{-1}$ is a definable homeomorphism, let $D:=\theta_{\mathbf g}^{-1}(Y^n_\cF\cap H_{n,\mathbf g}),$ and $b':=\theta_{\mathbf g}^{-1}(b)\in\overline D.$
Let $\{W_t\}_{t\in I}$ be the neighbourhood basis at $b'$ and let $\nu_{b'}$ be its global definable cofinal type from \Cref{lem:cofinal-gamma-family}. Consider the pro-definable set $P:=\{(y,t)\in Y\times I:
       \trop^n_\cF(y)\in\theta_{\mathbf g}(W_t)\}.$
Since for $t\in I$, $W_t$ meets $D$, the projection $\pi\colon P\to I$ is surjective. By \Cref{prop:lifting-prodefinable}, there is a global definable type $r$ on $P$ with $\pi_*r=\nu_{b'}$. Let $p$ be its pushforward to $Y$.

For every $t_0\in I$, cofinality of $\nu_{b'}$ implies $\nu_{b'}(t)\models W_t\subseteq W_{t_0}$. Consequently
$p\models \trop^n_\cF(y)\in\theta_{\mathbf g}(W_{t_0}),$
and hence $(\trop^n_\cF)_*p$ converges to $b$. Definable compactness of $Y$ gives a limit of $p$ $z=\lim p\in Y$. By continuity and uniqueness of limits, $\trop^n_\cF(z)=\lim(\trop^n_\cF)_*p=b.$
Thus $b\in Y^n_\cF$.
\end{proof}

\begin{proposition}\label[proposition]{prop:defcompact-bounded}
Assume \textup{(B)}. Every definably compact pro-definable subset of $\widehat X$ is bounded.
\end{proposition}
\begin{proof}
Let $Y\subseteq\widehat X$ be definably compact. Let $x_1,\ldots,x_m$ be the coordinate functions on $X$. Suppose that $h_i\colon Y\to\Gamma_\infty$, given by $h_i(y)=v(x_i(y))$, is unbounded below. Then
\[
P_i:=\{(y,t)\in Y\times\Gamma:h_i(y)<t\}
   \to\Gamma,\qquad  (y,t)\longmapsto t,
\]
is a surjective pro-definable map. Lift the global definable type $\tau_{-\infty}$ of \Cref{not:gamma-types} along this map by \Cref{prop:lifting-prodefinable}. Call such a lift $q$ and let $p$ be the pushforward of the lift to $Y$. For every $\gamma\in\Gamma$, $q\models h_i(y)<t<\gamma$. Thus $(h_i)_*p$ is at $-\infty$ and has no limit in $\Gamma_\infty$. This is a contradiction, since definable compactness of $Y$ gives a limit of $p$ in $Y$, whose image under the continuous map $h_i$ would be a limit of $(h_i)_*p$. Each coordinate therefore has a lower bound, which gives boundedness (cf. \Cref{def:bounded-sets}).
\end{proof}

\subsubsection{Directed refinements of definable types}\label{subsec:directed-refinements}

To construct a definable type concentrating on a pro-definable set with a prescribed limit in its closure, we construct a single definable directed family shrinking to the limit. A definable family $(G_t)_{t\in T}$ is \emph{downward directed} if any two members of it are refined by a third. i.e., for any $G_{t_1},G_{t_2}$, there is $G_{t_3}\subseteq G_{t_1}\cap G_{t_2}$. A global definable type $\nu$ on $T$ is \emph{cofinal} for $(G_t)_{t\in T}$, if for every $t_0\in T$,
$\{t\in T:G_t\subseteq G_{t_0}\}\in\nu.$ Thus an index $t'$ realizing $\nu$ selects a set contained in every member of the family in $T(\cU)$. 

\begin{definition}\label[definition]{def:directed-refinement}
Let $p(x)$ be a global definable type on a definable set $V$. We say that $p$ has the
\emph{directed refinement property}, denoted $\mathrm{(DR)}$, if
\begin{itemize}
    \item for every formula $\varphi(x;y)$, there is a downward-directed definable family $(G_t)_{t\in T}$  such that for every member $\varphi(V;b)$ of the family $\mathcal F_\varphi(p)=\{\varphi(V;b):\varphi(x;b)\in p\}$, there is $G_t\subseteq \varphi(V;b)$.
    \item Moreover, $G_t\in p$ for all $t$ and it admits a global definable cofinal type $\nu$ on $T$. 
\end{itemize}
Note that the family and its cofinal type may depend on
$\varphi$ and may use additional small sets of parameters.
\end{definition}

\begin{lemma}\label[lemma]{lem:residue-directed-refinement}
Every global definable type in a finite tuple of field variables in a
real closed or $p$-adically closed field has property $\mathrm{(DR)}$.
\end{lemma}
\begin{proof}
For real closed fields, \cite[Lemma~2.10]{AndujarTypes} gives a definable downward-directed refinement for any formula and \cite[Proposition~2.17]{AndujarTypes} gives
a definable cofinal type on the index set. For $p$-adically closed fields,
\cite[Proposition~7.5]{AndujarJohnsonTypes} gives a definable downward-directed refinement for any formula indexed by the value group, and the definable type at $\tau_\infty$ in the value group is cofinal.
\end{proof}

\begin{lemma}\label{lem:transfer-directed-refinement} Working in the theory of a $\bm$-field, suppose that every global definable type concentrating on finite tuples of residue field variables has property $\mathrm{(DR)}$. Then every global definable type orthogonal to $\Gamma$ concentrating on finite tuples of valued field variables has
property $\mathrm{(DR)}$.
\end{lemma}
\begin{proof}
Fix such a type $p(x)$ and a formula $\varphi(x;y)$. Choose a small
maximally complete elementary model $M$ containing the parameter of definition of $p$ and parameters in $\varphi$. By the quantifier elimination in \Cref{prop:QE}, the formula $\varphi$ is determined by
finitely many polynomials $F_j(x;y)$, power predicates applied
to them, and pure value group formulas involving their valuations. Let $d$ bound
the degrees in $x$ of the $F_j$ 's.

For $\xi\models p|_{M}$, put $I=\{f\in M[x]:f(\xi)=0\}$. By maximal completeness of $M$, the finite-dimensional $M$-vector space $V$
of values at $\xi$ of polynomials of degree at most $d$ in variable $x$ has a
valuation-independent basis.  Note further that $\Gamma(M\xi)=\Gamma(M)$ by orthogonality, so one can choose the basis to be $e_1,\ldots,e_s\in M[x]$ with the following:
\begin{equation}\label{eq:DR-separated}
 v(e_i(\xi))=0,\qquad
 v\left(\sum_i c_i e_i(\xi)\right)=\min_i v(c_i)
 \quad(c\in M^s).
\end{equation}
Since $p$ is definable over $M$, $p$ knows the above extend to every coefficient
tuple $c\in\cU^s$: the above each identity to holds throughout $M^s$ with a realization $\xi$, hence
throughout $\cU^s$ by elementarity for any $\xi'$ realizing $p\rest_{\mathcal{U}}$.

Note that the ideal $I$ is finitely generated. Define
\[
 D=\{x\in V(I):v(e_i(x))=0\text{ for all }i\},\qquad
 r:D\to k^s,\quad r(x)=(\res(e_i(x)))_i.
\]
Then $D$ is a definable subset of the variety $V(I)$ and $p$ concentrates on $D$.
The pushforward $\kappa=r_*p$ is a global definable type in the residue field.
There are polynomial functions $c_{ji}(y)$ over $M$ such
that, uniformly in $y$, $F_j(x;y)\equiv_I\sum_i c_{ji}(y)e_i(x).$

Fix a tuple $b$. If all $c_{ji}(b)$'s vanish, then $F_j(x;b)=0$
for every $x\in D$. Otherwise, choose one $c_{ji}(b)$ of least
valuation, denoted $C_j(b)$, and put
\[
 \alpha_{ji}(b)=\res(c_{ji}(b)/C_j(b)),\qquad
 L_{j,b}(z)=\sum_i\alpha_{ji}(b)z_i.
\]
By \eqref{eq:DR-separated}, the valuation of $F_j(x;b)/C_j(b)$ evaluated at $p$
is zero. Its residue is therefore nonzero, which gives
that $\kappa$ concentrates on $L_{j,b}(z)\ne0$ for some $j$. On the definable set $D\cap r^{-1}(L_{j,b}(z)\ne0)$, $p$ concentrats on the following:
\[
 v(F_j(x;b))=v(C_j(b)),\qquad \res(F_j(x;b)/C_j(b))=L_{j,b}(r(x)).
\]

The truth of $P_n$ of the $F_{j}$'s evaluated at $p$ is also determined by the residue field data. For each $P_n$ occurring in the relative quantifier elimination of $\varphi$, take representatives
$u_\ell\in M^\times$ of the cosets in the finite quotient
$\cU^\times/(\cU^\times)^n$. If
$C_j(b)\in u_\ell(\cU^\times)^n$ and $v(w)= 0$, then
$P_n(C_j(b)w)$ is equivalent to $P_n(u_\ell w)$, which is false when $v(u_\ell)\notin n\Gamma$. When $v(u_\ell)\notin n\Gamma$, there is
$d_\ell\in M^\times$ with $n v(d_\ell)=v(u_\ell)$; Hensel's lemma gives
$P_n(u_\ell w)$ if and only if $P_n^k\bigl(\res(u_\ell/d_\ell^n)\res(w)\bigr)$,
where $P_n^k$ denotes the $n$th-power predicate in the residue field.

It follows that, for each $b$ with $\varphi(x;b)\in p$, there is a
residue-field formula $\theta_b(z)\in\kappa$ such that
\begin{equation}\label{eq:DR-force}
 D(x)\wedge\theta_b(r(x))\quad\text{implies} \quad\varphi(x;b).
\end{equation}
The dependence of $b$ is also uniform by the above construction. So we indeed have a definable family.

The property $\mathrm{(DR)}$ in the residue field gives a directed refinement $(G'_t)$ with $G'_t\in \kappa$ and a definable cofinal type $\nu$. Then
$G_t=D\cap r^{-1}(G'_t)$
belongs to $p$, is downward directed, and refines
$\mathcal F_\varphi(p)$ by \eqref{eq:DR-force}, and $\nu$ is the cofinal type of this family by pure stable embeddedness of $\Gamma$. Thus $p$ has $\mathrm{(DR)}$.
\end{proof}

\subsubsection{From directed refinements to curve selection}\label{subsec:DR-curve-selection}

We now wish to further refine the definable sets in a directed refinement by shrinking family of open neighbourhoods of the corresponding point of $\widehat X$. We first specify the parameters used to shrink these neighbourhoods.

Put
\[
 S=\begin{cases}
 \Gamma\times\Gamma_{>0},&\Gamma\models\mathrm{DOAG},\\
 \Gamma,&\Gamma\models\mathrm{PRES}.
 \end{cases}
\]
In the divisible case write $s=(h,\varepsilon)$ and say that $s'$
\emph{refines} $s$ if $h'\geq h$ and $0<\varepsilon'\leq\varepsilon$.
In the $\ZZ$-group case $h'$ refines $h$ means $h'\geq h$.
A \emph{tail} of $S$ is a set containing all refinements of some fixed
parameter. Note that there exist a definable type $\tau$ on $S$ containing every
tail: use $\tau_{\infty}$ in the $\ZZ$-group case and a tensor product
of $\tau_{\infty}$ and $\tau_{0^+}$ in the divisible case. \\

For $b\in H_n^\ell$, let $W(b;s)$ be the product of the following neighbourhoods of $b$. At a coordinate $b_i=\mathbf0$, we take the neighborhood given by 
$\bar v(u_i)>h$. For $b_i\neq \mathbf 0$, (1)
in the case $\Gamma$ is divisible, take the neighbourhood given by the same class in $\RV^\times/(\RV^\times)^n$ as $b_i$ and
 $|\bar v(u_i)-\bar v(b_i)|<\varepsilon$; (2)
in the $\ZZ$-group case, take the neighbourhood to be $\{b_i\}$. These conditions give a
uniform definable neighbourhood basis of $b$, decreasing under refinement of
$s$. The definition is clearly uniform.

\begin{lemma}\label{lem:thicken-directed-refinements}
Assume \textup{(B)} and suppose the residue field $\RF$ is large. Let
$B\subseteq X$ be definable, with $\widehat B$ closed and bounded,
and let $a\in \widehat{B}$. Let $(G_t)_{t\in T}$ be a definable family indexed
by a definable set in a finite tuple of field variables, with $a \in G_{t}$
for every $t$. There is a definable family $(N_{t,s})_{(t,s)\in T\times S}$
of relatively definable open neighbourhoods of $a$ in $\widehat{B}$ such that,
for every relatively definable open $U\subseteq \widehat{B}$,
\[
 \widehat{B}\cap\widehat G_t\subseteq U
 \quad\text{implies} \quad
 N_{t,s}\subseteq U\text{ for all }s\text{ in some tail of }S.
\]
Particularly, for fixed $U$, $\{(s,t): N_{t,s}\subseteq U \}$ is definable.
\end{lemma}
\begin{proof}
Apply the quantifier elimination \cref{prop:QE} to the formula defining the family $G_t$. It is equivalent to a formula with only finitely many polynomial families and power predicates appearing in it.
Choose a common multiple $n$ of their exponents and list the polynomials
appearing in the quantifier elimination as $\mathcal F_t$. The $H_n$-values of $\mathcal{F}_t$ encodes all the information needed to decide truth of the formula. In other words, for the function $f_t=(\operatorname{trop}_{\mathcal F_t}^{n})_{|\widehat{B}}:\widehat{B}\to H^\ell_n$, $G_t$ is a union of fibres of $f_t$.
Take $b_t=f_t(a)$, this means
\begin{equation}\label{eq:DR-fiber}
 f_t^{-1}(b_t)\subseteq \widehat{B}\cap\widehat G_t.
\end{equation}
The definability of $a$ makes the function $t\mapsto b_t$ definable. Set $N_{t,s}=f_t^{-1}(W(b_t;s))$.

Suppose $\widehat{B}\cap\widehat G_t\subseteq U$. The complement $C=\widehat{B}\setminus U$
is closed, bounded, and pro-definable, hence definably compact by
\Cref{thm:closed-bounded-defcompact}. Its image $f_t(C)$ is closed by
\Cref{lem:level-images-closed} and does not contain $b_t$ by \eqref{eq:DR-fiber}.
Since $W(b_t;s)$'s forms a basis around $b_t$, for some $s_0$, $W(b_t;s_0)$ is disjoint from $f_t(C)$, and so is every refinement $s$ of $s_0$. 

Finally, the condition
$N_{t,s}\subseteq U$ if and only if
 $\neg\exists z\in \widehat{B}\,(z\in N_{t,s}\ \wedge\ z\notin U)$
uses only finitely many coordinates of $z$, uniformly in $(t,s)$.
The image of $\widehat{B}$ in these coordinates is definable by strict
pro-definability. Hence, the last statement also follows.
\end{proof}

\begin{proposition}\label{prop:curve-selection-DR}
Assume \textup{(B)} and the residue field $\mathbf{k}$ is large. If
$a\in\widehat X$ has property $\mathrm{(DR)}$ and $a\in\operatorname{cl}(Y)$ for a pro-definable set $Y\subseteq\widehat X$,
then there is a global definable type on $Y$ converging to $a$.
\end{proposition}
\begin{proof}

Fix a closed affine embedding of $X$ and choose $\beta\in\Gamma$ less than all of the valuations of the coordinates of $a$. The set
$B=\{x\in X:v(x_i)\geq\beta\text{ for all }i\}$
is the intersection of $X$ with a closed coordinate polydisc, and $a$ lives in the interior of $\widehat{B}$. We may replace $Y$ by $Y\cap \widehat B $ and we still have $a\in\operatorname{cl}(Y)$.

For each formula $\varphi(x;y)$, choose a directed refinement
$(G_{\varphi,t})_{t\in T_\varphi}$ in $a$ and a cofinal definable type
$\nu_\varphi$. We may index the family by a definable set $T_\varphi$ in the valued field sort and by \Cref{prop:lifting-definable}, we have lift of the cofinal type $\nu_\varphi$, which stays cofinal in the lifted family.
By \Cref{lem:thicken-directed-refinements}, we have
$N_{\varphi,t,s}$ as in the statement. Let $\lambda_\varphi(t,s)$ be the tensor product $\tau(s)\otimes \nu_{\varphi}(t)$.

Let $U$ be a basic open neighbourhood of $a$ in $\widehat{B}$. It has the form
$\widehat{B}\cap\widehat{\varphi(X;b)}$ for some formula and parameter tuple $b$.
Refinement and cofinality gives $\{t:G_{\varphi,t}\subseteq\varphi(X;b)\}\in\nu_\varphi.$
For each such $t$, the condition $N_{\varphi,t,s}\subseteq U$ contains
a tail in the $s$-variable. And this condition is definable, so its
$\tau$-definition holds for every such $t$. By the definition of $\lambda_\varphi$, we have
\begin{equation}\label{eq:DR-cofinalN}
 \{(t,s):N_{\varphi,t,s}\subseteq U\}\in\lambda_\varphi.
\end{equation}

 Taking iterated tensor products indexed by all the formulas gives a global definable type $\lambda$ on $J=\prod_\varphi(T_\varphi\times S)$ whose projection to each factor is $\lambda_\varphi$. Write
$j=((t_\varphi,s_\varphi))_\varphi$ and define
$ P=\{(y,j)\in Y\times J:
       y\in N_{\varphi,t_\varphi,s_\varphi}
       \text{ for every }\varphi\}.$
This is clearly pro-definable. Moreover, for a fixed $j$, every finite intersection of the specified neighbourhoods of $a$ meets $Y$. Together with the fact that $Y$ is pro-definable, these conditions form a finitely satisfiable small partial type in variable $y$. By saturation in the universal domain, it gives a point $e$ with $(e,j)\in P$, thus $P\to J$ is surjective.

By \Cref{prop:lifting-prodefinable}, there is a lift of $\lambda$ to a global definable type on $P$: let $p$ be its projection to $Y$. For every basic neighborhood $U$ of $a$, \eqref{eq:DR-cofinalN} and
the defining conditions of $P$ gives $U\in p$. Hence $\lim p=a$.
\end{proof}
\subsubsection{The algebraically closed residue-field case}
\label{subsec:stable-residue-curve}

For algebraically closed residue field and divisible value group, the
theory is $\ACVF$ and curve selection is covered by \cite[Chapter $4$]{HruLoe}.  However, we noticed that the proof of \cite[Lemma~4.2.12]{HruLoe} contains a gap, indeed the following remark shows that their statement is false. This is neverless a minor issue.

\begin{remark}\label{rmk:limit-not-definable}
Suppose $\Gamma(K)$ is such that every $n$-type over $\Gamma(K)$ admitting a limit in $\Gamma_\infty^n(K)$ is $\Gamma(K)$-definable. This property forces $\Gamma(K)$ to be Dedekind complete, that is, $\Gamma(K)\cong\RR$ or $\Gamma(K)\cong\ZZ$ as ordered groups. Let us show it fails already for $\Gamma(K)=\QQ$: 
let $a_1>\QQ$, let $s$ realize the cut of $\pi$ over $\QQ$, set $a_2=a_1+s$ and $p_0=\tp(a_1,a_2/\QQ)$. Every $K$-definable open neighbourhood $U$ of $(\infty,\infty)$ in $\Gamma_\infty^2$ contains a box $(\gamma_0,\infty]^2$ with $\gamma_0\in\Gamma(K)$, 
so $\lim p_0=(\infty,\infty)$; but $\{c\in\QQ : (x_2-x_1>c)\in p_0\}=\{c<\pi\}$, so $p_0$ is not $\QQ$-definable. This example shows that \cite[Lemma~4.2.12]{HruLoe} fails as stated. The argument below avoids this issue altogether: instead of completing convergent types, we construct definable convergent types directly from the cofinal types of \Cref{lem:cofinal-gamma-family}, in the spirit of \cite[Remark~4.2.15, Lemma~4.2.18]{HruLoe}.
\end{remark}
However, we want to emphasize that this causes minor issues on their general strategy. And indeed, their ideas heavily inspired our arguments in this section. For sake of completeness, and with the hope to clarify and clean as much as possible the literature we added a remark \cref{rmk:acvf} to explain how the case of $\mathrm{ACVF}$ is treated within our approach. However, the main structure and idea of the argument are still due to the authors of \cite{HruLoe}.

Throughout this subsection we consider the remaining case $T=\operatorname{Th}(\lau{\CC}{t})$. Thus $k$ is algebraically closed of characteristic zero and $\Gamma\models\mathrm{PRES}$. We use stable domination in the sense of \cite[Chapter~3]{HHM}, and denote by $\mathrm{St}_{A}$ the stable part of the structure of $A$-definable stable and stably embedded sets within a theory $T$ as described in \cite[Definition $3.1$]{HHM}. 

\begin{remark}\label{rmk:propertiesstabledomination} 
Let $A$ be a small set of parameters in $\mathcal{U}^{eq}$.
\begin{enumerate}
\item Definable pushforwards of stably dominated types are stably dominated. More precisely, if $\operatorname{tp}(a/A)$ is stably dominated and
$b\subseteq\operatorname{dcl}^{\mathrm{eq}}(Aa)$, then
$\operatorname{tp}(b/A)$ is stably dominated \cite[Proposition~3.32(iii)]{HHM}.
\item A stably dominated type over $A$ has a global extension definable over
$\operatorname{acl}^{\mathrm{eq}}(A)$. Moreover, if the base is already algebraically closed, then it has a unique global invariant extension
\cite[Proposition~3.13.(ii)]{HHM}. 
\item Invariant base extensions remain stably
dominated. Let $p$ be a global
$\operatorname{acl}^{\mathrm{eq}}(A)$-invariant type with
$p\rest_A$ stably dominated. For every small parameter set
$B\supseteq A$, the restriction $p\rest_B$ is stably dominated \cite[Proposition~4.1]{HHM}.\footnote{Note that this requires the going up of base change of stable domination, which was originally proved in \cite{HHM} for arbitrary theories. }
\end{enumerate}
\end{remark}
Stable domination for the structure $\lau{\CC}{t}$ was originally formulated in \cite[Theorem $16.1$,Lemma $16.6$]{HHM}. However, some gaps (specifically for the elements in the geometric sorts) where present in their proofs and therefore we give quickly some arguments around it for sake of completeness. A finer formulation of residue field domination (including the sorts internal to the residue field) is given in \cite[Definition $2.8$]{orthogonal}. The later notion was introduced because in some henselian valued fields, the stable part of the structure $\mathrm{St}_{A}$ is strictly larger than the family of sorts that are internal to the residue field.

We will require however the domination result in $\lau{\CC}{t}$ only over a maximally complete model by the sorts internal to the residue field. Since the residue field is stable and stably embedded this will guarantee that we obtain stable domination. We will use the subscripts $T$ and $\ACVF$ to distinguish the definable closure, algebraic closure or the types in $T$ or in $\mathrm{ACVF}$ respectively. 

\begin{remark}\label{rmk:stabledom} Let $M$ be a maximally complete model of $T=\mathrm{Th}(\lau{\CC}{t})$. Let $M \prec \mathcal{U} \models T$, where $\mathcal{U}$ denotes the monster model. Let $d$ be a tuple in the field sort, let $\Gamma(Md)=\Gamma\cap\operatorname{dcl}_{T}^{\mathrm{eq}}(Md)$. Then $\tp(d/ M \Gamma(Md))$ is stably dominated. 
\end{remark}

\begin{proof}
Let $L=\acl^{eq}_{T}(Md) \cap K$, then $\Gamma(L)=\Gamma(Md)=\Gamma \cap \dcl^{eq}(Md)$. Let $S=\big( \mathbb{Q} \otimes v(M(d)^{\times}\big) \cap \Gamma(\mathcal{U})$, and note that $S=\Gamma(L)$. For the left to right direction, if $n\gamma=v(a)$ with $a \in M(d)^{\times}$ and $\gamma \in \Gamma(\mathcal{U})$, since the residue field is algebraically closed by Hensel's lemma we can find $b \in \mathcal{U}$ such that $v(b)=\gamma$ and $b^{n}=a$. For the converse, $\Gamma(L) \subseteq \Gamma \cap \acl_{T}^{eq}(Md) = \Gamma \cap \dcl^{eq}(Md)$, since the value group is linearly ordered. Furthermore, $k_{L}$ is algebraically closed. Write $v_\rv$ for the map from $\RV\to \Gamma_\infty$. For each $\gamma \in S$ let $\RV_{\gamma}=v_{\mathrm{rv}}^{-1}(\gamma) \cap \mathrm{RV}$, and write $\mathcal{D}_{S}=(\mathrm{k}, (\mathrm{RV}_{\gamma})_{\gamma \in S})$. For a field $F$ whose value group contains $S$, we write $\mathcal{D}_{S}(F)$ for $(k(F), \mathrm{RV}_{\gamma \in S}(F))$. Let $A= M \cup S$, note that each fiber $\mathrm{RV}_{\gamma}$ is internally internal to the residue field $\mathrm{k}$, so these sorts are stable and stably embedded over $A$, and hence $\mathcal{D}_{S}(F) \subseteq \mathrm{St}_{A}(F)$. Let $N \preceq \mathcal{U}$ containing $A$.

\begin{claim*}
For $F=L$ or $F=N$ we have that $\acl_{\mathrm{ACVF}}(A \mathcal{D}_{\mathbb{Q}\otimes S}(F^{alg}))=\mathrm{acl}_{\mathrm{ACVF}}(A \mathcal{D}_{S}(F))$.    
\end{claim*}

\begin{proof}[Proof of the claim:]
The right to left inclusion is clear. For the converse, since $k_{F}$ is algebraically closed, $k_{F^{\mathrm{alg}}}=k_{F}$. Let $u \in \mathrm{RV}(F^{\mathrm{alg}})$ and $v_\rv(u) \in \mathbb{Q} \otimes S$. Choose $n \geq 1$, such that $n v_\rv(u)=\gamma \in S$ and choose $a \in X$ such that $v_\rv(a)=\gamma$. Then $\frac{u^{n}}{\mathrm{rv}(a)} \in k_{F^{\mathrm{alg}}}^{\times}=k_{F}^{\times}$. Therefore $u^{n} \in \mathcal{D}_{S}(F)$. Hence, the equation $z^{n}=u^{n}$ has finitely many solutions in $\mathrm{RV}$, so $u \in \acl_{\mathrm{ACVF}}(A D_{S}(F))$.
\end{proof}

Now we write $\mathrm{St}^{T}_{A}$ to denote the stable structure over $A$ within the theory $T$. Assume that, $\mathrm{St}_{A}^{T}(L) \downfree^{\mathrm{St^{T}_{A}}}_{A} \mathrm{St}_{A}^{T}(N)$, since $\mathcal{D}_{S}(F) \subseteq \mathrm{St_{A}}(F)$ then we have $\mathcal{D}_{S}(L) \downfree_{A} \mathcal{D}_{S}(N)$. By relative quantifier elimination down to $\mathrm{RV}$ and quantifier elimination of $\mathrm{RV}$ down to the residue field and the value group \cite[Proposition $2.9$]{vicariaresdom}, algebraic dependence over \(A\) on \(k\) and the fixed $\mathrm{RV}$ fibers agrees in \(T\) and $\mathrm{ACVF}$, and hence so does forking independence, since these sorts are strongly minimal.

Hence, $A\mathcal{D}_{S}(F) \downfree_{A}^{\mathrm{St_{A}^{\mathrm{ACVF}}}} A \mathcal{D}_{S}(N)$ within the stable structure $\mathrm{St}_{A}^{\mathrm{ACVF}}$ computed in $\mathrm{ACVF}$. Since forking independence is preserved by taking algebraic closure, by the claim $\mathcal{D}_{\mathbb{Q}\otimes S} (F^{\mathrm{alg}}) \downfree^{\mathrm{St}_{A}^{\mathrm{ACVF}}}_{A} \mathcal{D}_{\mathbb{Q} \otimes S}(L^{alg})$ within the stable structure $\mathrm{St}_{A}^{\mathrm{ACVF}}$. Note that since the residue field is algebraically closed,
$k^{\times}/(k^{\times})^{n}$ is trivial for every $n\geq 1$.  By \cite[Corollary $5.8$]{vicariaresdom}, we have $\tp_T(L/A\mathcal{D}_S(N))
\vdash \tp_T(L/N).$
Since $\mathcal{D}_S(N)\subseteq \mathrm{St}_A^T(N)$, it follows that
$\tp_T(L/A\mathrm{St}_A^T(N))
\vdash \tp_T(L/N)$. Consequently, $\tp_T(L/A)$ is stably dominated.
Finally, stable domination is preserved under definable images
\cite[Proposition~3.32(iii)]{HHM}. Since $d$ is a sub-tuple of an
enumeration of $L$, the type
$\tp_T(d/A)=\tp_T(d/M\Gamma(Md))$ is stably dominated.
\end{proof}

Note that formulations of residue field domination, including the geometric elements, are provided in \cite[Theorem $7.1$, Remark $7.4$]{orthogonal}.\\

\begin{lemma}
\label{lem:stable-gamma-lift}
Let $M$ be a maximally complete model, let $D$ be pro-$M$-definable, and
let $F:D\to\Gamma_\infty^I$ be pro-$M$-definable, with $I$ small.
Suppose that $q$ is a global definable type on $F(D)$, definable over
$M\cup E_0$ for some small $E_0\subseteq\Gamma$. Then $q$ has a global
definable lift $p$ to $D$ that is definable over
$\operatorname{acl}^{\mathrm{eq}}(M\cup E)$ for some small
$E\subseteq\Gamma$ containing $E_0$.
\end{lemma}
\begin{proof}
Choose $b\in F(D)$ realizing $q|_{M}$ and $c\in D$ with $F(c)=b$.
Choose a small field tuple $d$ representing the coordinates of $c$, so
that $c\in\operatorname{dcl}^{\mathrm{eq}}(Md)$, and let
$\gamma=(\gamma_j)_{j\in J}$ enumerate $\Gamma(Md)$.
Let $W=\{\gamma'\in\Gamma^J \ | \ \tp(\gamma'/M)=\tp(\gamma/M)\}.$
Note that this set is pro-$M$-definable set, it can be identified with the inverse limit of the formulas of
$\tp(\gamma/M)$ where all its coordinates all lie within $\Gamma$.

For each $i\in I$, we have 
$b_i=F_i(c)\in\operatorname{dcl}^{\mathrm{eq}}(Mc) \subseteq\operatorname{dcl}^{\mathrm{eq}}(Md).$
If $b_i\ne\infty$,then it is an entry $\gamma_{j(i)}$ of
$\gamma$. Otherwise it is the distinguished point $\infty$. Taking projections on the coordinates and constant coordinates when the entry is $\infty$ we obtain a 
pro-$M$-definable map $g:W\to\Gamma_\infty^I$ with $g(\gamma)=b$.
By homogeneity, 
$ g(W)=\{b' \  | \ \tp(b'/M)=\tp(b/M)\}.$
Since $b\models q|_{M}$, every formula presenting this image belongs to
$q$. Thus $q$ concentrates on $g(W)$.
By \cref{rmk:stabledom}, $\tp(d/M\gamma)$ is stably dominated. By \cref{rmk:propertiesstabledomination}.(1) its definable pushforward $\tp(c/M\gamma)$ is also stably dominated.
By \cref{rmk:propertiesstabledomination}.(2) there is a global stably dominated extension $s$ definable over
$\operatorname{acl}^{\mathrm{eq}}(M\gamma)$. Let $\lambda$ be a small
tuple of canonical parameters of its definitions. Then
$\lambda\subseteq\operatorname{acl}^{\mathrm{eq}}(M\gamma)$.
Let
\[\widetilde W=\{(\gamma',\lambda'):
       \tp(\gamma',\lambda'/M)=\tp(\gamma,\lambda/M)\}.
\]
The projection $\pi:\widetilde W\to W$ is surjective by homogeneity.

For $w=(\gamma',\lambda')\in\widetilde W$, choose
$\sigma\in\operatorname{Aut}(\cU/M)$ taking $(\gamma,\lambda)$ to $w$
and define $\mathscr H(w)=\sigma(s)$. This is independent of the choice
of $\sigma$, since fixing $\lambda$ fixes every definition of $s$.
Note that $\mathscr H$ a pro-$M$-definable family of global
stably dominated types on $D$. Its pro $M$-definability follows by identify each type with its canonical base $\lambda$, and the every finite tuple of such canonical set belongs to an interpretable set by uniform definability of types. (\emph{c.f.} \cref{prop:strict-prodefinability}).\\
The global type $s$ contains the formula
$F_i(x)=g_i(\gamma)$ for every $i$. Hence, its conjugates satisfy
\begin{equation}\label{eq:stable-family-image}
 F_*\mathscr H(\gamma',\lambda')
 =\tp(g(\gamma')/\cU).
\end{equation}

Since the value group is stably embedded, we can view $W$, $g(W)$, and $g$ as pro-definable
objects in the pure value-group structure over $\Gamma(M)$.
Thus we can apply \Cref{prop:lifting-prodefinable}, and we obtain
a global definable type $r$ on $W$ with $g_*r=q$. Enlarge a small
$E\subseteq\Gamma$ so that it contains $E_0$ and the parameters defining
$r$, and let $C=M\cup E$ and $A=\operatorname{acl}^{\mathrm{eq}}(C)$.
By \Cref{prop:lifting-prodefinable} applied to $\pi$, we obtain a global
definable lift $\widetilde r$ of $r$ to $\widetilde W$.

\begin{claim*} $\widetilde{r}$ is $A$-definable.
\end{claim*}

\begin{proof}[Proof of the claim:]
For any finite tuple $\lambda_0$ of coordinates of
$\lambda$, choose a finite tuple $\gamma_0$ from $\gamma$ such that
$\lambda_0\in\operatorname{acl}^{\mathrm{eq}}(M\gamma_0)$. An $M$-definable relation witnessing this algebraicity has
uniformly finite fibers after restricting its domain by a formula of
$\tp(\gamma_0/M)$. Hence the projection of $r$ to the $\gamma_0$-variables
has only finitely many complete global lifts compatible with
$\tp(\gamma_0,\lambda_0/M)$. Since $r$ is  $C$-invariant, the projection
of $\widetilde r$ to $(\gamma_0,\lambda_0)$ has only finitely many
conjugates over $C$.  It is definable, so the canonical parameters of
its definitions have finite orbit over $C$ and belong to $A$.
Applying this to every finite set of coordinates proves the claim. 
\end{proof}

Let $p=\int_{\widetilde r}\mathscr H$, as in \cref{def: integral}. It is a global $A$-definable type on $D$. To compute its pushforward,
take a small elementary model $N$ containing $A$, choose
$w^*=(\gamma^*,\lambda^*)\models\widetilde r\rest_{N}$, and then choose
$c^*\models\mathscr H(w^*)\rest_{Nw^*}$. By \cref{rmk:propertiesintegral} we have $p\res_{N}=\tp(c^*/N)$. By \eqref{eq:stable-family-image},
$F(c^*)=g(\gamma^*)$, and this tuple realizes $q\rest_{N}$ because
$\gamma^*\models r\rest_{N}$ and $g_*r=q$. Thus $F_*p=q$.
\end{proof}

We have now enough tools to show curve selection in $\lau{\CC}{t}$.

\begin{theorem}\label{thm:stable-residue-curve}
Let $X$ be an affine variety,
let $Y\subseteq\widehat X$ be pro-definable, and let
$a\in\operatorname{cl}(Y)$. There is a global definable type $p$ on $Y$
with $\lim p=a$.
\end{theorem}
\begin{proof}
Fix an affine embedding of $X$. Intersect $X$ with a closed coordinate
polydisc large enough that the resulting set $B$ has $a$ in the interior
of $\widehat B$. Replacing $Y$ by $Y\cap\widehat B$, we may assume
$Y\subseteq\widehat B$.

Let $M$ be a model of $T=\mathrm{Th}(\lau{\CC}{t})$ that is maximally complete and contains the parameters defining $X$, $B$, and $Y$, and the  canonical base of $a$. In particular, $a$ is a global $M$-definable field type. Enumerate
$M[X]$ as $(f_i)_{i\in I}$ and let
\[
 F_M:\widehat B\to \Gamma_\infty^I,\qquad
 F_M(z)=(v(f_i(z)))_{i\in I},\qquad
 V=F_M(Y),\quad b=F_M(a).
\]
The map is continuous and pro-$M$-definable, and
$b\in\Gamma(M)_\infty^I$. Its image $V$ is pro-definable. Since the value group is stably embedded, we might take the pro-definable definition of $V$ using formulas and parameters within the value group. Since $F_{M}$ is a continuous map, $b\in\operatorname{cl}(V)$.
By \Cref{lem:cofinal-gamma-family}, there is a global definable type
$q$ on $V$ converging to $b$, whose definitions use only value-group
parameters in addition to $M$.

Apply \Cref{lem:stable-gamma-lift} to $F_M\rest_{Y}$. We obtain a global definable type $p$ on $Y$ and a small set $E\subseteq\Gamma$ such that $(F_M)_*p=q$,
 $p$ is definable over $A=\operatorname{acl}^{\mathrm{eq}}(M\cup E)$. By \Cref{prop:bounded-limits}, $p$ has a limit $z\in\widehat X$. Since $\widehat{B}$ is closed, then $z \in \widehat{B}$. Moreover, since the topology on $\widehat{X}$ is Hausdorff, so it is any subspace, and therefore limits are unique. Every automorphism over $A$ fixes $p$ and its limit;
therefore the global field type $z$ is $A$-invariant. Since $F_{M}$ is a continuous map, 
$F_M(z)=\lim(F_M)_*p=\lim q=b=F_M(a).$

\begin{claim*}
$z \rest_{C}=a\rest_{C}$. 
\end{claim*}
\begin{proof}[Proof of the claim:]
Note that for $u\ne0$, $P_n(u)$ is equivalent to $v(u)\in n\Gamma$.
By \Cref{prop:QE}, the valuations of all regular functions over $M$
therefore determine the field type over $M$, and consequently $z\rest_{M}=a\rest_{M}$. Let $C=M\cup E$. Note that for every polynomial $P(x) \in M[X]$ there is some $\gamma \in \Gamma(M)_{\infty}$ such that $v(P(x))=\gamma \in a\rest_{M}$, as the type is orthogonal to $\Gamma$ and $M$ contains the canonical base of $A$. By quantifier-elimination and stable embeddedness of the value group, the additional value-group parameters in $E$ impose conditions
only on these fixed values, specifically they determine some order type of the elements of $e \in E$ relative to $\Gamma_{\infty}(M)$ and the equivalence classes of the terms generated by $E$ and $\Gamma_{\infty}(M)$. In particular, $a\rest_{M}$ has a unique extension to $C$, as it is orthogonal to $\Gamma$. Consequently, 
we must have $z\rest_{C}=a\rest_{C}$. This concludes the proof of the Claim. 
\end{proof}

We now show that the type $z \rest_{C}=a \rest_{C}$ has a unique extension to
$A=\operatorname{acl}^{\mathrm{eq}}(C)$. Let $e$ be a
finite tuple from $A$ and let $\delta(w)$ over $C$ isolate $\tp(e/C)$, such formula exists since the type is algebraic.  Its realizations form a finite orbit over $C$. The global type $a$ is
$C$-invariant, since it is $M$-definable, as $M$ contains the canonical base of $a$.  Furthermore, $a$ decides $\varphi(x;e')$
uniformly over the orbit of $e$ over $C$, since the type is $C$-invariant. Taking the finite conjunction shows that
$a\rest_{C}$ contains either $\forall w\,(\delta(w)\rightarrow\varphi(x;w))$ of the formula 
$ \forall w\,(\delta(w)\rightarrow\neg\varphi(x;w)).$ 
Every extension of $a\rest_{C}$ must respect that decision. Hence $z\rest_{A}=a\rest_{A}$.\\
Finally, if $\xi\models a\rest_{M}$,by orthogonality of $a$ we obtain that $\Gamma(M\xi)=\Gamma(M)$. By  \cref{rmk:stabledom} 
$a\rest_{M}$ is stably dominated over $M$. By \cref{rmk:propertiesstabledomination}.(3) its invariant base extension $a\rest_{A}$ is still stably dominated over $A$. Since $A$ is
algebraically closed in imaginary sorts, this type has a unique global
$A$-invariant extension by \cref{rmk:propertiesstabledomination}.(2).
Both $a$ and $z$ are such extensions, so $z=a$ and $\lim p=a$.
\end{proof}

\begin{remark}\label{rmk:acvf}
The preceding argument also proves curve selection in $\ACVF$,
without using \cite[Lemma~4.2.12]{HruLoe}. We recycle the notation $M$, $F_M$, $V=F_M(Y)$, and $b=F_M(a)$
from the proof of \Cref{thm:stable-residue-curve}. Here
$\Gamma\models\mathrm{DOAG}$. Let
$S=\Gamma\times\Gamma_{>0}$. Consider the pro-definable relation
\[
 R=\left\{(u,h,\varepsilon)\in V\times S:
 \begin{array}{ll}
 b_i-\varepsilon<u_i<b_i+\varepsilon
     &\text{if }b_i\ne\infty,\\
 u_i>h&\text{if }b_i=\infty
 \end{array}
 \text{ for every }i\right\}.
\]
Since $b\in\operatorname{cl}(V)$, every finite collection of these conditions is satisfiable for each fixed $(h,\varepsilon)$. Hence $R\to S$ is surjective. Choose a global definable type $\tau$ on $S$ with $h\to+\infty$ and $\varepsilon\to0^+$. Surjectivity transfer in the value group lifts $\tau$ to a global definable type on $R$; its pushforward to $V$ is a global definable type $q$ converging to $b$. Moreover, \Cref{lem:gamma-curve-selection} and \Cref{lem:stable-gamma-lift} holds in $\mathrm{ACVF}$, which gives the final part of the argument.
\end{remark}

\subsection{The definable Heine--Borel theorem}

\begin{theorem}\label[theorem]{thm:curve-selection}
Assume \textup{(B)} and suppose the residue field is algebraically closed,
real closed, or $p$-adically closed. Let $Y\subseteq\widehat X$ be
pro-definable and let $a\in \mathrm{cl}(Y)$. There is a global definable type
$p$ concentrating on $Y$ with $\lim p=a$.
\end{theorem}
\begin{proof}
If the residue field is real closed or $p$-adically closed,
\Cref{lem:residue-directed-refinement,lem:transfer-directed-refinement}
give property $\mathrm{(DR)}$ for $a$, and
\Cref{prop:curve-selection-DR} applies, since these residue fields are
large. For algebraically closed residue field and divisible value group,
the theory is $\ACVF$, and curve selection follows from
\cite[Section~4.2]{HruLoe} (in our approach \cref{rmk:acvf}). For a $\ZZ$-group, this is the theory of $\lau{\CC}{t}$ and the statement follows by  \cref{thm:stable-residue-curve}
\end{proof}

\begin{theorem}\label{thm:heine-borel-vhat}
Assume \textup{(B)}, suppose the residue field is algebraically closed,
real closed, or $p$-adically closed, and suppose $\Gamma$ is divisible
or a $\ZZ$-group. Let $X$ be an affine variety over the small base $K$,
and let $Y\subseteq\widehat X$ be pro-definable. Then:
\[
Y\text{ is definably compact}
\quad\text{if and only if}\quad
Y\text{ is closed and bounded}.
\]
\end{theorem}
\begin{proof}
Since these residue fields are large, closed and bounded implies
definably compact by \Cref{thm:closed-bounded-defcompact}.
Conversely, let $Y$ be definably compact. It is bounded by
\Cref{prop:defcompact-bounded}. For $a\in\mathrm{cl}(Y)$,
\Cref{thm:curve-selection} gives a global definable type $p$ on $Y$
converging to $a$. Since $Y$ is definably compact, then $p$ has a limit $\mathrm{lim}(p) \in Y$. Since the space is Hausdorff, limits are unique and therefore $\mathrm{lim}(p)=a$. Thus $Y$ is closed.
\end{proof}

\appendix\label{appendix}

\section{A characterization of formally $p$-adic fields via hyperfield morphisms}
We use $\pCF$ is the theory of $p$-adically closed fields of $p$-rank $1$ in
$\mathcal L_{\mathrm{Mac}} =\{+,-,\cdot,0,1,(P_n)_{n\ge2}\}$.  The results in this appendix can be generalised to any $p$-adically closed fields of $p$ rank $r$ for any $r$. We decide to stay in this level of generality since it is the primary example we study.
\begin{definition}
For a $p$-adically closed field $K$ of $p$-rank $1$, put
\[
 G_n^\times(K)=K^\times/(K^\times)^n,\qquad
 G_n(K)=G_n^\times(K)\sqcup\{\mathbf0\},
\]
and
\[
 G^\times(K)=\varprojlim_{n\vert m}G_n^\times(K),\qquad
 G(K)=G^\times(K)\sqcup\{\mathbf0\}.
\]
Each $G_n(K)$ is a hyperfield as a Krasner quotient, and hyperaddition on $G(K)$ is defined coordinatewise. Write $\iota_K:K\to G(K)$ for the canonical map. When $K=\QQ_p$,  we simply write $\iota$ for $\iota_{\QQ_p}$ and $G$ for $G(K)$. We denote hyperaddition on these hyperfields by $\boxplus$.
\end{definition}

\begin{definition}
Let $K\models\pCF$ and let $F\supseteq K$ be a field extension.
We say that $F$ is \emph{formally $p$-adic over $K$} if it admits an
$\Lmac$-expansion which is a model of $\pCF_\forall$ and for which
$K\subseteq F$ is an $\Lmac$-embedding. Note that the predicates in this
expansion need not be the actual power predicates of $F$. We say a valuation is a $p$-valuation if the value group is discrete with $v(p)$ its minimal positive element and the residue field is $\mathbb{F}_p$. 
\end{definition}

\begin{theorem}\label{thm:formally-p-adic}
Let $K$ be $p$-adically closed of $p$-rank $1$ and let $F\supseteq K$
be a field extension. The following are equivalent.
\begin{enumerate}[label=\textnormal{(\arabic*)}]
\item There is a hyperfield morphism $\varphi:F\to G(K)$ such that
      $\varphi|_K=\iota_K$.
\item $F$ is formally $p$-adic over $K$.
\end{enumerate}
\end{theorem}
Note that the implication \textnormal{(2)}$\implies$\textnormal{(1)} is trival. So we will only prove  the other direction. 
 As we will see in the proof, if \textnormal{(1)} holds, the interpetation of $P_n$ making $F$ formally $p$-adic according to $\varphi$ can be described by:
\[
 P^\varphi_n(x)\quad\Longleftrightarrow\quad
 x=0\ \text{or}\ \varphi_n(x)=1,
\]
where $\varphi_n$ is the composition of $\varphi$ with the projection to $G_n$. 
In other words, in this case, there are an elementary extension $K'\succeq K$
and a $K$-embedding $j:F\to K'$ such that
\[
 \varphi_n(x)=[a]_n
 \quad\Longleftrightarrow\quad
 j(x)/a\in(K'^\times)^n
 \qquad(x\in F^\times,\ a\in K^\times,\ n\ge2).
\]
To prove this, we need to understand the structure of the hyperfield $G$ and how the data of a hyperfield morphism to $G$ determines a $p$-valuation. Let's start with the first.
\begin{lemma}\label{lem:scalar-addition}
The map $\iota:\QQ_p\to G$ is injective. Identify $\QQ_p$ with its
image, which is closed under multiplication. If $h\in G^\times$,
$c\in\QQ_p^\times$, and $h\ne-c$, then
\[
 h\boxplus c\subseteq
 \begin{cases}
 \{h,c\},&h\notin\QQ_p^\times,\\
 \{h,c,h+c\},&h\in\QQ_p^\times.
 \end{cases}
\]
Here $+$ denotes addition in $\QQ_p$ as a field.
\end{lemma}

\begin{proof}
Since $\bigcap_{n\ge2}(\QQ_p^\times)^n=\{1\}$, this implies
injectivity of $\iota$.
After multiplying by $c^{-1}$, we may assume $c=1$. Let $z\in1\boxplus h$. If $z=\mathbf0$, then $h=-1$ at every level,
contradicting the hypothesis. For each $j\ge2$, the quotient
hyperaddition in $G_{j!}$ gives some $x_j\in\QQ_p$ with
\[
 [x_j]_{j!}=h_{j!},\qquad [1+x_j]_{j!}=z_{j!}.
\]

Topological compactness of $\mathbb P^1(\QQ_p)$ in the $p$-adic topology gives a subsequence converging to some $a\in\QQ_p\cup\{\infty\}$.

Fix $n\ge2$. For all sufficiently large $j$ in this subsequence,
$n\vert j!$, so
\[
 h_n=[x_j]_n,\qquad z_n=[1+x_j]_n.
\]
Since the subgroup $(\QQ_p^\times)^n$ is open,  for $a\in\QQ_p\setminus\{0,-1\}$ and large enough $j$ in the subsequence, 
$x_j/a$ and $(1+x_j)/(1+a)$ are $n$th powers. Thus $h_n=[a]_n$ and $z_n=[1+a]_n$ for every $n$, giving $h=\iota(a)$ and $z=\iota(1+a)$.
If $a=0$, the same argument applied to $1+x_j$ gives $z=1$.
If $a=\infty$, apply it to $(1+x_j)/x_j$ to obtain $z=h$.
Finally, $a=-1$ would give $h=-1$, which is impossible since $h\neq -c$. 
\end{proof}

\begin{remark}
The two properties used in this proof fail for $\lau{\CC}{t}$. Every unit in $\CC[[t]]$ is an $n$th power for every $n\ge2$, so the canonical map to the inverse limit of power quotients is not injective. Moreover, $\mathbb P^1(\lau{\CC}{t})$ is not compact in the $t$-adic topology. See also \Cref{prop:counterexample}. 
\end{remark}

\begin{lemma}\label{lem:translation-patterns}
Let $F\supseteq\QQ_p$ be a field and let $\chi:F\to G$ be a
hyperfield morphism extending $\iota$. For each $x\in F$, exactly one
of the following holds:
\begin{align}
 &\chi(x+c)=\chi(x)\quad\text{for every }c\in\QQ_p;
 \label{eq:infinite-pattern}\\
 &\chi(x+c)=a+c\quad\text{for some }a\in\QQ_p \quad \text{for every }c\in\QQ_p\setminus\{-a\}.
 \label{eq:finite-pattern}
\end{align}
In the second case, $a$ is unique.
\end{lemma}

\begin{proof}

The two patterns are clearly exclusive. Comparing the second case for two possible values of $a$ proves uniqueness of $a$ immediately, so it suffices to show existence.
For $x\in\QQ_p$, \eqref{eq:finite-pattern} holds with $a=x$.
Suppose $x\notin\QQ_p$, and let $r_c=\chi(x+c)$.
Note that every $r_c$ is not $\mathbf{0}$.

If $r_0\notin\QQ_p^\times$, Lemma~\ref{lem:scalar-addition} gives
$r_c\in\{h,c\}$ for $c\ne0$. If $r_c=r_0$ and $r_d=d$ for
nonzero $c,d$, then
\[
 c-d\in r_0\boxplus(-d)\subseteq\{r_0,-d\}.
\]
Since $r_0\notin\QQ_p^\times$, this forces $c=0$, a contradiction.
Thus either $r_c=r_0$ for every $c$, or $r_c=c$ for every $c\ne0$.
These are \eqref{eq:infinite-pattern} and
\eqref{eq:finite-pattern} with $a=0$, respectively.

Suppose $r_0\in\QQ_p^\times$. Write $r_0=\iota(d)$ with
$d\in\QQ_p^\times$. Replacing $x$ by $x/d$, we may assume
$r_0=1$. Lemma~\ref{lem:scalar-addition} gives $r_1\in\{1,2\}$.
Applying the same lemma to $x+c$ and to $(x+1)+(c-1)$ gives,
for all but finitely many $c\in\QQ_p$,
\[
 r_c\in\{1,c,c+1\}\cap\{r_1,c-1,r_1+c-1\}.
\]
If $r_1=2$, this intersection is $\{c+1\}$ for all but finitely
many $c$. If $r_1=1$, then $r_2\in\{1,2\}$, and comparison with
$x+c=(x+2)+(c-2)$ gives, again outside a finite set,
\[
 r_c\in\{1,c\}\cap\{r_2,c-2,r_2+c-2\}
 =
 \begin{cases}
 \{1\},&r_2=1,\\
 \{c\},&r_2=2.
 \end{cases}
\]
Here the finitely many excluded values are those for which \Cref{lem:scalar-addition} does not apply or two distinct expressions in the displayed sets are equal.

Actually, these identities for all but finitely many $c$'s extend to the same identities for any $c\neq -a$. Indeed, suppose $r_d=a+d$ for all but finitely many $d$.
For fixed $c\ne-a$, comparison with $x+c=(x+d)+(c-d)$ gives
\[
 r_c\in\{a+d,c-d,a+c\}
\]
for all but finitely many $d$. Any element other than $a+c$ belongs
to at most two of these sets, so $r_c=a+c$.
If instead $r_d=1$ cofinitely, then
\[
 r_c\in\{1,c-d,1+c-d\}
\]
for all but finitely many $d$, and the same argument gives $r_c=1$
for every $c$. 
\end{proof}

Let $F\supseteq\QQ_p$ be a field with $\chi:F\to G$ a
hyperfield morphism extending $\iota$, we wish to cook up a $p$-valuation on $F$. We do this by defining a place. 
Define
\[
r:F\longrightarrow\QQ_p\cup\{\infty\},\quad r(x)=
\begin{cases}
\infty,&\text{if \eqref{eq:infinite-pattern} in \Cref{lem:translation-patterns} holds},\\
a,&\text{if \eqref{eq:finite-pattern} in \Cref{lem:translation-patterns} holds with parameter }a.
\end{cases}
\]
Put $O=\{x\in F:r(x)\ne\infty\}$.

\begin{lemma}\label{lem:place}
Let $F\supseteq\QQ_p$ be a field and let $\chi:F\to G$ be a
hyperfield morphism extending $\iota$. Then $O$ above is a valuation ring
with residue field $\QQ_p$, and $r|_O$ is its residue map.
Writing $u$ for the associated valuation, one has
\[
 \chi(y)=\iota(r(y))\quad(\forall y\in\cO_u^\times).
\]
In particular, $\chi(1+\m_u)=1$.
\end{lemma}

\begin{proof}
The definition of $r$ show that it is identity on $\QQ_p$ and it satisfies
$r(x+c)=r(x)+c$ and $r(dx)=d\,r(x)$ for $c\in\QQ_p$ and
$d\in\QQ_p^\times$, with $\infty+c=d\infty=\infty$.

Let $r(x)=a$ and $r(y)=b$ with $a,b\neq \infty$. For $c\ne a+b$, the identity
\[
 x+y-c=(x-d)+(y-(c-d))
\]
and Lemma~\ref{lem:scalar-addition} give
\[
 \chi(x+y-c)\in\{a-d,b-c+d,a+b-c\}
\]
for all but finitely many $d\in\QQ_p$. Intersecting over different $d$'s
gives $\chi(x+y-c)=a+b-c$. Thus $x+y\in O$ and
$r(x+y)=a+b$. Moreover, since $r$ fixes $\QQ_p$, we have $0\in O$. $r(dx)=dr(x)$ for $d\in\mathbb{Q}_p^\times$ gives $r(-x)=-r(x)$. So $O$ is closed under additive
inverses. This shows that $O$ is an additive subgroup of $F$ and that $r|_O$ respects addition. 

For $x\ne0$ and $c\in\QQ_p^\times$, multiplicativity of $\chi$ gives
 $\chi(x^{-1}-c)=(-c)\chi(x-c^{-1})\chi(x)^{-1}$.
If $r(x)=a\in\QQ_p^\times$, then $\chi(x)=a$ and hence
\[
 \chi(x^{-1}-c)=(-c)(a-c^{-1})a^{-1}=a^{-1}-c
 \qquad(c\ne0,a^{-1}).
\]
For $c=0$ by multiplicativity of $\chi$ implies$
\chi(x^{-1})\,a =\chi(x^{-1})\chi(x) =\chi(x^{-1}x) =\chi(1)=1$. Therefore
$r(x^{-1})=a^{-1}$. 

If $r(x)=0$, we have
\[
 \chi(x^{-1}-c)=(-c)(-c^{-1})\chi(x)^{-1} =\chi(x^{-1})\qquad(c\ne0),
\]
so $r(x^{-1})=\infty$. 

Finally, if $r(x)=\infty$, it gives
\[
 \chi(x^{-1}-c)=(-c)\chi(x)\chi(x)^{-1}=-c
 \qquad(c\ne0),
\]
so $r(x^{-1})=0$. Consequently, every nonzero $x\in F$ satisfies
$x\in O$ or $x^{-1}\in O$.

It remains to show that $O$ is closed under multiplication. First let $x\in O$ with
$a=r(x)\notin\{0,-1\}$. 
Note that
\[
 x^2=\bigl(x^{-1}-(x+1)^{-1}\bigr)^{-1}-x.
\]
Since
$r(x^{-1}-(x+1)^{-1})=a^{-1}-(a+1)^{-1}\ne0$, the preceding formulas for inverses and addition gives $x^2\in O$ and $r(x^2)=a^2$.
For arbitrary $x\in O$, take $c\in\QQ_p$ with $r(x)+c\notin\{0,-1\}$. Applying the preceding case to $x+c$ and using $x^2=(x+c)^2-2cx-c^2$ gives that $x^2\in O$. Now, $O$ is closed under multiplication because of the following:
\[
 xy=\frac{(x+y)^2-x^2-y^2}{2}\in O,
 \quad r(xy)=r(x)r(y)
 \quad(\forall x,y\in O).
\]
Thus $O$ is a valuation ring with
residue field is $\QQ_p$ and $r|_O:O\to\QQ_p$ is a surjective place. For $y\in O^\times$, one has $r(y)\ne0$;
taking $c=0$ in \eqref{eq:finite-pattern} \Cref{lem:translation-patterns} gives $\chi(y)=\iota(r(y))$.
\end{proof}

\begin{lemma}\label{lem:residue-realization}
Let $F\supseteq\QQ_p$ be a field and let $\chi:F\to G$ be a
hyperfield morphism extending $\iota$, there are an
elementary extension $k'\succeq\QQ_p$ and an embedding
$j:F\hookrightarrow k'$ over $\QQ_p$ such that
\[
 \chi_n(x)=[a]_n
 \quad\Longleftrightarrow\quad
 j(x)/a\in(k'^\times)^n
 \qquad(x\in F^\times,\ a\in\QQ_p^\times,\ n\ge2).
\]
\end{lemma}

\begin{proof}
Interpret predicates on $F$ by
\[
 P_n^\chi(x)\quad\Longleftrightarrow\quad
 x=0\ \text{or}\ \chi_n(x)=1.
\]
Compose the valuation $u$ of Lemma~\ref{lem:place} with the $p$-adic valuation on its residue field. The resulting valuation $w$ has residue field $\FF_p$. Its value group contains a convex subgroup $\ZZ\delta$, where $\delta=w(p)$, and the quotient is $u(F)$. In particular, $\delta$ is the least positive value and $w$ is a $p$-valuation of $p$-rank $1$.

We verify the axiomatization of $\mathrm{pCF}_\forall$ given by B\'elair in \cite[Section~2]{Belair} for $(F,w)$. Axioms 1 and 2 in \cite{Belair} are granted since $w$ is a $p$-valuation.  

The predicates $P_n^\chi$ contain actual $n$th powers, are multiplicatively closed, and respect inversion
and the transition maps. They also satisfy $P_n^\chi(x)\Longrightarrow P_{mn}^\chi(x^m)$. For $x\ne0$, represent $\chi_{mn}(x)$ by $a\in\QQ_p^\times$. If $\chi_n(x)=1$, then $a=b^n$ for some $b\in\QQ_p$, and
$a^m=b^{mn}$. The case $x=0$ is immediate.
The usual finite coset representatives in $\QQ_p^\times$ also
represent the cosets of $P_n^\chi\setminus\{0\}$ in $F^\times$. This shows \cite[Axiom 4]{Belair} holds. 

For \cite[Axiom 3]{Belair}, put $N=2v_p(n)+1$ and
\[
 R_n=\{i\in\ZZ:1\le i<p^N,\quad
       i\equiv j^n\pmod{p^N}
       \text{ for some }j\in\ZZ\text{ with }p\nmid j\}.
\]
If $w(x)=0$, then $x$ is a unit in the $u$-valuation ring and
$\chi(x)=\iota(\res_u(x))$, with $\res_u(x)\in\ZZ_p^\times$.
Hensel's lemma therefore gives
\[
 P_n^\chi(x)\quad\Longleftrightarrow\quad
 \bigvee_{i\in R_n}w(x-i)\ge N\delta.
\]
Indeed, this is the usual $n$th-power test for $\res_u(x)$ modulo
$p^N$. If $u(x-i)>0$, then $w(x-i)>m\delta$ for every integer $m$;
otherwise $w(x-i)=v_p(\res_u(x)-i)\delta$.

Finally, for \cite[Axiom 5]{Belair}, for every $z\in F$,
\begin{equation}\label{eq:ramification}
 P_n^\chi(x)\quad\Longrightarrow\quad
 w(z^nx)\le0\ \text{or}\ w(z^nx)\ge n\delta.
\end{equation}
If this failed, convexity of $\ZZ\delta$ would give
$w(z^nx)=j\delta$ for some integer $0<j<n$. Then $z^nx$ would be
a $u$-unit whose residue has $p$-adic valuation $j$, so its
$n$th-power class would be nontrivial. This contradicts
$\chi_n(z^nx)=\chi_n(z)^n\chi_n(x)=1$.

Thus $(F,w,(P_n^\chi)_{n\ge2})\models T_1$.
By \cite[Theorem~2.4]{Belair}, it embeds in a $p$-adically closed
field $k'$ of $p$-rank $1$, preserving and reflecting all the
predicates. The restriction to $\QQ_p$ is elementary by Macintyre
quantifier elimination. Applying predicate preservation to $x/a$
proves the assertion.
\end{proof}
Now we can prove \Cref{thm:formally-p-adic}.
\begin{proof}[Proof of Theorem~\ref{thm:formally-p-adic}]

As explained earlier, \textnormal{(2)}$\implies$\textnormal{(1)} is immediate: Since models of $\pCF_\forall$ embed in a model $L\models\pCF$. The
restriction of this embedding to $K$ is elementary by quantifier
elimination, so we may identify $F$ with its image and assume
$K\preceq L$. For each $n$, the natural map $G_n(K)\longrightarrow G_n(L)$
is an isomorphism. These isomorphisms commute with the
transition maps. Composing the canonical map $F\to G(L)$ with
its inverse gives the required morphism $F\to G(K)$.

Suppose first that \textnormal{(1)} holds, and equip $F$ with the
predicates $P_n^\varphi$ in the statement. For each $n$, choose a
finite set of representatives in $\mathbb{Q}$ of
$\QQ_p^\times/(\QQ_p^\times)^n$. Such representatives exist
because the subgroups $(\mathbb{Q}_p^\times)^n$ are open and $\QQ$ is dense in
$\QQ_p$. Since $K\equiv\QQ_p$ and $\QQ_p^\times/(\QQ_p^\times)^n$ is finite, 
$\QQ_p^\times/(\QQ_p^\times)^n\cong K^\times/(K^\times)^n$ as hyperfields, and the transition maps are compatiable. These identifications identify $G(K)$ with $G(\mathbb{Q}_p)$.

Consider $(F,K,(P_n^\varphi)_{n\ge2})$ as a $\mathcal{L}_{\mathrm{Mac}}\cup\{P\}$-structure with a
predicate $P$ for the subfield $K$, and take a sufficiently saturated elementary extension $(F^*,K^*,(P_n^*)_{n\ge2})$.
The field $K^*$ contains an copy of $\QQ_p$ as an elementary substructure.
Note that on $F$, for each $n$, the nonzero fibers of $\varphi_n$ are the sets
defined by $x\ne0\wedge P_n^\varphi(x/a)$, where $a$ ranges over the chosen representatives of $\QQ_p^\times/(\QQ_p^\times)^n$ in $\mathbb{Q}$. Hence the hyperfield structure on $G_n(F)$ and the transition maps are definable in $\mathcal{L}_\mathrm{Mac}$. Thus their definition extend to $F^*$ and gives a hyperfield morphism $\chi:F^*\to G(\mathbb{Q}_p)$. On $K^*$ the predicates $P_n^*$ are the actual power predicates, so $\chi$ agrees to the canonical map $\iota$ on the chosen copy of $\QQ_p$.

Lemma~\ref{lem:residue-realization} embeds $F^*$ in a model
$L\models\pCF$, preserving and reflecting every $P_n^*$.
Restricting to $F$ gives  the desired embedding $j:F\to K'\succeq K$. In other words,
$(F,(P_n^\varphi)_{n\ge2})\models\pCF_\forall$.
\end{proof}

\bibliographystyle{alpha}
\bibliography{biblio}

\begin{thebibliography}{CKRKV24}

\bibitem[ACGZ22]{acgz}
Matthias Aschenbrenner, Artem Chernikov, Allen Gehret, and Martin Ziegler.
\newblock Distality in valued fields and related structures.
\newblock {\em Transactions of the American Mathematical Society},
  375(7):4641--4710, 2022.
\newblock arXiv:2008.09889.

\bibitem[AG21]{AndujarTypes}
Pablo And{\'u}jar~Guerrero.
\newblock Types, transversals and definable compactness in o-minimal
  structures.
\newblock arXiv:2111.03802, 2021.

\bibitem[AGJ24]{AndujarJohnsonTypes}
Pablo And{\'u}jar~Guerrero and Will Johnson.
\newblock Around definable types in {$p$}-adically closed fields.
\newblock {\em Ann. Pure Appl. Logic}, 175(10):103484, 2024.

\bibitem[B{\'e}l88]{Belair}
Luc B{\'e}lair.
\newblock Substructures and uniform elimination for {$p$}-adic fields.
\newblock {\em Ann. Pure Appl. Logic}, 39(1):1--17, 1988.

\bibitem[Bel02]{poly}
Oleg Belegradek.
\newblock Poly-regular ordered abelian groups.
\newblock {\em Contemporary Mathematics}, 302:101--112, 2002.

\bibitem[BT26]{baldi}
Lorenzo Baldi and Máté~L. Telek.
\newblock A note on o-minimal tropicalizations.
\newblock arXiv:2606.14523, 2026.
\newblock arXiv:2606.14523 [math.AG].

\bibitem[CKHY26]{beautiful}
Pablo Cubides~Kovacsics, Martin Hils, and Jinhe Ye.
\newblock Beautiful pairs.
\newblock {\em Journal f\"ur die reine und angewandte Mathematik (Crelles
  Journal)}, 2026(833):235--283, 2026.

\bibitem[CKKR20]{cubi-kuh-rze}
Pablo Cubides~Kovacsics, Franz-Viktor Kuhlmann, and Anna Rzepka.
\newblock On valuation independence and defectless extensions of valued fields.
\newblock {\em J. Algebra}, 555:69--95, 2020.

\bibitem[CKRKV24]{orthogonal}
Pablo Cubides~Kovacsics, Silvain Rideau-Kikuchi, and Mariana Vicaría.
\newblock On orthogonal types to the value group and residual domination, 2024.

\bibitem[CS19]{inpminimal}
Artem Chernikov and Pierre Simon.
\newblock Henselian valued fields and inp-minimality.
\newblock {\em The Journal of Symbolic Logic}, 84(4):pp. 1510--1526, 2019.

\bibitem[DHLY24]{tropical-function}
Antoine Ducros, Ehud Hrushovski, François Loeser, and Jinhe Ye.
\newblock Tropical functions on a skeleton.
\newblock {\em J. \'Ec. polytech. Math.}, 11:613--654, 2024.

\bibitem[DWY25]{field-top-2}
Philip Dittmann, Erik Walsberg, and Jinhe Ye.
\newblock When is the \'etale open topology a field topology?
\newblock {\em Israel J. Math.}, 269(1):67--102, 2025.

\bibitem[Fle11]{flenner}
Joseph Flenner.
\newblock Relative decidability and definability in henselian valued fields.
\newblock {\em The Journal of Symbolic Logic}, 76(4):1240--1260, 2011.
\newblock arXiv:0910.2682.

\bibitem[HHM08]{HHM}
Deirdre Haskell, Ehud Hrushovski, and Hugh~Dugald Macpherson.
\newblock {\em Stable domination and independence in algebraically closed
  valued fields}, volume~30 of {\em Lect. Notes Log.}
\newblock Cambridge Univ. Press, {A}ssoc. {S}ymbol. {L}ogic edition, 2008.

\bibitem[HL16]{HruLoe}
Ehud Hrushovski and Fran{\c{c}}ois Loeser.
\newblock {\em Non-{Archimedean} tame topology and stably dominated types},
  volume 192 of {\em Ann. Math. Stud.}
\newblock Princeton, NJ: Princeton University Press, 2016.

\bibitem[JSW17]{dpminimal}
Franziska Jahnke, Pierre Simon, and Erik Walsberg.
\newblock Dp-minimal valued fields.
\newblock {\em The Journal of Symbolic Logic}, 82(1):151--165, 2017.

\bibitem[JSY22]{jell}
Philipp Jell, Claus Scheiderer, and Josephine Yu.
\newblock Real tropicalization and analytification of semialgebraic sets.
\newblock {\em International Mathematics Research Notices}, 2022(2):928--958,
  2022.

\bibitem[JTWY24]{firstpaper}
Will Johnson, Chieu-Minh Tran, Erik Walsberg, and Jinhe Ye.
\newblock The \'etale-open topology and the stable fields conjecture.
\newblock {\em J. Eur. Math. Soc. (JEMS)}, 26(10):4033--4070, 2024.

\bibitem[Jun21]{Jun}
Jaiung Jun.
\newblock Geometry of hyperfields.
\newblock {\em Journal of Algebra}, 569:220--257, 2021.

\bibitem[JWY23]{field-top-1}
Will Johnson, Erik Walsberg, and Jinhe Ye.
\newblock The \'{e}tale open topology over the fraction field of a {H}enselian
  local domain.
\newblock {\em Math. Nachr.}, 296(5):1928--1937, 2023.

\bibitem[Kam07]{kamensky}
Moshe Kamensky.
\newblock Ind- and pro- definable sets.
\newblock {\em Annals of Pure and Applied Logic}, 147(3):180--186, 2007.

\bibitem[Kra83]{Krasner1983}
Marc Krasner.
\newblock A class of hyperrings and hyperfields.
\newblock {\em International Journal of Mathematics and Mathematical Sciences},
  6(2):307--311, 1983.

\bibitem[Kuh04]{kuhlmann2004}
Franz-Viktor Kuhlmann.
\newblock Places of algebraic function fields in arbitrary characteristic.
\newblock {\em Adv. Math.}, 188:399--424, 2004.

\bibitem[Pay09]{Payne}
Sam Payne.
\newblock Analytification is the limit of all tropicalizations.
\newblock {\em Mathematical Research Letters}, 16(3):543--556, 2009.

\bibitem[Pop10]{henselianlarge}
Florian Pop.
\newblock Henselian implies large.
\newblock {\em Ann. of Math. (2)}, 172(3):2183--2195, 2010.

\bibitem[Pop14]{pop-large}
Florian Pop.
\newblock Little survey on large fields --- old \& new.
\newblock In {\em Valuation theory in interaction}, EMS Ser. Congr. Rep., pages
  432--463. Eur. Math. Soc., Z\"urich, 2014.

\bibitem[Vic26]{vicariaresdom}
Mariana Vicar{\'\i}a.
\newblock Residue field domination.
\newblock {\em Israel Journal of Mathematics}, 272(1):371--405, 2026.

\end{thebibliography}

\end{document}